\documentclass[11pt]{article}
\usepackage{hyperref}
\usepackage{amsbsy}
\usepackage[T1]{fontenc}
\usepackage{amsfonts,eufrak}
\usepackage[english]{babel}
\usepackage{a4wide,times}

\usepackage[latin1]{inputenc}
\usepackage{amssymb}
\usepackage{epsfig}
\usepackage{amsbsy}
\usepackage{verbatim}
\usepackage{color}
\usepackage{mathrsfs}
\usepackage{graphicx}
\usepackage{epstopdf}
\usepackage{makeidx}
\usepackage{amsmath,amsbsy}
\usepackage{mathabx}
\usepackage{amsthm}
\usepackage{xargs}
\usepackage{ulem}
\usepackage[prependcaption]{todonotes}
\newcommandx{\fab}[2][1=]{\todo[inline, author={Fab}, linecolor=green,backgroundcolor=red!25,bordercolor=red,#1]{#2}}
\newcommandx{\fabnote}[2][1=]{\todo[author={Fab}, linecolor=red,backgroundcolor=red!25,bordercolor=red,#1]{#2}}

\theoremstyle{plain}
\newtheorem{thm}{Theorem}[section]
\newtheorem{cor}[thm]{Corollary}
\newtheorem{lem}[thm]{Lemma}
\newtheorem{prop}[thm]{Proposition}

\theoremstyle{definition}

\usepackage{dsfont}

\theoremstyle{remark}
\newtheorem{rem}{\bf Remark}[section]
\theoremstyle{remark}

\newtheorem{com*}{\bf Comment}

\usepackage{fancyhdr}

\makeatletter
\def \newequation#1#2{
   \@definecounter{#1}
   \@namedef{the#1}{\hbox{#2}}
   \@namedef{#1}{$$\refstepcounter{#1}}
   \@namedef{end#1}{
      \eqno \csname the#1\endcsname $$\global\@ignoretrue
      }
}
\makeatother
\newequation{E1}{($E_{b,\sigma,W}$)}
   
\makeatletter
\def \newequation#1#2{
   \@definecounter{#1}
   \@namedef{the#1}{\hbox{#2}}
   \@namedef{#1}{$$\refstepcounter{#1}}
   \@namedef{end#1}{
      \eqno \csname the#1\endcsname $$\global\@ignoretrue
      }
   }
\makeatother
\newequation{hyp3}{($\mathcal{H}_{b,\sigma}$)}

\makeatletter
\def \newequation#1#2{
   \@definecounter{#1}
   \@namedef{the#1}{\hbox{#2}}
   \@namedef{#1}{$$\refstepcounter{#1}}
   \@namedef{end#1}{
      \eqno \csname the#1\endcsname $$\global\@ignoretrue
      }
   }
\makeatother
\newequation{shleg}{(ShLeg)}

\makeatletter
\def \newequation#1#2{
   \@definecounter{#1}
   \@namedef{the#1}{\hbox{#2}}
   \@namedef{#1}{$$\refstepcounter{#1}}
   \@namedef{end#1}{
      \eqno \csname the#1\endcsname $$\global\@ignoretrue
      }
   }
\makeatother
\newequation{kl}{(KL)}

\makeatletter
\def \newequation#1#2{
   \@definecounter{#1}
   \@namedef{the#1}{\hbox{#2}}
   \@namedef{#1}{$$\refstepcounter{#1}}
   \@namedef{end#1}{
      \eqno \csname the#1\endcsname $$\global\@ignoretrue
      }
   }
\makeatother
\newequation{haar}{(Haar)}

\def\R{{\mathbb{R}}}
\def\ER{{\mathbb{R}}}

\def\N{{\mathbb{N}}}
\def\E{{\mathbb{E}}}
\def\ES{{\mathbb{E}}}
\def\h{h}

\def\P{{\mathbb{P}}}
\def\PE{{\mathbb{P}}}

\def\Q{{\mathbb{Q}}}

\def\FF{F}
\def\a{\alpha}

\def\g{\gamma}
\def\d{\delta}
\def\G{\Gamma}
\def\Gam{\Gamma}
\def\s{\sigma}
\def\ve{\varepsilon}
\def\HTVO{\mathbf{(H_{TV})}}
\def\HWO{\mathbf{(H_{{\cal W}_1})}}

\def\un{\underline}
\def\bX{\bar{X}}

\def\lt{\left}
\def\rt{\right}

\def\ELLIP{{{\bf (\mathcal{E}}{\ell}{\bf )}_{\underline{\s}^2_0}}}

\def\Der{D}
\def\Tun{\mathfrak{t}}
\def\Tdeux{\tau}
\def\gsup{\|\gamma\|_{{\rm sup}}}
\def\tens{\mathfrak{T}}
\def\alplunmoinszero{(|\alpha|+1)\setminus0}
\def\alpldeuxmoinszero{(|\alpha|+2)\setminus0}
\def\kmoinszero{k\setminus0}
\author{Gilles Pag\`es~\thanks{Sorbonne Universit\'e,  Laboratoire de Probabilit\'es, Statistique et Mod\'elisation, UMR~8001, case 158, 4, pl. Jussieu, F-75252 Paris Cedex 5, France. E-mail: \texttt{gilles.pages@sorbonne-universite.fr}}  $\;$and 
Fabien Panloup~\thanks{LAREMA, Facult\'e des Sciences, 2 Boulevard Lavoisier, Universit\'e d'Angers, 49045 Angers, France. E-mail: \texttt{fabien.panloup@univ-angers.fr}}
}
\title{{Unadjusted Langevin algorithm with multiplicative noise: Total variation and Wasserstein bounds}}

\begin{document}
\maketitle

\begin{abstract}
In this paper, we focus on non-asymptotic bounds related to the Euler scheme of an ergodic diffusion with a possibly  multiplicative diffusion term (non-constant diffusion coefficient). More precisely, the objective of this paper is to control the distance of the standard Euler scheme with decreasing step ({usually called Unadjusted Langevin Algorithm in the Monte Carlo literature}) to the invariant distribution  of such an ergodic diffusion. In an appropriate Lyapunov setting and  under  {uniform} ellipticity  assumptions on the diffusion coefficient, we establish (or improve) such bounds for Total Variation and $L^1$-Wasserstein distances in both multiplicative and  additive and  frameworks. These bounds rely on weak error expansions using  {Stochastic Analysis} adapted to decreasing step  setting.
\end{abstract}

{{\footnotesize \textit{Mathematics Subject Classification:} Primary 65C05-37M25-60F05-62L10 Secondary 65C40-93E3}


%
\section{Introduction}
Let $(X_t)_{t\in [0,T]}$ be the unique strong solution to the stochastic differential equation ($SDE$)
\begin{equation}\label{eds:intro}
dX_t = b(X_t)dt +\sigma(X_t)dW_t
\end{equation}
starting at $X_0$ where $W$ is a  standard $\R^q$-valued standard Brownian motion, independent of $X_0$, both defined on a probability space $(\Omega, {\cal A}, \PE)$, where $b:\R^d\to \R^d$ and $\sigma:\R^d \to \mathbb{M}(d,q, \ER)$ {($d\times q$-matrices with real entries)} 
are  Lipschitz continuous functions. 
The process $(X_t)_{t\ge 0}$ is a {homogeneous} Markov process, denoted $X^x=(X^x_t)_{t\ge 0}$ if $X_0=x$,  {with transition semi-group $P_t(x,dy)= \P(X^x_t\!\in dy)$}. We denote by $\PE_{\mu}$ its distribution starting from $X_0\sim \mu$ {(and $\P_x$ when $\mu = \delta_x$)}. 
Let ${\cal L}={\cal L}_{_X}$ denote  its infinitesimal generator, defined on twice differentiable functions $g:\R^d\to \R$ by
\[
{\cal L}g = (b|\nabla g) +\frac 12 {\rm Tr}\big(\sigma^*D^2g\,\sigma\big),
\]
where $(\,.|.\,)$ \textcolor{black}{denotes the canonical inner product on $\ER^d$,}
$D^2 g$ denotes  the Hessian matrix of $g$ and ${\rm Tr}$ denotes the Trace  operator.

\smallskip
Let $(\g_n)_{n\ge 1}$ be a {non-increasing} sequence of positive {\em steps}. We consider the Euler scheme of the $SDE$ with  step $\g_n>0$ starting from $\bar X_0= X_0$ defined by
\begin{equation}\label{eq:EulerDec}
\bar X_{\G_{n+1}} = \bar X_{\G_n} +\g_{n+1}b(\bar X_{\G_n} ) + \s (\bar X_{\G_n} )(W_{\G_{n+1}}-W_{\G_n}),\quad n\ge 0.
\end{equation}
where 
{
$$
\G_0=0 \quad \mbox{ and }\quad \G_n = \g_1+\cdots+\g_n.
$$
with $(\g_n)_{n\ge1}$ a sequence of varying time steps.} We define the genuine (continuous time) Euler scheme by interpolation as follows: let $t\!\in [\G_k, \G_{k+1})$. 
\begin{equation}\label{eq:genuine}
\bar X_t = \bar X_{\G_k} + (t-\G_k)b(\bar X_{\G_k}) + \s (\bar X_{\G_k} ) (W_t-W_{\G_k}).
\end{equation}

If we set $\underline t= \G_k$ on the time interval $[\G_k, \G_{k+1})$, the genuine Euler scheme appears as an It\^o process solution to the pseudo-$SDE$ with frozen coefficients
\begin{equation}\label{eq:EulerDecCont}
d\bar X_t =   b(\bar X_{\underline t})dt +\sigma(\bar X_{\underline t})dW_t.
\end{equation}
It will be convenient in what follows to introduce
\begin{equation}\label{eq:N(t)t}
N(t)= \min\big\{k\ge 0: \G_{k+1} >t  \big\}= \max\big\{k\ge 0: \G_k \le t  \big\}.
\end{equation}


The Euler scheme is a discrete time non-homogeneous Markov process with transitions $$\bar P_{\G_n,\G_{n+1}}(x,dy)= \bar P_{\g_{n+1}}(x,dy)$$ where the {transition probability}  $\bar P_{\g}(x,dy)$ reads on Borel test functions
\begin{equation}\label{eq:pbargam}
\bar P_{\g}g(x)= \E\, g\big(x+\g g(x) +\sqrt{\g} \s (x)Z\big),\qquad Z\sim {\cal N}(0;I_d).
\end{equation}
{ We assume that the time step sequence $(\g_n)_{n\ge1}$ satisfies Assumption $(\G)$ defined by:}
\begin{equation}\label{eq:Gamma}
(\G):\qquad
(\g_n)_{n\ge 1} \mbox{ non-increasing}, \quad \lim_n \g_n =  0 \qquad \mbox{ and }\qquad \sum_{n\ge 1}\g_n =+\infty.
\end{equation}
Then $\g_1=\sup_{n\ge 1}Ê\g_n$ and we will denote indifferently this quantity by $\|\boldsymbol{\gamma}\|$ or $\g_1$ depending on the context.

It is well-known that for a twice continuously differentiable  function $V: \R^d\to \R_+ $ such that $e^{-V}\!\in L^1_{\R_+}(\R^d, \lambda_d)$ ($\lambda_d$ Lebesgue measure on $\R^d$), then for every $\sigma\!\in (0, 1]$
\[
\nu_{\s}(dx) =C_{\s} e^{-\frac{V(x)}{\s^2}}\lambda_d(dx) \quad \mbox{ with } \quad C_{\s} = \Big( \int_{\R^d}e^{-\frac{V(x)}{\s^2}}\lambda_d(dx) \Big)^{-1}
\]
is the unique invariant distribution of the  Langevin (reversible) Brownian {\em SDE}
\begin{equation} \label{eq:LangevinEDS}
dX_t  =  -\nabla V(X_t) dt +\sqrt{2}\, \s dW_t
\end{equation}  
where $(W_t)_{t\ge 0}$ is $d$-dimensional standard Brownian motion.

A first application of this property is to devise an approximate simulation method of $\nu= \nu_1= C_1^{-1} e^{-V}\cdot \lambda_d$
by introducing the above Euler scheme with decreasing step~\eqref{eq:EulerDec} with $b= -\nabla V $ and $\sigma(x) = \sqrt{2}$. Coupled with a Metropolis-Hasting speeding method, this simulation procedure is known as the {\em Metropolis Adjusted Langevin algorithm} whereas in absence of such an additional procedure it is known as  the {\em Unadjusted Langevin Algorithm} ({\em ULA}) extensively investigated in  the literature since  the 1990's (see $e.g.$ ~\cite{Pelletierthese},~\cite{MarPel1996}) and more recently in a series of papers, still in the additive setting,  motivated  by applications in machine learning (in particular in Bayesian or PAC-Bayesian statistics). Among others, we refer to ~\cite{MouDur2017,MouDur2019, dalalyan,MouFlamWainBart} and to the references therein. 

\noindent A second application is to directly consider, $\s$ being  a   fixed real number (or possibly a matrix of  $\mathbb{M}(d,d, \ER)$), the Euler scheme 
\begin{equation}~\label{eq:LangevinEDS2}
\bar X^{\s}_{\G_{n+1}} = \bar X_{\G_n}^{\s}  -\g_{n+1}\nabla V(\bar X_{\G_n}^{\s}) + {\sqrt{2}}\,\s \sqrt{\g_{n+1}} Z_{n+1}
\end{equation}  
where $(Z_k)_{k\ge 1}$ is an  ${\cal N}(0,{ I_d})$-distributed i.i.d. sequence{. It appears  as a  perturbation by a Gaussian white noise of the {\em gradient descent}}
\[
x_{n+1}= x_n-\g_{n+1} \nabla V(x_n)
\]
aiming at minimizing the {\em potential} $V$. Then,  {using the notation $[Y]$ to denote the distribution of a random vector $Y$}, 
\[ 
[\bar X^{\s}_{\Gamma_n}]\stackrel{TV}{\longrightarrow} \nu_{\sigma}  \quad \mbox{ and }\quad \nu_{\s}\stackrel{weakly}{\longrightarrow} \d_{x^*} \; \mbox{ as }\sigma \to 0 
\]
\noindent if ${\rm argmin}_{\R^d}V= \{x^*\}$ (or $\nu_{\s}$ is asymptotically supported by ${\rm argmin}_{\R^d}V$ when simply finite). So simulating~\eqref{eq:LangevinEDS2} on  the long run provides sharper and sharper information on  the localization of ${\rm argmin}_{\R^d}V$. In fact making $\s=\s_n$ slowly vary in a decreasing way to $0$ at rate $(\log n)^{-1/2}$ makes up a {\em simulated annealing version} of the above perturbed stochastic gradient procedure. This stochastic optimization procedure has  been investigated in-depth in~\cite{GelfMit} with, as a main result,  the convergence in probability of $\bar X^{\s_n}_{\G_n}$ toward the (assumed) unique minimum $x^*$ of $V$ under various assumptions on  the step $\g_n$ and the invertibility  of the Hessian of $V$ at $x^*$.

For much more general multidimensional diffusions, say Brownian driven here for convenience, of the form~\eqref{eds:intro} with infinitesimal generator ${\cal L}$ satisfying an appropriate mean-reverting drift (typically ${\cal L}V\le \beta -\alpha V^a$, $a\!\in (0,1]$ for some Lyapunov function $V$), it is a natural problem of numerical probability to have numerical access to its invariant distribution $\nu$ (when unique). Taking full advantage of ergodicity, this can be achieved by introducing  the weighted empirical measure 
\begin{equation}\label{eq:nubar}
\bar \nu_n(\omega, d\xi) = \frac{1}{\G_n} \sum_{k=1}^n \g_k \d_{\bar X_{\G_{k-1}}(\omega)}(d\xi), \quad n\ge 1,
\end{equation}
where $(\bar X_{\G_k})_{k\ge0}$ is given  by~\eqref{eq:EulerDec} (and the Brownian increments are simulated by  a $\R^q$-valued white noise  $(Z_k)_{k\ge 1}$ with $ W_{\G_{k+1}}-W_{\G_k}= \sqrt{\g_{k+1}} Z_k$, $k\ge 1$. $A.s.$ weak convergence {of 
$\bar \nu_n(d\xi)$ to $\nu$, its convergence rate as well as}  deviation inequalities depending on the rate of decay of the time step $\g_n$ have been extensively investigated in a series of papers in various settings, including  the case of jump diffusion driven by L\'evy processes (see~\cite{LambPag1},~\cite{LambPag2},~\cite{Panloup1},~\cite{Panloup2},~\cite{lemairethese},~\cite{LemaireSPA},~\cite{HonMenPag}, etc).  One specificity of interest of this method {based on the simulation of the above weighted empirical measures $\bar \nu_n(\omega, d\xi)$ (see~\eqref{eq:nubar})} for applications  is  that no ellipticity is required to establish most of the main results.
 This turns out to be crucial for Hamiltonian systems or  more generally for mean-reverting {\em SDE}s with more or less degenerate diffusion coefficients.

However it is a quite natural question to tackle the total variation ({\em TV}) and $L^1$-Wasserstein (rates of) convergence of $[\bar X_{\G_n}]$ toward the (necessarily) unique invariant distribution $\nu$ when $\s$ is not constant but uniformly elliptic. In particular, one aim of this paper is to check whether  or not  the {\em VT} and $L^1$-Wasserstein (or Monge-Kantorovich) rates of convergence remain unchanged in such a more general setting {(in terms of $(\gamma_n)_{n\ge1}$)}. Moreover, considering such diffusions with non constant $\s$ will deeply impact the methods of proof. {When $\s$ is constant, the continuous-time Euler scheme  $(\bar X^x_{t})_{t\ge0}$ and the diffusion $(X^x_t)_{t\ge0}$  have the same diffusion component $\sigma W$. Girsanov's theorem then implies that their distributions are equivalent and provides an explicit expression of the density of the distribution of $\bar X^x_{\g}$ with respect to the one of $X^x_\g$. This is the key to establish the estimates of  $d_{TV}([X^x_{\g_n}], [\bar X^x_{\g_n}])$ through Pinsker's inequality ({see~\cite{MouDur2017} or Proposition~\ref{prop:TVsigcst} and Theorem~\ref{thm:additive} of our paper}). In the multiplicative case, such an approach no longer works  and will be replaced here by stochastic analysis arguments (see below for details). } 

Such investigations also have applied motivations since in the blossoming literature produced by the data science community to analyze and improve the performances of stochastic gradient procedures, non-constant matrix valued diffusion coefficients $\s(x)$ are introduced in such a way (see~\cite{MaChFo2015} and the references therein with in view Hamiltonian Monte Carlo, see~\cite{LiChCaCa2015} among others) that the invariant distribution is unchanged but the exploration of the state space becomes non-isotropic, depending on the position of the algorithm or the value of  potential function to be minimized with the hope to speed up its preliminary convergence phase. Note that a script of the~\cite{LiChCaCa2015} version of  Unadjusted Langevin Algorithm is made available 
in the API   TensorFlowProbability}~(\footnote{{see {\href{www.tensorflow.org/probability/api\_docs/python/tfp/optimizer/StochasticGradientLangevinDynamics}{{www.tensorflow.org/probability/api\_docs/python/tfp/optimizer/StochasticGradientLangevinDynamics}}}}}). 

As mentioned above we mainly focus on {$TV$ or $L^1$-Wasserstein bounds} in the so-called {\em multiplicative} setting i.e. when the diffusion coefficient is state dependent, which is new in this field to our best knowledge. However we also show how to refine our methods of proof (see below) in order to derive improved rates in the {\em additive } setting (when $\s$ is constant). {These results  improve those obtained  $e.g.$ in~\cite{MouDur2017} or in~\cite{dalalyan} in terms of  $(\g_n)_{n\ge1}$ and seem quite consistent with more recent works (by very different methods) like~\cite{MouFlamWainBart} or {~\cite{MouDur2019}}. {In fact,  we slightly improve the results of these papers by killing some logarithmic terms with the help of Malliavin calculus (see Remark~\ref{rem:comparisonwithothers}  for details). However, compared with these papers, we do not tackle the problem related to the dependence of the bounds with respect to the dimension $d$, which would lead to very heavy technicalities, especially in the multiplicative setting {which is the main goal of this paper.}}\\

\noindent {Although this problem  seems not  to have been already tackled  in the multiplicative case}, {we can yet connect our work with several other papers where non-asymptotic bounds between the Euler scheme and the invariant distribution have been  established: in the recent paper~\cite{crisan-dobson-ottobre}, the authors provide uniform in time bounds for the weak error (which in turn may be used to derive some bounds for the error with respect of the invariant distribution). Nevertheless, this paper only considers smooth functions which is clearly not adapted to $TV$ or $1$-Wasserstein bounds. We can also refer to~\cite{delmo-singh} where, with the help of a {new} \textit{Backward It\^o-Ventzell formula}, the authors  interpolate the diffusion and its continuous-time Euler discretization  to derive   nice $L^p$-bounds under some pathwise contraction assumptions (close to Assumption $\mathbf{(C_\alpha)}$ below). These $L^p$-bounds lead in turn to $1$-Wasserstein bounds {but it is not clear that they may produce  $TV$-bounds in an optimal way}. The interesting fact is that our so-called \textit{domino decomposition} described below can be seen  as a discrete weak version of the pathwise interpolation proposed in~\cite{delmo-singh}. {In particular, o}ur approach is different from that in~\cite{delmo-singh} since we {rely on the contraction of the semi-group of the diffusion}  instead of the pathwise assumptions required {everywhere there (whereas contraction of the semi-group may hold in settings where pathwise contraction holds only outside a compact set, see $e.g.$ Corollary \ref{cor:weaklyconvex}).}}


{Now,} let us be more specific about our results and methods. We start from some assumptions on the diffusion~\eqref{eds:intro} itself: {we mainly assume a classical Lyapunov mean-reverting assumption (denoted by $\mathbf{(S)}$),  an  exponential contraction property (in $1$-Wasserstein distance) of the distributions  $[X^x_t]$ and $[X^y_t]$\footnote{See Assumption $\mathbf{(H_\mathfrak{d})}$ and Remark~\ref{rk:2.3} for details.}  and  uniform ellipticity and boundedness assumptions on the diffusion coefficient $\s$ (denoted by $\ELLIP$, see~Section~\ref{subsec:Asump} for details).} 

 {{In the  {\em multiplicative setting}, under these general assumptions (including uniform ellipticity)}, our main result  (see Theorem~\ref{thm:multiplicative}) establishes that the Total Variation (TV) distance between the distribution of $X_{\G_n}$   and the invariant distribution $\nu$ (denoted by $\| [ \bar{X}^x_{\G_n}]-\nu \big\|_{TV}$, see below for notations) {converges to $0$ at  rate}}

\smallskip
\centerline{{$O(\g_n^{1-\varepsilon})$, for every $\varepsilon \!\in (0, 1)$,  for the $TV$-distance (if $b$ and $\s$ are $C^6$), }}

\smallskip
\noindent  {whereas its $1$-Wasserstein counterpart (denoted by ${\cal W}_1([\bar{X}_{\G_n}^x],\nu)$) converges to $0$ at rate}

\smallskip
\centerline{{$O\big(\g_n \log (1/\g_n)\big)$ for the ${\cal W}_1$-distance  (if $b$ and $\s$ are $C^4$).}}

\smallskip
 {In the {\em additive case} (see Theorem~\ref{thm:additive} e.g. if $b$ is $C^3$), we prove that the distance between the distribution of $X_{\G_n}$   and $\nu$ is :}

\noindent \centerline{{$O(\g_n)$ for both  the $TV$-distance and
the ${\cal W}_1$-distance.}}

\medskip

{As mentioned before, these results are established under general contraction assumptions made on the dynamics of the underlying diffusion. Thus, in order to be more concrete, we recall and provide in Section~\ref{subsec:Applications} practical criterions which imply exponential contraction (and thus exponential convergence rate). Typically, such an assumption holds true  if the drift coefficient is \textit{strongly contracting} outside a compact set (see Corollary~\ref{cor:weaklyconvex}). }

Our method of proof mostly relies on  Numerical Probability and Stochastic Analysis  techniques developed for diffusion processes since the 1980's, adapted to both decreasing step and long time behaviour. Namely, we carry out an in-depth analysis of the weak error of the one step Euler scheme (bounded) Borel and smooth functions, with a a special case in the  latter case to the dependence of the resulting rate with respect to the regularity of the function. Then we rely on the regularizing properties of the semi-group of the underlying diffusion through an extensive use of Bismut-Elworthy-Li  ({\em BEL}) identities and their resulting upper-bounds (see~\cite{bismut, elworthy-li}).  To deal with (non-smooth) bounded Borel functions we call upon the Malliavin Calculus machinery adapted to the decreasing step setting relying, among others, on recent papers by Bally, Caramellino and Poly (see~\cite{bally_caramellino, bally_caramellino_poly}) which make these methods more accessible. 

Our global strategy of proof (initiated by~\cite{Talay-Tubaro, Bally-Talay}) relies either on a partial (for {\em TV}-distance in the multiplicative case) or a full domino decomposition  of the error to be controlled, formally reading in our long run behaviour as follows (here for the full one)
\begin{align*}
|\E\, f(\bar X^x_{\G_n})-\E f(X^x_{\G_n})\big]& = \big |\bar P_{\g_1}\circ\cdots\bar P_{\g_{n}}f(x) - P_{\G_n}f(x)   \big| \\
&\le 
\sum_{k=1}^n \Big|\bar P_{\g_1}\circ\cdots\bar P_{\g_{k-1}}\big(\bar P_{\g_{k}}-P_{\g_{k}}  \big) P_{\G_n-\G_{k}}f(x)   \Big|.
\end{align*}
\indent Depending on the nature of the distance and $\s$ we will subdivide the above sum in two or three partial sums and analyze them using the various tools briefly described above. 
The paper is organized as follows. Section~\ref{sec:mains} is devoted to the assumptions, the main results and the applications.  In Section~\ref{sec:Toolbox} we first provide some background on our main tools, especially on Stochastic Analysis ({\em BEL}, weak error by Malliavin calculus, having in mind that most background and proof are postponed in Appendices~\ref{app:A} and~\ref{app:C} and, in a second part of the section, we analyze in-depth the weak error of the one-step Euler scheme with in mind the strong specificity of our long run problem. In Section~\ref{sec:proofmaintheorem}, we provide proofs for our main convergence results.

\medskip
\noindent {\sc Notations.} 
\noindent --  The canonical Euclidean norm of a vector $x=(x_1,\ldots,x_d)\!\in \R^d$ is denoted by $|x|= (x_1^2+\cdots+x_d^2)^{1/2}$.

\noindent --   $\N= \{0,1,\ldots\}$ and $\N^*= \{1,2,3,\ldots\}$. 

\noindent --   ${\|A\|_{_F}}= [{\rm Tr}\,(AA^*)]^{1/2}$ denotes the Fr\"obenius (or Hilbert-Schmidt) norm of a matrix $A\!\in \mathbb{M}(d,q, \R)$ where $A^*$ stands for the transpose of $A^*$ and ${\rm Tr}$ denotes the  trace operator of a square matrix.

\noindent --   ${\cal S}(d,\R)$ denotes   the  set of symmetric  $d\times d$ square matrices and ${\cal S}^+(d,\R)$ the subset of non-negative symmetric matrices.

\noindent --   $\| a \|=Ê\sup_{n\ge 1}  |a_n|$ denotes the sup-norm of a sequence $(a_n)_{n\ge 1}$.

\noindent --   For $f:\R^d\to \R$, $[f]_{\rm Lip}=\sup_{x\neq y}\frac {|f(x)-f(y)|}{|x-y|}$.

\noindent --   For a transition $Q(x,dy)$ we define $[Q]_{\rm Lip}= \sup_{f,\, [f]_{\rm Lip}\le 1} [Qf]_{\rm Lip}$.

\noindent --   $[X]$ denotes the distribution of the random vector $X$.

\noindent --   $a_n\asymp b_n$ means that there are positive real constants $c_1,c_2>0$ such that $c_1 \,a_n \le b_n\le c_2 \,a_n$. 

\noindent --   For every $x,y\!\in \R^d$, $(x,y)= \big\{u x + (1-u)y, \, u\!\in (0,1)\big\}$. One defines likewise $[x,y]$, etc.

\noindent -- {The space of probability distributions on $(\ER^d,{\cal B}or(\ER^d))$, endowed with the topology of weak convergence is denoted by ${\cal P}(\ER^d)$.}

\noindent --   $\mathcal{W}_p(\mu, \mu')=  \inf \left\{\left(\int |x-y|^p \pi(dx,dy)\right)^{1/p},\, \pi \!\in {\cal P}_{\mu,\nu}(\R^d)\right\}$ denotes the $L^p$-Wasserstein distance between the probability distributions $\mu$ and $\mu'$ where ${\cal P}_{\mu,\nu}(\R^d)$ stands for the set of probability distributions on $(\R^d\times\R^d, {\cal B}or(\R^d)^{\otimes 2}$ with respective marginals $\mu$ and $\nu$. 

\noindent --   $\|\mu\|_{TV}= \sup\big\{\int f d\mu, \, f:\R^d\to \R,\,\hbox{Borel},\, \|f\|_{\sup}\le 1\big\}$ where $\mu$ denotes a signed measure on $(\R^d, {\cal B}or(\R^d))$ {and $d_{TV}$ denotes the related distance: $d_{TV}(\mu,\nu)=\|\mu-\nu\|_{TV}$.}

\section{Main Results}\label{sec:mains}
\subsection{Assumptions}\label{subsec:Asump}
In whole the paper, we assume that  $b$ and $\sigma$  are Lipschitz continuous and satisfy the {\em strong} mean-reverting assumption

\medskip
\noindent$\mathbf{(S)}$: There exists a positive ${\cal C}^2$-function $V:\ER^d\rightarrow (0,+\infty)$   such that
\begin{equation}\label{eq:controlVdd}
{\lim_{|x|\rightarrow+\infty} V(x)=+\infty},\quad |\nabla V|^2\le C V\quad \textnormal{and }\quad \sup_{x\in\ER^d}{\|D^2 V(x)\|_{_F}}<+\infty
\end{equation}
(Frobenius norm) and  there exist some   real constants $C_{b}>0$, $\alpha>0$  and $\beta\ge 0$  such that:
\begin{align*} 
\textit{(i)}\; |b|^2\le C_b V \;\mbox{and}\;   \s
 \mbox{ is bounded (e.g. in Frobenius norm),}\quad &\textit{(ii)}\; \big(\nabla V| b \big) \le \beta-\alpha V
\end{align*}

\begin{rem} $\bullet$ Note that $\mathbf{(S)}$ implies that $V$ attains a minimum value $\underline v>0$ (possibly at several points in $\R^d$). 

 \noindent $\bullet$   Note that 
since $\sigma$ is bounded, $(ii)$ is equivalent to the existence of $\alpha>0$ and $\beta\ge0$ such that
$$
{\cal L} V\le \beta-\alpha V.
$$
%
%
$\bullet$ Let us also remark that~\eqref{eq:controlVdd} implies that $V$ is a subquadratic function, $i.e.$ there exists a constant $C>0$ such that $V\le C(1+|\, .\,|^2)$. 

 \end{rem}
 
 Under $\mathbf{(S)}$, it is classical background (see {e.g.~\cite[Theorem ~9.3 and Lemma~9.7 with $\varphi= V$ and $\psi ={\cal L}V$]{EthierKurtz}}  that  the diffusion $(X_t)_{t\ge 0}$ (in fact its semi-group $(P_t)_{t\ge 0}$) has at least one  invariant distribution $\nu$ i.e. such that $\nu P_t = \nu$, $t\ge 0$.  {Furthermore, Assumption  $\mathbf{(S)}$  implies \textit{stability} of the diffusion and of its discretization scheme by involving long-time bounds on polynomial {(and exponential)}  moments of   $V(X_t)$ and $ V(\bar X_{\Gam_n})$. Such properties are recalled in {Proposition~\ref{prop:unifboundsES}}.}
 
 In all the main results of the paper, we will also assume  that the diffusion coefficient $\s$ satisfies the following  \textit{uniform ellipticity} assumption:
\begin{equation}\label{eq:ellipsig}
\ELLIP\; \equiv\; \exists \,\underline{\s}_0>0\; \mbox{ such that }\;\forall\, x\!\in \R^d,\qquad \sigma\sigma^*(x) \ge \underline{\s}^2_0  I_d \quad\mbox{in} \quad {\cal S}^+(d,\R).
\end{equation}

This uniform ellipticity assumption implies that, when existing, the invariant distribution is unique (see e.g.~\cite{PagesESAIM} among others).

%
%

 \noindent Finally,  we suppose that  {the semi-group $(P_t)_{t\ge 0}$ of the diffusion satisfies a contraction property at exponential rate  in for a given  distance $\mathfrak{d}$ on ${\cal P}(\ER^d)$, namely}
 
{ 
\noindent  $\mathbf{(H_{\mathfrak{d}})}$: There exist $t_0>0$ and positive constants $c$ and $\rho$ such that for every $t\ge t_0$,
$$ 
\forall\, x,\, y\!\in \R^d, \quad  \mathfrak{d}([X_t^x],[X_t^y])\le c |x-y| e^{-\rho t}.
$$
In the sequel we will use $\HWO$ and $\HTVO$, $i.e.$ the conditions  related to $\mathfrak{d}={\cal W}_1$ ($1$-Wasserstein) and  to $\mathfrak{d}=d_{TV}$ (Total variation) respectively.
Note that owing to the Monge-Kantorovich representation of ${\cal W}_1$, see~e.g.\cite{Villani},  (\textit{resp.} the definition of $d_{TV}$),  the condition $\HWO$ (\textit{resp.} $\HTVO$) also reads on Lipschitz continuous functions $f:\R^d\to \R$ (\textit{resp.} on bounded Borel-measurable functions) $f:\R^d\to \R$
$$
\forall\, t\ge t_0, \qquad [P_tf]_{\rm Lip}\le ce^{-\rho t}[f]_{\rm Lip} \quad (\textit{resp.} \quad [P_tf]_{\rm Lip}\le ce^{-\rho t}[f]_{\infty}).
$$
}
%
{
In fact, only $\HWO$ appears in the next theorems. Actually, by the regularizing effect of the elliptic semi-group, we have the  following result (whose proof is postponed to Appendix~\ref{annexe:prophwo}):
\begin{prop} \label{prop:hwo} Suppose that  $b$ and $\s$ are $C^{1}$ with bounded partial derivatives and that $\ELLIP$ is in force.  
If 
$\HWO$ holds with some positive $\rho$ and $t_0$, then  $\HTVO$  holds with the same $\rho$ and $t_0$.
\end{prop}
}
\smallskip
\begin{rem}\label{rk:2.3} 
%
 \noindent $\bullet$  If $b$ and $\s$ are  both  Lipschitz continuous and $\HWO$ holds true, then it  holds true  from the origin, i.e. for   $t_0=0$, with the same $\rho$, up to a change of the real constant $c$. {Actually},    if    $f:\R^d\to \R$ is Lipschitz continuous, then, for every $t\!\in [0,t_0]$ and every $x,\, y\!\in \R^d$,
 \[
 \big| \E f(X^x_t)-f(X^y_t)\big| \le [f]_{\rm Lip}\E |X^x_t-X^y_t|\le C_{t_0,[b]_{\rm Lip},[\s]_{\rm Lip}} [f]_{\rm Lip} |x-y|
 \] 
 by standard arguments {on the flow of the {\em SDE} (see e.g.~\cite[Theorem~7.10]{PagesBook})}. One concludes by   the Kantorovich-Rubinstein representation of ${\cal W}_1$.  {Note that for $\HTVO$, the property does certainly not extend to $t_0=0$ since $\|\delta_x-\delta_y\|_{TV}=2$ for any $x\neq y$.} 
 
%
\noindent $\bullet$ {In Assumption~$\mathbf{(H_{\mathfrak{d}})}$ (and especially in Assumption $\HWO$), we choose to base our main results on a contraction property of the semi-group, in order to avoid to mix up  discretization problems and ergodic properties of the diffusion. However, we provide in Section~\ref{subsec:Applications} a large class of examples where this assumption is fulfilled: in the {\em uniformly convex/dissipative} setting  as established in  Corollary~\ref{cor:stronglyconvex} later on but   also, when   $b$ is only  strongly contracting {\em outside a compact set} (see Corollary~\ref{cor:weaklyconvex}). When $\sigma$ is constant, one can refer to~\cite{Luo_Wang} or~\cite{eberle-guillin} for bounds in Wasserstein distance for diffusions. For background on ergodicity properties of diffusions, we also refer to~\cite{Bakry-Gentil-Ledoux,DrKoZe} or to~\cite{crisan} for the degenerate setting.}
\end{rem} 

%
\subsection{Main results}
To a non-increasing sequence of positive steps {denoted $\boldsymbol{\gamma}=(\g_n)_{n\ge 1}$} we associate  the index $$\displaystyle \varpi:=\varlimsup_n \frac{\g_n-\g_{n+1}}{\g_{n+1}^2}\!\in [0, +\infty].$$ This index is finite if and only if the  convergence of $\g_n$ to $0$ is not too fast. To be more precise,  if $\g_n = \frac{\g_1}{n^a}$ ($a>0$), $\varpi=0 $     if  $0<a<1$ and $\varpi = \frac{1}{\g_1}$ if $a=1$ and $\varpi=+\infty$ if $a>1$. We are now in position to state our main result.

\begin{thm}\label{thm:multiplicative}   Assume $\ELLIP$ and $\mathbf{(S)}$ {and $\HWO$ {with $\rho>\varpi$}}. {Let $\nu$ be the (unique) invariant distribution of $(X_t)_{t\ge0}$.}
 Suppose that the step sequence $(\g_n)_{n\ge 1}$ satisfies $(\Gam)$, that $\varpi\!\in [0, +\infty)$ and $\int |\xi|\nu(d\xi)<+\infty$. 

\smallskip 
\noindent $(a)$ If   $b$ and $\s$ are ${{\cal C}^{4}}$ with bounded derivatives, then
$$
\forall\, n\ge 1,\quad {\cal W}_1([\bar{X}_{\Gam_n}^x],\nu)\le C_{b,\sigma,\boldsymbol{\gamma}, V}\cdot \gamma_n\big| \log(\gamma_n)\big|\,  \vartheta(x)
$$
where $C_{b,\sigma,\boldsymbol{\gamma}}$  is a constant depending only on $b$, $\sigma$, $\boldsymbol{\gamma}$  and $\vartheta (x)=  (|x|+1)\vee V^2(x)$.

 \smallskip  
\noindent $(b)$  If   $b$ and  $\sigma$ are  ${C^{6}}$ with bounded existing partial derivatives and  if $\displaystyle \liminf_{|x|\to +\infty} V(x)/|x|^r>0$ for some $r\!\in(0,2]$ (resp. $\displaystyle  \liminf_{|x|\to +\infty} V(x)/\log(1+|x|)=+\infty$), then, for every small enough $\varepsilon>0$, there exists a real constant $C_\ve= C_{\ve,b,\s,\boldsymbol{\gamma}, V}>0$ such that 
\[
\forall n\ge 1,\quad \big\| [ \bar{X}^x_{\G_n}]-\nu \big\|_{TV} \le C_{\ve} \cdot \g_n ^{1-\ve}\vartheta(x)
\]
where $\vartheta(x) =  
V^{8/r}(x)\!\in L^1(\nu)$ (resp. $\vartheta(x) = e^{\lambda_0V(x)}\!\in L^1(\nu)$ for some $\lambda_0\!\in (0, \lambda_{\sup}/2)$ where $\lambda_{\sup}$ is defined in Proposition~\ref{prop:unifboundsES}$(b)$).
\end{thm}
\begin{rem}  
%
$\bullet$ The parameter $\rho$ does not appear in the above constants  since $\rho$ can be in turn considered as a function of $b$ and $\sigma$. But the
constant clearly depends on it.  For the sake of readability, we will sometimes omit the dependency in the next results. The main point is that these constants do not depend on $x$.

\noindent $\bullet$ The proofs of the above {convergence rates} certainly rely on ergodic arguments but also on {refined} bounds on the \textit{one-step weak error} between the Euler scheme and the diffusion for non-smooth functions. In particular, one important tool for the total variation bound is a {one-step} control of the weak error for bounded Borel functions when the initial condition is an ``almost'' non-degenerated  (in a Malliavin sense) random variable~(\footnote{This result is established in Theorem~\ref{thm:MalliavinWeakEr}. Among other arguments, the related proof relies on recent Malliavin bounds obtained in~\cite{bally_caramellino_poly}.}). More precisely, this random initial condition is precisely an Euler scheme at a given positive (non-small) time and $\ELLIP$ guarantees that the related Malliavin matrix is non-degenerated with high probability but not almost surely (since the tangent process of the continuous-time Euler scheme does not almost surely map into $GL_d(\ER)$). This almost but not everywhere non-degeneracy induces a cost which mainly explains that the bound in Theorem~\ref{thm:multiplicative}$(b)$ is proportional to $\gamma^{1-\varepsilon}$ and not to $\gamma|\log(\gamma)|$, as in the above claim $(a)$. However, one could wonder about the optimality of this bound and on the opportunity to get a bound in $\gamma$. Such a result could perhaps follow from a sharper control of the probability of non-degeneracy of the Euler scheme but this appears as a non trivial task, {not achieved in~\cite{bally_caramellino_poly}}. An alternative (used for instance in~\cite{Guyon})  is to {base} the proof on parametrix-type expansions of the error between the density of the Euler scheme and that of the diffusion obtained in~\cite{Konakov-Mammen}. But relying  on such an alternative would require to adapt their arguments to the decreasing step setting and to prove that the coefficients of the resulting expansion do not depend on the considered step sequence~(\footnote{More precisely, the main result of~\cite{Konakov-Mammen} {establishes existence of  error expansions reading as polynomials (null at $0$) of the step but, surprisingly,  with coefficients still ``slightly'' varying with the step. Then the authors claim that such a dependence can be canceled by further (non-detailed) arguments.}}). {Solving this problem would yield a $TV$ bound in $\g_n$ as can be checked from proof of the theorem.}
\end{rem}

Let us now turn to the {so-called additive case},  $\sigma(x)=\sigma$.
 \begin{thm}[Additive case] \label{thm:additive} Assume that   $b$ is  $C^{3}$ with bounded existing partial derivatives and $\sigma(x)\equiv \sigma$ {with $\sigma\sigma^*$ is definite positive.
 Assume $\mathbf{(S)}$ holds  and $\varpi\!\in (0, +\infty)$}. 
 %
%
If $\HWO$ holds with $\rho>\varpi$ and $\int |x|\nu(dx)<+\infty$, then there exists  a real constant $C=C_{b,\sigma,\boldsymbol{\gamma}, V}>0$ such that for all $n\ge1$,
\[
{{\cal W}_1([\bar{X}_{\Gam_n}^x],\nu)} \le C\cdot \gamma_n\,\vartheta(x)\quad \mbox{ and }\quad\big\| [ \bar{X}_{\G_n}^x ]-\nu \big\|_{TV} \le C\cdot \g_n\big|\log\big(\g_n\big)\big|\,\vartheta(x)
\]
 with  $\vartheta(x)= (1+|x|)\vee V^a(x)$ with $a=2$ for $\|\cdot\|_{TV}$ and $a=3/2$ for ${\cal W}_1$.
 
  If, furthermore,  $\liminf_{|x|\to +\infty} V(x)/|x|^r>0$ for some $r\!\in(0,2]$ (resp. $\displaystyle  \liminf_{|x|\to +\infty} V(x)/\log(1+|x|)=+\infty$),
 then there exists  a real constant $C= C_{b,\sigma,\boldsymbol{\gamma},V}$ such that for all $n\ge1$,
 \[
 {\big\| [ \bar{X}_{\G_n}^x ]-\nu \big\|_{TV}}\le C\cdot \gamma_n\,\vartheta(x)
 \]
 where $\vartheta(x)=  V^{2 \vee \frac 1r}(x)\!\in L^1(\nu)$ (resp. $\vartheta(x) = e^{\lambda_0V(x)}\!\in L^1(\nu)$ for some $\lambda_0\!\in (0, \lambda_{\sup}/2)$).

 \end{thm}
 The Wasserstein bound is thus proportional to $\gamma_n$ whereas the one in Total Variation is proportional to $\g_n\log(1/\g_n)$ or to
$\g_n$ under a very slight additional assumption.  Note that in our proof, passing from $\g_n\log(1/\g_n)$  to $\g_n$, without adding smoothness assumptions on $b$, results from a sharp combination of Bismut-Elworthy-Li formula and Malliavin calculus (see end of Subsection~\ref{subsec:proofadditive}). This bound in  $O(\gamma_n)$ is optimal (in Wasserstein or in {\it TV}-distance). Actually, explicit  computations can be done for  the Ornstein-Uhlenbeck process which lead to lower-bounds proportional to $\g_n$. To be more precise, let us  consider the $\a $-confluent centered Ornstein-Uhlenbeck process defined by
 \[
 dX_t = -\a X_t dt +\sigma dW_t, \quad X_0=0,
 \]
where $\a,\,\sigma>0$. Then, there exists $c_{\alpha}>0$ such that,  for large enough $n$ (see Section~\ref{subsec:ornuhl} for a proof),
 \[
 \big\| [\bar X_{\Gam_n}] - \nu\big\|_{TV}\ge \frac{1}{200} \min\Big( 1, \Big|1-\frac{\s^2_n}{\s^2/(2\a)} \Big|\Big)
  \ge c_\a \g_n.
   \]

 \begin{rem}\label{rem:comparisonwithothers} Although, this paper is mainly concerned with the multiplicative setting, it is interesting to compare our additive result in Theorem~\ref{thm:multiplicative} with the literature. First, note that such  bounds  have been extensively  investigated in the literature. For instance,  one retrieves $TV$-bounds  in a somewhat hidden way in works about recursive simulated annealing (see~\cite{GelMit1991},~\cite{MarPel1996}).  But more recently, many papers tackled this question, in decreasing or constant step settings with a  focus on the dependency of the constants in the dimension. Here, we  consider the first setting  and the dependency in $\g_n$. From this point of view, our {\it TV}-bounds  improve those obtained in~\cite{MouDur2017} or~\cite{dalalyan} (in $O(\sqrt{\gamma_n})$)  and are { mostly comparable  to the more recent~\cite[Theorem 14]{MouDur2019} or ~\cite{MouFlamWainBart}, up to logarithmic terms. More precisely, these two papers respectively lead (in a constant step setting)  to bounds in $O(\gamma\log \gamma)$ or  $O(\gamma\sqrt{|\log \gamma|})$, whereas in our work, we obtain a rate in $O(\gamma_n)$ with the help of a refinement of the proof based on Malliavin calculus techniques.}

%
%
%
\end{rem}
\subsection{Applications}\label{subsec:Applications}
The assumptions of the above theorems  hold under \textit{contraction} assumptions of the semi-group of the diffusion. Here, we provide some standard settings where the result applies (proofs  are postponed  to Sections~\ref{sec:proof-applic1} and~\ref{sec:proof-applic} respectively).

\medskip
\noindent $\rhd$ \textbf{Uniformly dissipative (or convex) setting}. A first classical assumption which ensures contraction properties is the following:
\begin{equation}\label{eq:Calpha}
{\bf (C_{\a})}\;  \equiv\; {\forall x,y\in\ER^d,}\quad \big(b(x)-b(y)\,|\, x-y\big)+ \tfrac 12
\|\sigma(x)-\sigma(y)\|_{_F}^2\le -{\a} |x-y|^2.
\end{equation}

In particular, if $b=-\nabla U$ where $U:\ER^d\rightarrow\ER$ is ${\cal C}^2$ and $\sigma$ is constant, this assumption is satisfied as soon as $D^2 U\ge \alpha I_d$ where $\alpha>0$ i.e. $U$ is $\a$-convex. This leads to the following result which appears as a corollary of the above theorems   (its proof is postponed in to Subsection~\ref{sec:proof-applic1}).
\begin{cor}\label{cor:stronglyconvex} Assume $\ELLIP$ and $\mathbf{(S)}$. Assume ${\bf (C_{\a})}$. Then, $\HWO$ is satisfied with $\rho=\alpha$. As a consequence, the conclusions of Theorem~\ref{thm:multiplicative} (resp. Theorem~\ref{thm:additive}  when $\sigma$ is constant)  hold true.
\end{cor}
{\begin{rem} When $\sigma$ is constant and ${\bf (C_{\a})}$ holds true, a $2$-Wasserstein bound can be directly deduced {by some discrete Gronwall like arguments based on  recursive estimates of $\E\, |X_{\Gam_n}-\bar X_{\Gam_n}|^2$}
(with $X_{\G_n}$ and $\bX_{\G_n}$ built from the same Brownian motion) combined with expansions of the one step error similar to those which lead to the control of the $L^p$-error in finite horizon  for the Milstein scheme (which coincides with the Euler-Maruyama scheme when $\sigma$ is constant), see $e.g.$ \cite[Corollary 7.2]{PagesBook}.
\end{rem}}
%
\noindent $\rhd$ \textbf{Non uniformly dissipative settings}. In fact, our main results are adapted to some settings where the contraction holds only outside a compact set. The following result is a fairly simple consequence of~\cite{wang-feng} and of our main theorems (see Section~\ref{sec:proof-applic} for a detailed proof).
\begin{cor}\label{cor:weaklyconvex}  Assume $\ELLIP$ and $\mathbf{(S)}$ {(in particular $\s$ is bounded)}. Assume  that $b$ is Lipschitz continuous and that some positive $\alpha$ and $R>0$ exist such that for all $$\forall\, x,y\in B(0,R)^c,\quad \big(b(x)-b(y)\,|\, x-y\big)\le -\alpha |x-y|^2.
$$
 Then,  $\HWO$ is  satisfied. Hence, the conclusions of Theorem~\ref{thm:multiplicative} (resp. Theorem~\ref{thm:additive} when $\sigma$ is constant)  hold true.
\end{cor}
\begin{rem} It is clear  that Assumption $\mathbf{(C_{\a})}$ implies that $\big(b(x)-b(y)\,|\, x-y\big)\le -\alpha |x-y|^2$ for all $x,y$, hence  outside any compact set. Thus Corollary~\ref{cor:weaklyconvex} contains 
Corollary~\ref{cor:stronglyconvex}. However, the first result emphasizes  that the exponent $\rho$ in Assumption $\HWO$ can be made explicit in the uniformly dissipative case, opening the way to more precise error bounds.

\smallskip
 When $\sigma$ is constant,  one can also deduce $\HWO$ in the non-uniformly dissipative case  from~\cite{Luo_Wang}  or~\cite{eberle-guillin}. 
\end{rem}


\subsection{Langevin Monte Carlo and multiplicative (multi-dimensional) SDEs}
A significant portion of the paper is devoted to the multiplicative case (in particular, a significant part of the proof of Theorem~\ref{thm:multiplicative}).
However, in applications  and in particular in the Langevin Monte-Carlo method (whose principle is recalled below), diffusions with constant {$\sigma$ are more frequently used}.  Below, we show that using multiplicative SDEs may be of interest for applications to the Langevin Monte-Carlo method. Let us recall that for a potential $V:\ER^d\rightarrow\ER$
and its related Gibbs distribution 
$$
{\nu_{_V}}(dx) = C_V{e^{-V(x)}}\cdot\lambda_d(dx) \quad\textnormal{with $C_V^{-1}=\int e^{-V(x)}\cdot\lambda_d(dx)$,}
$$
the Langevin Monte-Carlo usually refers to the numerical approximation of ${\nu_{_V}}$, viewed as the invariant distribution of  the additive SDE
\begin{equation}\label{eq:xtlangevin}
 dX_t=-\sigma^2 \nabla V (X_t) dt +\sqrt{2}\sigma d W_t,
 \end{equation}
where $\sigma$ is a positive constant (usually equal to $1$). {In fact, 
 it is  possible to exhibit a large class of multiplicative diffusions which also share with the same invariant distribution $\nu_{_V}$ as shown in Proposition~\ref{thm:generalsigma}  below. }
\begin{prop} \label{thm:generalsigma} 
Let $V:\R^d\to \R_+$ be  a $C^2$ function such that $\nabla V$ is Lipschitz continuous and $e^{-V}\!\in L^1(\lambda_d)$. Let $\s: \ER^d\to {\mathbb M}(d,q, \R)$ be a $C^1$, bounded matrix valued field   with bounded partial derivatives and satisfying $\ELLIP$. Let $(X^x_t)_{t\ge 0}$ be solution to the $SDE$
\begin{equation}\label{eq:SDEddim}
dX_t= b(X_t)dt + \sigma(X_t)dW_t,\; X_0=x,
\end{equation}
($W=(W_t)_{t\ge 0}$ standard Brownian motion defined on a probability space $(\Omega,{\cal A}, \P)$) with  drift  
\[
b = -\tfrac 12\left((\sigma\s^*)\nabla V -\Big[\sum_{j=1}^d  \partial_{x_j}(\sigma\sigma^*)_{ij}\Big]_{i=1:d}\right).
\]
Then, the distribution
\[
{\nu_{_V}}(dx) = C_{_V}e^{-V(x)}\cdot\lambda_d(dx)
\]
is the unique invariant distribution of the above  Brownian diffusion~\eqref{eq:SDEddim}.
%
\end{prop}
\noindent {The proof of this proposition is postponed to Appendix~\ref{annexe:thm:generalsigma}.}


\smallskip
\noindent For a given Gibbs distribution ${\nu_{_V}}$,  the existence of such a family of diffusions opens the opportunity  {to optimize the choice of the diffusion coefficient in view of the numerical approximation $\nu_{_V}$. In some cases, it is clearly of interest to introduce non constant diffusion coefficients. For instance, in the example below, we show that the weak mean-reverting of the Langevin diffusion (with constant $\sigma$) related to a particular Gibbs distribution $\nu_{_V}$ can be dramatically strengthened by replacing it by a diffusion with non-constant diffusion coefficient (which is shown to be strongly reverting and exponentially contracting).}

 In the same direction, in~\cite{BJY2016}, the authors show that  the optimal constant in one-dimensional weighted Poincar\'e inequalities can be obtained as the spectral gap of diffusion operators with non constant $\sigma$. This toy-example and the above reference emphasize the fact that considering non constant $\sigma$ may help  devising  procedures  whose rate of convergence can be more precisely controlled.  Using non-constant $\s$, i.e. non-isotropic colored noises in stochastic gradient procedures frequently appears in the abundant literature on machine learning (see e.g.~\cite{ma_chen_fox} or~\cite{li_chen_carlson_carin} among many others).  Nevertheless, investigating  this problem in greater depth is  beyond the scope of the paper and will be the object of future works.

\noindent {{\bf Example.}} { Let us consider the  distribution on $\R^d$ with exponent $\kappa>0$ defined by
\[
\nu_{\kappa}(dx) = \frac{C_{\kappa}}{(1+|x|^2)^{d+\kappa}} \lambda_d(dx) = C_{\kappa}e^{-V(x)}\lambda_d(dx)\quad \mbox{ with }\quad V(x) = (d+\kappa)\log(1+|x|^2)+1.
\]
By~\eqref{eq:xtlangevin} applied with $\sigma= I_d$, the distribution $\nu_\kappa$ is the invariant distribution of the one-dimensional Brownian diffusion, 
\[
dY_t = -(d+\kappa)\frac{Y_t}{1+|Y_t|^2}dt + dW_t.
\] 
Let ${\cal L}_{_Y}$denote the infinitesimal generator of this SDE. One has 
\[
{\cal L}_{_Y} V(y) = -|\nabla V(y)|^2 +\tfrac 12{\rm Tr}(\nabla^2V(y)) =-(d+\kappa)\frac{(2(d+\kappa)+1)|y|^2-1}{(1+|y|^2)^2}\sim -\frac{\big(2(d+\kappa)+1\big)(d+\kappa)}{|y|^2}
\]
as $|y|\to+\infty$. Hence,  the diffusion cannot be strongly mean-reverting since 
\[
{\cal L}_{_Y}V(y)  \to 0\quad \mbox{ as }\quad |y|\to +\infty.
\]
On the other  hand, applying now~\eqref{eq:xtlangevin} applied with $\sigma(x)= (1+|x|^2)^{1/2}I_d$, the distribution
$\nu_{\kappa}$ is also the invariant distribution of the Brownian diffusion
\begin{equation}\label{eq:renforceddiff}
dX_t = -(d+ \kappa -1) X_t dt + \sqrt{1+|X_t|^2}\,dW_t
\end{equation}
whose infinitesimal generator ${\cal L}_{_X}$ satisfies, when applied to the functions $W_{\a}(x)=(1+|x|^2)^{\a}$, $\a \!\in(0,1]$,
\[
{\cal L}_{_X}W_\a(x) \sim -\a \big(2(d+\kappa)-1-2\a\big)(|x|^2+1)^\a  \quad \mbox{ as }\quad |x|\to +\infty.
\]
Hence, {one can easily deduce} that, strong mean-reversion $\mathbf{(S)}$ holds for $W_\a$  iff  $\a <d+\kappa-\tfrac 12$ and  $\alpha \in(0,1]$ (in particular, this is always true for $\alpha=1$ when $d\ge 2$).} { Furthermore, setting $b(x)=-(d+\kappa-1) x$ and using that $x\mapsto (1+|x|^2)^\frac{1}{2}$ is $1$-Lipschitz, one also remarks that
$$\big(b(x)-b(y)\,|\, x-y\big)+ \tfrac 12
\|\sigma(x)-\sigma(y)\|_{_F}^2\le \left(-(d+\kappa-1)+\frac{d}{2}\right)|x-y|^2$$
so that  ${\bf (C_{\a})}$ is satisfied as soon as $\kappa>1-\frac{d}{2}$. Hence, for  \eqref{eq:renforceddiff}, $\HWO$ and  $\mathbf{(S)}$ hold true for any 
$\kappa>(1-\frac{d}{2})_+$ (true for any $\kappa>0$ when $d\ge 2$).
}

\subsection{Roadmap of the proof}

The sequel of the paper is devoted to the proof of the above theorems. 
%
The aim of the next Section~\ref{sec:Toolbox}  is to recall or provide   tools used to establish our main results: thus we recall in Section~\ref{subsec:friedmanbismut}, basic confluence properties,  the \textit{Bismut-Elworthy-Li} formula ({\em BEL} in what follows),  
%
Then, in Subsection~\ref{subsec:Lp-weak-errorOneStep}, we provide  a series of strong and weak error bounds for a one-step Euler scheme which will play a key role  to deduce the results (see also Appendix~\ref{app:A}).
Finally, we state  in Subsection~\ref{subsec:MalliavinWeakEr}   a general  result on weak error expansions for non-smooth functions of  the Euler scheme with decreasing step under an ellipticity assumption which relies on Malliavin calculus. The proofs of both Theorems~\ref{thm:multiplicative} and~\ref{thm:additive} are divided in several steps and  detailed in Section~\ref{sec:proofmaintheorem}, some parts of  the proofs are  postponed in the  Appendices $A$, $B$, $C$ and $D$ (to improve te readability).

\section{Toolbox and preliminary results}\label{sec:Toolbox}
Throughout the paper we will use the notations 
\begin{equation}\label{eq:S_p(x)}
S(x)= 1+|b(x)|+\|\s(x)\|\quad\mbox{ and }\quad S_{p,b,\s,\ldots}(x)= C_{_{p,b,\s,\ldots} \cdot}S(x)
\end{equation}
where $ C_{p,b,\s,\ldots}$ denotes a real constant depending on $p$, $b$, $\s$, etc, that may vary from line to line. These dependencies will sometimes be (partially) omitted.

\subsection{BEL formula and differentiability of the diffusion semi-group }\label{subsec:friedmanbismut}
We now recall the classical Bismut-Elworthy-Li formula (see~\cite{bismut,elworthy-li,cerrai2000}), referred to as {\em BEL} formula in what follows.
\begin{thm}[Bismut-Elworthy-Li formula] \label{thm:bismutmultidim} Assume $b$ and $\sigma$ are ${\cal C}^1$ with bounded first order partial derivatives. Assume furthermore that $\ELLIP$ holds. Let $f:\R^d\to \R$ be a bounded Borel function.
Then, denote by $\sigma^{-1}$ the right-inverse matrix of $\sigma$. Then,
for every $t>0$, the mapping  $x\mapsto P_tf(x)= \E\, f(X^x_t)$  is differentiable and
\begin{equation}\label{eq:BElw}
\nabla_{\!x} P_t f(x)= \E\,\nabla_{\!x}   f(X^x_t) =\nabla_x \E\Big[f(X^x_t)\frac 1t \int_0^t \big(\sigma(X^x_s)^{-1}Y^{(x)}_s\big)^* dW_s  \Big]
\end{equation}
where $(Y^{(x)}_s)_{s\ge 0}$ stands for  the tangent process  at $x$ of  the $SDE$~\eqref{eds:intro} {defined by $Y^{(x)}_t= \frac{d X^x_t}{dx}$, $t\ge 0$}.

 Moreover the above result remains true if $f$ is a
Borel function  with polynomial growth. 
 \end{thm}
 
 The proof for unbounded $f$ is postponed to Annex~\ref{Annex:BEL}.
 \normalsize
\begin{prop}\label{Prop:higherdiff}
$(a)$ Let $f:\R^d\to \R$ be a bounded Borel function.  Let $T>0$. Then for every $k=1,2,3$, there exist a real constant $C_k$ depending on $b$ and $\s$ (and possibly  on $T$)  such that,  
\begin{equation}\label{newPtfbounds1}
\forall\, t\in(0,T],\quad |\partial_{x^k} P_tf(x) |\le \frac{C_k}{\underline{\s}^k_0t^{\frac k2}}\|f\|_{\sup} .
\end{equation}
%
\noindent $(b)$ Let  $f:\R^d\to \R$ be a Lipschitz continuous  function.  Let $T>0$. Then for every $k=1,2,3$, there exist a real constant $C'_k$ depending on $b$ and $\s$ (and possibly  on $T$)  such that,  
\begin{equation}\label{newPtfbounds2}
\forall\, t\in(0,T],\quad |\partial_{x^k} P_tf(x) |\le \frac{C'_k}{\underline{\s}^k_0t^{\frac{k-1}{2}}} [f]_{\rm Lip}S(x)
\end{equation}
 \end{prop}
\noindent The proof is postponed to Appendix~\ref{subset:Proofhigherdiff}.

\subsection{One step $L^p$-strong  and weak error bounds for the  Euler scheme}\label{subsec:Lp-weak-errorOneStep}
\paragraph{Strong error.}
\begin{lem}[One step strong error~I] \label{lem:A1} Let $p\!\in [2, +\infty)$. Assume $b$ and $\sigma$ Lipschitz continuous so that $(X^x_t)_{t\ge 0}$ is well-defined as the unique  strong solution of $SDE$ starting from $x\!\in \R^d$. Let $(\bar X^{\gamma,x}_{t})_{t\in [0,\g]}$ denote the (continuous) one step Euler scheme with step $\g >0$ starting from $x$ at time $0$.

\noindent $(a)$ For every $t\!\in [0, \g]$, 
\[
\|X^x_t-\bar X^{\g,x}_t\|_p  \le  [b]_{\rm Lip} \int_0^t \| X^x_s-x\|_p ds +C_p[\sigma]_{\rm Lip}\left(\int_0^t   \| X^x_s-x\|_p^2ds \right)^{1/2}.
\]
where $C_p$ is a positive real constant only depending on $p$.

\noindent $(b)$  In particular, if $\sigma(x)= \sigma$ is a constant matrix,
\[
\|X^x_t-\bar X^{\g,x}_t\|_p  \le  [b]_{\rm Lip} \int_0^t \| X^x_s-x\|_p ds.
\]
\end{lem}


\begin{lem}[One step strong error~II] \label{lem:lem2} Assume {$b$ and $\sigma$ Lipschitz continuous}. Let $\bar \g>0$.  

\smallskip
%
\noindent $(a)$   $p\!\in [2, +\infty)$.
The diffusion process $(X^x_t)_{t\ge 0}$ satisfies for every $t\!\in [0, \bar \g]$
\begin{equation}\label{eq:lem2}
\|X^x_t-x\|_p \le S_{d,p,b,\s,\bar \g} (x)\sqrt{t} 
\end{equation}
where the underlying real constant  $C_{d, p,b,\s,\bar \g}$   depends  on $b$ and $\s$ only through $[b]_{\rm Lip}$, $[\s]_{\rm Lip}$.
As for the  one step Euler scheme $(\bar X^{\g,x}_t)_{t\ge 0}$ with step $\g\!\in (0, \bar \g]$, we have
\begin{equation}\label{eq:lem2b}
\forall\, t\!\in [0, \g],Ê\quad \|\bar X^{\g,x}_t-x\|_p \le S_{d,p,b,\s,\bar \g} (x)\sqrt{t}.
\end{equation}

\noindent $(b)$
Let  $p\!\in [1, +\infty)$. The one step  strong error satisfies, for every $\g \!\in (0,\bar \g]$ and every $t\!\in [0, \g]$,  
\begin{equation}\label{eq:Eulerfort1pas2}
 \|X^x_{t}-\bar X^{\g,x}_{t}\|_p \le   S_{d,p\vee 2,b,\s,\bar \g} (x) \left(  \tfrac 23[b]_{\rm Lip} \sqrt{t} +\frac{[\s]_{\rm Lip}}{\sqrt{2}} \right)t.
  \end{equation}
 %
  $(c)$ Let  $p\!\in [1, +\infty)$.  In particular, if $\sigma(x)= \sigma>0$ is constant, then, for every $\g>0$ and every $t\!\in [0, \g]$,
   \begin{equation}\label{eq:Eulerfort1passigmacst}
 \|X^x_{t}-\bar X^{\g,x}_{t}\|_p \le S_{d,p\vee 2,b,\s,\bar \g}(x)   
 t^{3/2}.
  \end{equation}
\end{lem}
Both proofs are postponed to the Appendix~\ref{subsec:A1}.
\paragraph{Weak error.}
We first establish a weak error bound for smooth enough functions ($C^3$, see below) with a control by its fist three derivatives. Then we apply this to the semigroup $P_tf$ where $f$ is simply Lipschitz to take advantage of the regularizing effect of the semi-group.
\begin{prop}[Weak error for smooth functions] \label{prop:holderlip2} 
Assume $b$ and $\sigma$ are ${\cal C}^2$ with bounded first and second order derivatives. Let $\bar \g>0$.
Let $g:\ER^d\rightarrow\ER$ be a  three times differentiable function. 

\smallskip
\noindent $(a)$ There exists a real constant $C_{d,b,\s,\bar\g } >0$ such that, for every $\g \!\in (0, \bar \g]$, 
\begin{equation}\label{eq:WeakErSmooth}
|\E\,[g(\bar{X}_\gamma^x)]-\E\,[g(X_\gamma^x)]|\le S_{d,b,\s,\bar\g}(x)^3\gamma^2
\Phi_{1,g}(x)
\end{equation}
where $\Phi_{1,g}(x) = \max\!\Big(|\nabla g(x)|, \|D^2g(x)\|, \Big \| \sup_{\xi\in (X^x_\g,\bar X^x_\g)}\|D^2g(\xi) \| \Big\|_2, \Big \| \sup_{\xi\in (x,X^x_\g)}\|D^3g(\xi) \| \Big\|_4 \,\Big)$
{and $(a,b)\!=\! \{\lambda a \!+\!(1\!-\!\lambda)b,  \lambda \!\in \!(0,1)\}$ stands for the open geometric interval with endpoints $a$, $b$}.

\smallskip
\noindent $(b)$  If $\s(x)= \s$ is constant, the inequality can be refined  for every $\g \!\in (0, \bar \g]$ as follows
\begin{align}
\nonumber |\E\,[g(\bar{X}_\gamma^x)]-\E\,[g(X_\gamma^x)]&- \tfrac{\g^2}{2} \tens(g,b,\s)(x)|\\
\label{eq:WeakErSmooth2} & \le 
 \gamma^2
S_{d,b,\s,\bar\g}(x)^2|\nabla g(x)|
+ \g^{5/2}\Phi_{2,g}(x)S_{d,b,\s,\bar\g}(x)^3
\end{align}
where 
\begin{equation}\label{eq:mathfrakT}
\tens(g,b,\s)(x) = \sum_{1\le i,j\le d}\partial^2_{x_ix_j}g(x)\big((\s\s^*)_{i\cdot}|\nabla b_j\big)(x),
\end{equation}
and $\displaystyle  \Phi_{2,g}(x)= \max\Big( \|D^2g(x)\|, \Big \| \sup_{\xi\in [x,X^x_\g)}\|D^3g(\xi) \| \Big\|_4$\Big).

\end{prop}
\noindent {\it Proof.} $(a)$ By the second order Taylor formula, for every $y,z\in\ER^d$,
$$
g(z)-g(y)=(\nabla g(y)|z-y) +\int_0^1 (1-u)D^2g\big(uz+(1-u)y  \big)du  (z-y)^{\otimes 2}
$$
where,  for a $d\times d$-matrix $A$ and a vector $u\in\ER^d$,
$A u^{\otimes 2}=(A u|u)$. For a given $x\in\ER^d$, it follows that
\begin{align*}
g(z)-g(y)&=(\nabla g(x)|z-y)+(\nabla g(y)-\nabla g(x)|z-y)+\int_0^1 (1-u)D^2g\big(uz+(1-u)y  \big)  (z-y)^{\otimes 2}du\\
&=(\nabla g(x)|z-y)+\big(D^2 g(x) (y-x) |z-y)\\
&\quad+\int_0^1(1-u)D^3 g(uy+(1-u)x)(y-x)^{\otimes 2} (z-y)du\\
&\quad +\int_0^1 (1-u)D^2g\big(uz+(1-u)y  \big)du  (z-y)^{\otimes 2}.
\end{align*}
Applying this expansion with $y=X_\g^x$ and $z=\bar{X}_\g^x$, this yields:
\begin{align*}
\ES\,[&g(\bar{X}_\g^x)-g(X_\g^x)]=\underbrace{(\nabla g(x)|\ES\,[\bar{X}_\g^x-X_\g^x])}_{=:A_1}+\underbrace{\ES\left[(D^2 g(x) (X_\g^x-x) |\bar{X}_\g^x-X_\g^x)\right]}_{=:A_2}\\
 &+\underbrace{\ES\left[\int_0^1(1-u)D^3 g(uX^x_\g+(1-u)x) (X_\g^x-x)^{\otimes 2} (\bar{X}_\g^x-X_\g^x)du\right]}_{=:A_3}\\
& +\underbrace{\int_0^1 (1-u)\E\, \big[D^2g\big(u\bar{X}_\g^x+(1-u)X_\g^x  \big)  (\bar{X}_\g^x-X_\g^x)^{\otimes 2}\big]du}_{=:A_4}.
 \end{align*}
Let us inspect successively the four terms of the right-hand member.  

\noindent 
\smallskip {\em Term  $A_1$.} First, 
\begin{equation}\label{eq:wekb}
\ES\,[(\bar{X}_\g^x-X_\g^x)_i]=\ES\Big[\int_0^\g \big(b(X_s)-b(x)\big)_i ds\Big]=\int_0^\g \int_0^s \ES\,[ {\cal L} b_i(X^x_u) ]du ds,
\end{equation}
{Since $b$ has bounded partial derivatives}, $|{\cal L} b_i(x)|\le C_{b,\sigma}\big(|b(x)|+\|\sigma(x)\|^2\big)$   so that
$$
|(\nabla g(x)|\ES\,[\bar{X}_\g^x-X_\g^x])|\le |\nabla g(x)| |\ES\,[\bar{X}_\g^x-X_\g^x]|\le C_{b,\sigma}\Psi (x) |\nabla g(x)| \g^2
$$
with 
\begin{equation}\label{eq:psi}
\Psi( x)= \sup_{0\le t\le\bar  \gamma} \ES\,[|b(X^x_t)| +\|\sigma(X^x_t)\|^2].
\end{equation}
 Now note  that 
\begin{align}
\nonumber \Psi(x)&\le \big( |b(x)|+ 2\,\| \s(x)\|^2 \big) + [b]_{\rm Lip}  \sup_{0\le t\le\bar  \gamma} \|X^x_t-x\|_1 +2\, [\sigma]^2_{\rm Lip}   \sup_{0\le t\le\bar  \gamma} \|X^x_t-x\|_2^2\\
\nonumber &\le  \big( |b(x)|+ 2\| \s(x)\|^2 \big) + [b]_{\rm Lip}C_{d,b,1,\s,\bar \g} S_1(x) +  [\sigma]^2_{\rm Lip}  C_{d,b,2,\s,\bar \g}S(x)^2\\
\label{eq:Psibound} &\le  S_{d,b,\s,\bar \g}(x)^2
\end{align}
 (where   real constants   $C_{d,b,p,\s,\bar \g}$  come  from Lemma~\ref{lem:lem2}).

For the sake of simplicity, we omit the dependence in $x$ in the notations  of the sequel of the proof. 
\smallskip

\noindent {\em Term $A_2$.} Temporary denoting by $u_1,\ldots,u_d$ the components of a vector $u$ of $\ER^d$, we have for every $i,j\in\{1,\ldots,d\}$,
\begin{align*}
|A_2| &\le \sum_{1\le i,j\le d}\big| \partial_{x_ix_j} g(x)\big| \big| \ES\,[(X_\g-x)_i(X_\g-\bar{X}_\g)_j]\big|
\end{align*}
\begin{align*}
\mbox{with }\quad \ES\,[(X_\g-x)_i(X_\g-\bar{X}_\g)_j]& =-\ES\,[(X_\g-\bar{X}_\g)_i(X_\g-\bar{X}_\g)_j]+\ES\,[(\bar{X}_\g-x)_i(X_\g-\bar{X}_\g)_j].\quad
\end{align*}
By  Lemma~\ref{lem:lem2}$(c)$,   
we deduce the existence of  a positive constant  $C_{b,\s,\bar \g}$ such that
$$
|\ES\,[(X_\g-\bar{X}_\g)_i(X_\g-\bar{X}_\g)_j]|\le \ES\,[|X_\g-\bar{X}_\g|^2]\le  S_{b,\s,\bar \g}(x)^2\gamma^2.
$$

 On the other hand, 
$$
(\bar{X}_\g-x)_i(X_\g-\bar{X}_\g)_j=\lt(\g b(x)+\sigma(x) W_\g\rt)_i\lt(\int_0^\g \big(b(X_s)-b(x)\big) ds+\int_0^\g \big(\sigma(X_s)-\sigma(x)\big) dW_s\rt)_{\!j},
$$
 hence (using that the increments of the Brownian Motion are independent and centered),
\begin{align}
\nonumber\ES\Big[(\bar{X}_\g-x)_i(X_\g-\bar{X}_\g)_j\Big]&=\g\, b_i (x)\ES\Big[\int_0^\g\int_0^s {\cal L}b_j(X_u)\Big] du+
\ES\Big[\int_0^\g (\sigma(x)W_\g)_i (b(X_s)-b(x))_j ds\Big]\\
\label{eq:PtDur}&\quad+\ES\left[(\sigma(x)W_\g)_i\Big(\int_0^\g  (\sigma(X_s)-\sigma(x)) dW_s\Big)_j\right].
\end{align}
By the same argument used to upper-bound $A_1$, we first get  
$$
\g \Big|b_i (x)\ES\Big[\int_0^\g\int_0^s {\cal L} b_j(X_u)\Big] duds\Big|\le C_{b,\sigma}\Psi(x)|b(x)|\gamma^3,
$$
where $\Psi$ is defined by~\eqref{eq:psi}.
Then, it follows from Cauchy-Schwarz inequality and~\eqref{eq:lem2} that
\begin{align*}
\ES\,[|(\sigma(x)W_\g)_i (b(X_s)-b(x))_j |]&\le \Big \|\sum_{1\le j\le q}\s_{ij}(x)W^j_\g\Big\|_2\| (b(X_s)-b(x))_j\|_2\\
&\le|\sigma_{i\cdot}(x)| \sqrt{\g}\,  [b]_{\rm Lip} \|X_s-x\|_2
\le   [b]_{\rm Lip}  \|\sigma(x)\|  {S_{d,2,b,\s,\bar \g}(x)}\sqrt{\g}\sqrt{s}.
\end{align*}
Hence, as $\int_0^\g \sqrt{s}ds = \frac 23 \g^{3/2}$, one has 
$$
\left|\ES\lt[\int_0^\g (\sigma(x)W_{\g})_i (b(X_s)-b(x))_j ds\rt] \right| \le C_{d,2,b,\s,\bar \g}  [b]_{\rm Lip}  \|\sigma(x)\| {S(x)}\gamma^2.
$$
For the third  term in the right hand side of~\eqref{eq:PtDur},  we deduce from It\^o's isometry that
\begin{align*}
\ES\left[(\sigma(x)W_\g)_i\Big(\int_0^\g  (\sigma(X_s)-\sigma(x)) dW_s\Big)_j\right]&=\sum_{k=1}^d \int_0^\g
\ES\,[\sigma_{i,k}(x)(\sigma_{jk}(X_s)-\sigma_{jk}(x)]ds\\
&=\sum_{k=1}^d  \sigma_{ik}(x)\int_0^\g\int_0^s \ES\,[{\cal L}\sigma_{jk}(X_u)] du ds.
\end{align*}
Since the partial derivatives of $\sigma$ are bounded, we again deduce that this term is bounded $C'_{b,\sigma}\|\sigma(x)\|\Psi(x)\gamma^2$.
Finally, collecting the above bounds yields
$$
|A_2|\le C_{b,\s,\bar{\gamma}} {\max\big(\|D^2g(x)\|, |\nabla g(x)|\big)}\max\big({S(x)}, \Psi(x)\big)(1+  \|\sigma(x)\|+   \g |b(x)|)\gamma^2.
$$

Now, we focus  on $A_3$:
\begin{align*}
|A_3|&\le\tfrac 12\ES\left[ \sup_{\xi\in (x,X^x_\g)}\|D^3g(\xi) \|  |X_\g^x-x|^2 |\bar{X}_\g^{\g,x}-X_\g^x|\right].
\end{align*}
By (three fold) Cauchy-Schwarz inequality and Lemma~\ref{lem:lem2}$(b)$
\begin{align}
\nonumber |A_3|&\le      \tfrac 12\Big\|  \sup_{\xi\in (x,X^x_\g)}\|D^3g(\xi) \| \Big\|_4  \|X_\g^x-x\|_4^{2}\|\bar{X}_\g^{\g,x}-X_\g^x\|_4\\
\label{eq:WeakErA3} &\le \tfrac 12\Big\|  \sup_{\xi\in (x,X^x_\g)}\|D^3g(\xi) \| \Big\|_4 C_{d,4,b,\s,\bar \g} S(x)^3\g^2.
\end{align}
Note that   the power $3$ in $b$ (and $\s$) comes from this term.
To conclude the proof, let consider $A_4$:
\[
 |A_4|\le \tfrac 12\Big \| \sup_{\xi\in (X^{\g,x}_\g,\bar X^{\g,x}_\g)}\|D^2g(\xi) \| \Big\|_2 \big\| \bar{X}_\g^{\g,x}-X_\g^x\big\|^2 _4\le  \tfrac{C'_{d,4,b,\s,\bar \g}}{2}\Big \| \sup_{\xi\in (X^x_\g,\bar X^x_\g)}\|D^2g(\xi) \| \Big\|_2  S(x)^2\g^2. 
\]

\noindent $(b)$ First note that the third  term in the right hand side of~\eqref{eq:PtDur} vanishes since $\s$ is constant. Secondly, note that  using the improved bound for $\|\bar{X}_\g^{\g,x}-X_\g^x\|_4$ (in $\g^{3/2}$) from Lemma~\ref{lem:lem2}$(c)$ in that setting, $\g^2$ can be replaced in the above bound for $|A_4|$ by $\g^{5/2}$.

Let us focus now on the  second term in the right hand side
of~\eqref{eq:PtDur}. We write
\begin{align*}
&\int_0^\g (\s W_s)_i\big(b_j(X^x_s)-b_j(x) \big)ds  = \int_0^{\g} (\s
W_s)_i\int_0^s {\cal L}b_j(X^x_u) duds \\
&\quad +  \int_0^{\g} (\s W_s)_i (\nabla b_j(x)|\s W_s)ds + \int_0^\g
(\s W_s)_i\int_0^s \big(\nabla b_j(X^x_u)-\nabla b_j(x)|\s
dW_u\big)ds.
\end{align*}
We inspect these three terms. Using that $W$ has independent increments, we get
\begin{align*}
\E\, \big[\int_0^\g (\s W_s)_i\int_0^s {\cal L}b_j(X^x_u) duds\big] &
=\int_0^\g \int_0^s \E\,\big[ (\s  W_u)_i{\cal L} b_j(X^x_u) \big]duds
\end{align*}
so that,  by Cauchy-Schwarz inequality,
\begin{align*}
\Big| \int_0^\g \int_0^s \E\,\big[(\s  W_u)_i{\cal L}b_j(X^x_u)
\big]duds\Big| & \le \int_0^\g \int_0^s \|(\s  W_u)_i\|_2 \|{\cal
L}b_j(X^x_u)\|_2 du ds\\
& \le C_{\|\nabla b_j\|_{\sup}, \|\s\|} \big(1+ \sup_{u\in (0,\g)}
\|b(X^x_u)\|_2 \big) \g^{5/2}\\
& \le C'_{b, \|\s\|} \big(1+  |b(x)| \big)\g^{5/2}.
\end{align*}
On the other hand, noting $(\s\s)^*_{i.}= [(\s\s)^*_{ik}]_{1\le k\le d}$,
\begin{align*}
\E\,  \int_0^\g (\s W_s)_i (\nabla b_j(x)|\s W_s\big)ds  &=
\frac{\g^2}{2}\big((\s\s^*)_{i\cdot}|\nabla b_j)
\end{align*}
Finally, using It\^o's isometry and the boundedness of second partial
derivatives of $b$, we get
\begin{align*}
\Big|\E  \int_0^\g (\s W_s)_i\int_0^s \big(\nabla b_j(X^x_u)-\nabla
b_j(x)|\s dW_u\big)ds\Big|&= \Big| \int_0^\g\E\, \Big[(\s
W_s)_i\int_0^s \big(\nabla b_j(X^x_u)-\nabla b_j(x)|\s
dW_u\big)\Big]ds\Big|\\
&\le C_{b,\s} \int_0^\g \int_0^s\|X^x_u-x\|_2\,du\,ds \le C'_{b,\s} \g^{5/2} S(x)
\end{align*}
which completes the proof.
 \hfill $\Box$

\bigskip
Combining the above results with Proposition~\ref{Prop:higherdiff}$(b)$  and Lemma~\ref{lem:PhiVrp} yields the following precise error bound for the one step weak error.

\begin{prop}[One step weak error at time $t$]  \label{prop:ptw1}   Assume  $b$ is ${\cal C}^3$  and $\sigma$ is ${\cal C}^4$  with bounded  existing partial derivatives and $|b|^2+\|\s\|^2Ê\le C\cdot V$.  Assume that $\ELLIP$ holds. Let $T, \, \bar \g>0$. 

Then, there exists a positive constant $C\!=\! C_{b,Ê\s, \underline{\s}_0,T,\bar \g,V }$ such that, for every  Lipschitz continuous function $f$ and every $t\!\in(0,T]$,
$$
\forall\,  \g \! \in (0, \bar \g], \quad |\ES\,[P_t f(\bar{X}_\gamma^{\g,x})]-\ES\,[P_t f(X_\gamma^x)]|\le C  [f]_{\rm Lip}{\gamma^2}t^{-1}
V^2(x).
\big(1+|b(x)|^3+\|\s(x)\|^3\big).
$$
\end{prop}
\noindent { \it Proof.} We apply Proposition~\ref{prop:holderlip2}$(a)$ to $g_t=P_t f(x)$ with $t>0$. It follows from Proposition~\ref{Prop:higherdiff}$(b)$ (see~\eqref{newPtfbounds2}) that the function $\Phi_{1,g}$ in~\eqref{eq:WeakErSmooth} satisfies
\begin{align*}
\Phi_{1,g_t}(x) &\le C_{b,\s,\underline{\s}_0}\frac{[f]_{\rm Lip}}{t} \max\Big( S(x),\big\|\sup_{\xi\in(X^x_\g, \bar X^x_\g)} S(\xi)\big\|_2,\big\|\sup_{\xi\in(x, X^x_\g)} S(\xi)\big\|_4 \Big)\\
&\le  C_{b,\s,\underline{\s}_0}\frac{[f]_{\rm Lip}}{t} V^{\frac 12}(x)
\end{align*}
owing to Lemma~\ref{lem:PhiVrp} in Appendix~\ref{app:A} and  where we used that $S \le C_{b,\s} V^{\frac 12}$. Consequently 
\begin{align*}
 \hskip 2cm |\E\,[P_tf(\bar{X}_\gamma^{\g,x})]-\E\,[P_t f(X_\gamma^x)]|  &\le C[f]_{\rm Lip}\gamma^2\ (1+|b(x)|^3+\|\s(x)\|^3\big) V^{\frac 12}(x) t^{-1}\\
&\le C[f]_{\rm Lip}\gamma^2 t^{-1} V^2(x). \hskip 5cm \Box
\end{align*}

\subsection{Domino-Malliavin for non smooth  functions }\label{subsec:MalliavinWeakEr}
For the control in  variation distance, we will need  a {weak error estimate for  Borel functions of the one step Euler scheme starting from a ``non-degenerate'' random variable to produce a ``regularization form the past''}. It mainly relies on a Malliavin calculus approach. In  the theorem below $(h_n)_{n\ge1}$ denotes  a non-increasing step sequence. Set $t_n=\sum_{k=1}^n h_k$ (and $t_0=0$) in what follows.  
\begin{thm}[Domino-Malliavin]\label{thm:MalliavinWeakEr}
Assume that $\sigma$ is bounded and satisfies $\ELLIP$,  that $b$ has sublinear growth: $|b(x)|\le C(1+|x|)$. Assume that 
$b$ and $\sigma$ are ${\cal C}^6$-functions with bounded partial derivatives.
 Then,  for every $\varepsilon>0$, ${T}>0$ and $\bar{h}>0$, there exists  $C_{T,\bar{h},\varepsilon}>0$  such  that for any  $h_1\in(0, \bar h)$ and any $n\ge1$ satisfying $\frac{T}{2}\le t_n\le T$
and any bounded Borel function $f:\ER^d\rightarrow\ER$, 
\begin{equation}\label{eq:mallboundun}
|\bar  P_{\h_1}\circ \cdots\circ \bar  P_{\h_{n-1}}\circ(P_{\h_n} - \bar P_{\h_n}) \circ   f(x)|\le C_{T,\bar{h},\varepsilon} (1+|x|^8) \|f\|_{\sup}h_1^{2-\varepsilon}.
\end{equation}
\end{thm}
\begin{rem} With further technicalities, it seems that we could obtain $1+|x|^6$ instead of $1+|x|^8$. Nevertheless, since the degree of the polynomial function involved in the result is not fundamental for our paper, we did not detail this point (more precisely, the improvement could be obtained by separating drift and diffusion components in the Taylor formula~\eqref{eq:taylorformulabis}.
\end{rem}

\section{Proof of the main theorems}\label{sec:proofmaintheorem}
The starting point of the proofs of both claims of  the main theorem is to decompose the error using a \textit{domino strategy}. Let us provide the heuristic by only considering a given function $f:\ER^d\rightarrow\ER$ (typically, a bounded Borel function when dealing with  the total variation distance or a $1$-Lipschitz continuous function if dealing with the $L^1$- Wasserstein distance ${\cal W}_1$). In this case, we can write:
\begin{align*}
\big| \E\, f(X^x_{\G_n})-\E\, f(\bar X^x_{\G_n})\big| & \le  \sum_{k=1}^{n} \big|  \bar  P_{\g_1}\circ \cdots\circ \bar  P_{\g_{k-1}}\circ (  P_{\g_k} - \bar P_{\g_k}) \circ P_{_{\G_n-\G_k}} f(x)  \big|.
\end{align*} 
\subsection{Proof of Theorem~\ref{thm:multiplicative}$(b)$ (Total variation distance)}
{Let $\bar \g = \|\boldsymbol{\gamma}\|=Ê\sup_{n\ge1}\g_n$. Let $T>2\bar{\gamma}$ be fixed. We may  assume without loss of generality (w.l.g.)} that $\G_n>2T$~(\footnote{When $\Gam_n\le 2T$, we can artificially upper-bound $\big| \E\, f(X^x_{\G_n})-\E\, f(\bar X^x_{\G_n})\big|$ by $2\ \|f\|_{\sup}\gamma_{N(2T)}^{-1}\gamma_n$.}). {Furthermore, under $\ELLIP$, $\HTVO$ holds for any $t_0>0$ owing to Proposition~\ref{prop:hwo}, so  we may set $t_0=\bar \g$ throughout the proof.}

\noindent For the {\it TV} distance, the idea is then to separate this sum into two partial sums, namely,
\begin{align*}
\big| \E\, f(X^x_{\G_n})-\E\, f(\bar X^x_{\G_n})\big| & \le 
 \sum_{k=1}^{N(\Gam_n-T)} \big|  \bar  P_{\g_1}\circ \cdots\circ \bar  P_{\g_{k-1}}\circ (  P_{\g_k} - \bar P_{\g_k}) \circ  P_{_{\G_n-\G_k}} f(x)  \big| \\
&\quad + \sum_{k=N(\Gam_n-T)+1}^{n} \big|  \bar  P_{\g_1}\circ \cdots\circ \bar  P_{\g_{k-1}}\circ (  P_{\g_k} - \bar P_{\g_k}) \circ   P_{_{\G_n-\G_k}}f(x)  \big|.
\end{align*}
where $f:\R^d\to \R$ is bounded Borel function. 

These two terms, say $(A)$ and $(B)$ respectively, correspond to two  different types of \textit{weak errors}:  first the  ``ergodic term'' where the exponential contraction of the semi-group can be exploited and weak error results for smooth functions (here $P_{_{\G_n-\G_k}}f$ {with $\G_n-\G_k\ge T$}) can be  used (see Proposition~\ref{prop:ptw1}), then the second term where the smoothing effect of the operator  $P_{_{\G_n-\G_k}}$ ($\Gam_n-\Gam_k\in[0,T]$) is no longer smooth enough leading us to establish a one step weak error expansion for bounded Borel functions (see Theorem~\ref{thm:MalliavinWeakEr}).




\noindent {\em Term~$(A)$.} Let $k\!\in \{1,\ldots, N(\G_n-T)\}$. Then $\Gam_n -\G_k >T$ and 
\begin{align}
\nonumber |P_{\g_k}\circ P_{_{\G_n-\G_k}} f(x)- \bar{P}_{\g_k} \circ &P_{_{\G_n-\G_k}} f(x)|   \\
&=\big| P_{\g_k}\circ P_{\frac T2} \circ P_{\G_n-\G_k-T/2} f(x)-\bar P_{\g_k} \circ P_{\frac T2} \circ P_{\G_n-\G_k-T/2} f(x) \big| \\
\nonumber  &=\big | \E\, P_{\G_n-\G_k-T/2}  f\big( \Xi^x_{k} \big)- \E \,P_{\G_n-\G_k-T/2}  f\big(\bar \Xi^x_{k} \big)  \big|\\
\label{eq:VT(A)}&\le c e^{-\rho (\G_n-\G_k-T/2)}\|f\|_{\sup}\ES\,[| X^{\Xi^x_{k}}_{\frac T2}  -X^{\bar \Xi^x_{k}}_{\frac T2}|]
\end{align}
where we applied  $\HTVO$ with $t_0= \bar \g$ at time $t= \G_n-\G_k-\frac T 2\ge \frac T 2 \ge  \bar \g=t_0$, the bounded function $f$ 
 and $\bar \Xi^x_k$ and $\Xi^x_k$ are any random vectors such that $ \Xi^x_k\stackrel{d}{=}  X_{\frac T2}^{X^x_{\g_k}}$ and $\bar \Xi^x_k\stackrel{d}{=} X_{\frac T2}^{\bar X^x_{\g_k}}$ {(having in mind that $X^x_t$ denotes the  solution of ({\em SDE})~\eqref{eds:intro} starting from $x$ at time $t$)}.

 Thus,  it follows from the definition of the $L^1$-Wasserstein distance that 
$$
\big|P_{\g_k}\circ P_{_{\G_n-\G_k}} f(x)- \bar{P}_{\g_k} \circ P_{_{\G_n-\G_k}} f(x)\big| \le C_{\rho,T} e^{-\rho (\G_n-\G_k)}\|f\|_{\sup } {\cal W}_1\big(P_{\g_k}\circ P_{\frac T2}(x,dy) ,\bar{P}_{\g_k}\circ P_{\frac T2} (x,dy) \big)
$$
with     $C_{\rho,T}=c_{{t_0}} e^{\rho T/2}$.
 On the one hand, the Kantorovich-Rubinstein  (see~\cite{Villani}) representation of the  $L^1$-Wasserstein distance says that
 \begin{align*}
  {\cal W}_1\big(P_{\g_k} \circ P_{\frac T2}  (x,dy) ,\bar{P}_{\g_k}\circ P_{\frac T2} (x,dy) \big) &=\sup_{[g]_{\rm Lip}\le 1} \E\big[ g\big(X_{\frac T2}^{X^x_{\g_k}}\big) - g\big(X_{\frac T2}^{\bar X^x_{\g_k}}\big)\big]\\
  & = \sup_{[g]_{\rm Lip}\le 1} \E\big[ P_{\frac T2}g(X^x_{\g_k}) - P_{\frac T2}g(\bar X^x_{\g_k}\big)\big]
   \end{align*}
Now, it follows from  Proposition~\ref{prop:ptw1} applied with $t=T/2$ that
$$
 \big| \E \big[   P_{\frac T2}g (X^x_{\g_k})-P_{\frac T2}g(\bar X^x_{\g_k})\big]  \big| \le   [g]_{\rm Lip} \frac 2T C_{b,Ê\s, \underline{\s}_0,T} \gamma_k^2 V^2(x)\le C'_{b,Ê\s, \underline{\s}_0,T,\|\boldsymbol{\gamma}\| } \gamma_k^2 V^2(x)
$$
 so that $ {\cal W}_1(P_{\g_k} \circ P_{\frac T2} ,\bar{P}_{\g_k}\circ P_{\frac T2})\le C'_{b,Ê\s, \underline{\s}_0,T,\|\boldsymbol{\gamma}\| } \gamma_k^2 V^2(x)$. Hence
\begin{equation}\label{eq:termergodic2b}
|P_{\g_k}\circ P_{_{\G_n-\G_k}} f(x)- \bar{P}_{\g_k} \circ P_{_{\G_n-\G_k}} f(x)|  \le C_{b,Ê\s, \underline{\s}_0,T,\|\boldsymbol{\gamma}\| } e^{-\rho (\G_n-\G_k)} \|f\|_{\sup} \g_k^2 V^2(x).
 \end{equation}
Finally, integrating with respect to $\bar P_{\g_1}\circ \cdots \circ \bar P_{\g_{k-1}}$ yields 
\begin{align*}
\big|  \bar  P_{\g_1}\circ \cdots\circ \bar  P_{\g_{k-1}}\circ (  P_{\g_k} - \bar P_{\g_k}) \circ  P_{_{\G_n-\G_k}} f(x)  \big|& \le C_{b,Ê\s,  \underline{\s}_0,T,\|\boldsymbol{\gamma}\| } e^{-\rho (\G_n-\G_k)} \|f\|_{\sup} \g_k^2 \sup_{\ell\ge 0} \E\, V^2(\bar X_{\G_\ell}^x)\\
&\le C_{b,Ê\s,  \underline{\s}_0,T,\boldsymbol{\gamma}} e^{-\rho (\G_n-\G_k)} \|f\|_{\sup} \g_k^2  V^2(x) 
\end{align*}
owing to Proposition~\ref{prop:unifboundsES}$(a)$ ({and where the constant $C_{\dots}$ may vary from line to line}). As $\varpi <\rho$, Lemma~\ref{lem:gestionsuite}$(i)$ implies  the existence of a constant $C_{\boldsymbol{\gamma}}>0$ such that
$$
 \sum_{k=1}^{N(\Gam_n-T)}\gamma_{k}^2 e^{-\rho(\Gam_n-\Gam_k)}\le C_{\boldsymbol{\gamma}}\cdot\gamma_n 
$$
so that $|(A)|Ê\le C^{(4)}_{b,Ê\s, \s_0,T,\boldsymbol{\gamma}}\|f\|_{\sup} \g_n V^2(x)$.

\smallskip
\noindent  {\em Term~$(B)$.}  Let us deal now with the the second term, when $k\!\in \{N(\G_n-T)+1, \ldots, n\}$. We assume that $n$ is large enough so that $\G_n>2T$ and temporarily set $\varphi_k= P_{\G_n-\G_k-T/2}f$. 
{We apply Theorem~\ref{thm:MalliavinWeakEr} with $t_{\ell}= \Gam_{N(\G_n-2T)+\ell}-\Gam_{N(\G_n-2T)+\ell}$, $\ell\ge 1$, $2T$ (instead of $T$), $\bar{h}=\bar{\gamma}$ and $\varepsilon\in(0,2)$. Owing to the very definition of $N(t)$ and the fact that   $\g_\ell \le \bar \g$ for every $\ell\ge 1$, one checks that
$
 \G_k-\G_{N(\G_n-2T)+1} \le \G_n-(\G_n-2T)= 2T 
 $
  and 
  $$
    \G_k-\G_{N(\G_n-2T)+1} \ge\G_n-T-(\G_n-2T+\|\gamma\|)\ge T-\bar \g\ge T/2.
$$
Hence, it follows form~\eqref{eq:mallboundun} that}
%
\[
\big|\bar P_{\gamma_{N(\G_n-2T)+1}}\circ \cdots\circ \bar P_{\g_{k-1}}\circ(P_{\g_k}- \bar   P_{\g_k})\varphi_k(x)\big|\le C_{\varepsilon} (1+|x|^8) \g_{N(\G_n-2T)+1}^{2-\varepsilon}\|\varphi_k\|_{\sup}.
\]
As a consequence
\[
\big|\bar P_{\g_1}\circ \cdots \circ \bar P_{\g_{k-1}}\circ(P_{\g_k}- \bar   P_{\g_k})\varphi_k(x)\big |\le C_{\varepsilon} \sup_{\ell\ge 1}\E(1+|\bar X_{\G_\ell}^x|^8) \g_{N(\G_n-2T)+1}^{2-\varepsilon}\|f\|_{\sup}.
\]
Finally as the step sequence satisfies~$\varpi<\rho<+\infty$  , $\g_{N(\G_n-2T)+1}= O(\g_n)$ (see Lemma~\ref{lem:gestionsuite}$(ii)$), one has
\[
\big|\bar P_{\g_1}\circ \cdots \circ P_{\g_{k-1}}\circ(P_{\g_k}- \bar   P_{\g_k})\varphi_k(x)\big |\le C'_{\g,\varepsilon}  \sup_{\ell\ge 1}\E(1+|\bar X_{\G_\ell}^x|^8)  \g_{k}^{2-\varepsilon}\|f\|_{\sup}.
\]
\indent If $c_{_{V,r}}=\liminf_{|x|\to +\infty} \frac{V(x)}{|x|^r}>0$, it follows from Proposition~\ref{prop:unifboundsES}$(a)$  that
\begin{equation}\label{eq:Vcoercive} 
\sup_{\ell\ge 1}\E\,\big(1+|\bar X_{\G_\ell}^x|^8\big)  \le  c'_{_{V,r}}  \sup_{\ell\ge 1} \E\,\big(1+ V^{8/r}(\bar X_{\G_\ell}^x)\big) \le C'_{_{V,r,\ \boldsymbol{\gamma} }} \big(1+V(x)^{8/r}\big). 
\end{equation}
 Now, by the definition of $N(\G_n-T)$ and using again  that $\varpi < \rho$, one has
 \[
 \sum_{k=N(\Gam_n-T)+1}^{n} \gamma_{k}^{2-\varepsilon} \le\g_{N(\Gam_n-T)+1}^{1-\varepsilon} \sum_{k=N(\Gam_n-T)+1}^{n} \gamma_{k}  \le \g^{1-\varepsilon} _{N(\Gam_n-T)} T\le C'_{\|\boldsymbol{\gamma}\|} T\cdot \g_n^{1-\varepsilon}.
 \]
Applying $\HTVO$,  Proposition~\ref{prop:hwo} (which allows to choose $t_0= \g_1>0$) and using that $\nu$ has a finite first moment, we have  for the diffusion   and for every $n\ge 1$,
\begin{align*}
d_{TV}([X^x_{\Gam_n}],\nu) &= \int\nu(dy)  d_{TV}([X^x_{\Gam_n}],[X^y_{\Gam_n}]) \le c_{\|\boldsymbol{\gamma}\| }\,\nu(|x-\cdot|) e^{-\rho \G_n}\\
& \le  c_{\|\boldsymbol{\gamma}\| }\,\nu(|x-\cdot|)\,\big(|x|+ \nu(|\cdot|)\big) e^{-\rho \G_n}
\end{align*}
where we used that $\nu$ is invariant. Collecting all what precedes, we get for large enough $n$,   
\begin{align*}
d_{TV}([\bar{X}^x_{\Gam_n}],\nu)\le d_{TV}([X^x_{\Gam_n}],\nu)+ d_{TV}([\bar{X}^x_{\Gam_n}],[X^x_{\G_n}])&  \le C_{b,\s, \|\boldsymbol{\gamma}\|}   {\psi}(x)\lt(  e^{-\rho\Gam_n}
+\  \gamma_{n}^{1-\varepsilon} 
+ \gamma_{n} \rt)\\
&\le C_{b,\s, \|\boldsymbol{\gamma}\|}\vartheta(x)\g_n^{1-\varepsilon}
\end{align*}
 with $ \vartheta(x) = C_{b,\s, \|\boldsymbol{\gamma}\|} V^{8/r}(x) $ (since $V^{8/r}$ dominates both $V^2$ and $|x|$) and where we used Lemma~\ref{lem:gestionsuite}$(iii)$ with $a=1$ to control $e^{-\rho\G_n}$ by $\g_n$. As $d_{TV}$ is bounded by $2$ this holds for every $n$ by changing the constant {$C_{b,\s, \|\boldsymbol{\gamma}\|} $} if necessary. 
 
\smallskip
If $\liminf_{|x|\to+\infty} V(x)/\log(1+|x|)=+\infty$, it  follows from Proposition~\ref{prop:unifboundsES}$(b)$ that  $1+|x|^8 \le c_{_{V,\lambda_0}}e^{\lambda_0V(x)}$ for any fixed $\lambda _0\!\in (0, \lambda_{\sup}]$ and that $\sup_{n\ge 1} \E \, e^{\lambda_0V(\bar X^x_{\G_n})} \le C_{b,\s,\lambda_0,\boldsymbol{\gamma}} e^{\lambda_0V(x)}$ so that one may set $\vartheta(x) = e^{\lambda_0V(x)}$ since this function also dominates $V(x)$ and $|x|$.

\subsection{Proof of Theorem~\ref{thm:multiplicative}$(a)$ (Wasserstein distance)}

Let $f:\R^d\to \R$ be a Lipschitz continuous function with coefficient $[f]_{\rm Lip}$. The idea is now to separate this sum into three parts, namely, for a given $T>0$~(\footnote{Once again, we assume w.l.g. that $\Gam_n\ge T$ keeping in mind that if $n\in\{1,\ldots, N(T)\}$, $| \E\, f(X^x_{\G_n})-\E\, f(\bar X^x_{\G_n})|$ can be artificially controlled (for instance) by $C[f]_{\rm Lip}\gamma_{N(T)}^{-1}\gamma_n$
with $C=2\big(1+\sup_{n\ge1}\ES[|X^x_{\G_n}|]+\sup_{n\ge1}\ES[|\bar{X}^x_{\G_n}|]\big)$}).
\begin{align*}
\big| \E\, f(X^x_{\G_n})-\E\, f(\bar X^x_{\G_n})\big| & \le 
 \sum_{k=1}^{N(\Gam_n-T)} \big|  \bar  P_{\g_1}\circ \cdots\circ \bar  P_{\g_{k-1}}\circ (  P_{\g_k} - \bar P_{\g_k}) \circ  P_{_{\G_n-\G_k}} f(x)  \big| \\
&\quad + \sum_{k=N(\Gam_n-T)+1}^{n-1} \big|  \bar  P_{\g_1}\circ \cdots\circ \bar  P_{\g_{k-1}}\circ (  P_{\g_k} - \bar P_{\g_k}) \circ   P_{_{\G_n-\G_k}}f(x)  \big|\\
&\quad + \big|   \bar  P_{\g_1}\circ \cdots\circ \bar  P_{\g_{n-1}}\circ(P_{\g_n} - \bar P_{\g_n}) \circ   f(x)  \big|.
\end{align*}
The three terms on the right hand side of the inequality denoted from the left to the right $(a)$, $(b)$ and $(c)$ respectively, contain three different types of \textit{weak errors}: respectively, the ``ergodic term'' $(a)$ where the exponential contraction of the semi-group can be exploited, the ``semi-regular weak error term'' $(b)$, where the smoothing effect of the operator  $P_{_{\G_n-\G_k}}$ ($\Gam_n-\Gam_k\in[ \gamma_n,T]$) helps us in controlling the weak error related to the function $x\mapsto P_{_{\G_n-\G_k}}f(x)$ and finally, the ``less smooth term'' $(c)$ where the weak error applies directly on $f$. The control of each term then relies on quite different arguments. 
 
 
\smallskip
 --  Term~$(c)$: first,   it follows from  Lemma~\ref{lem:lem2}$(b)$ with $p=2$ { and $\bar \g = \|\boldsymbol{\gamma}\|$} that
$$
\big| P_{\gamma_n}f(x)- \bar{P}_{\gamma_n}f(x)\big| \le [f]_{\rm Lip} \|X_{\g_n}^x-\bar{X}_{\g_n}^x\|_2\le  [f]_{\rm Lip} \gamma_n\Psi_1(x),
 $$
where $\Psi_1(x)= C_{d,b,\s,\|\boldsymbol{\gamma}\|}(1+|b(x)|+\|\s(x)\|)\le C_{V,d,b,\s,\|\boldsymbol{\gamma}\|}\cdot V(x)$ with  $C= C_{V,d,b,\s,\|\boldsymbol{\gamma}\|}>0$.
 
 Consequently, it follows from Proposition~\ref{prop:unifboundsES}$(a)$ 
  \[
 |(c)|\le   C\,[f]_{\rm Lip} \gamma_n\E \,V(\bar X^x_{\G_{n-1}})\le C\,[f]_{\rm Lip} \gamma_n\sup_{k\ge 0} \E\, V(\bar X^x_{\G_{k}})\le   C\,[f]_{\rm Lip} \gamma_n V(x)
  \]
 where  $C_{V,d,b,\s,\boldsymbol{\gamma}}>0$ (may vary in the above inequalities).
 
 -- Term~$(b)$. Let $k\!\in \{N(\G_n-T)+1, n-1\}$. It follows from Proposition~\ref{prop:ptw1} applied with $t= \G_n-\G_k$ and {$\bar \g=\|\boldsymbol{\gamma}\|$ so that $\g_k\le \bar \g$}  that 
$$
\big| P_{\gamma_k}\circ P_{_{\G_n-\G_k}}f(x)- \bar{P}_{\gamma_k}\circ P_{_{\G_n-\G_k}}f(x)\big|\le C_{{b,\s,\| \boldsymbol{\gamma}\|}}\,[f]_{\rm Lip} \frac{\gamma_k^{2}}{\G_n-\G_k}  V^2(x)
$$
%
which in turn implies (up to an update  of the  real constant $C_{{b,\s,\| \boldsymbol{\gamma}\|}}$)
\[
|(b)| \le   C_{{b,\s,\| \boldsymbol{\gamma}\|}}\,  V^2(x)\sum_{k= N(\G_n-T)+1}^{n-1} \frac{\gamma_k^{2}}{\G_n-\G_k}.
\]
-- Term~$(a)$.  We adopt a strategy very  similar  to that of  the proof of Theorem~\ref{thm:multiplicative}$(b)$, namely we get  a variant of~\eqref{eq:VT(A)} where $\|f\|_{\sup}$ is replaced by $[f]_{\rm Lip}$ i.e., for $n$ large enough, 
\begin{align*}
\big| P_{\g_n}\circ P_{_{\G_n-\G_k}}f(x)- \bar{P}_{\g_n} \circ P_{_{\G_n-\G_k}}f(x)\big|& \le c e^{-\rho(\G_n-\G_k-T/2)}[f]_{\rm Lip} \E \, \big
|X_{\frac T2}^{\Xi_k^x} - X_{\frac T2}^{\bar \Xi_k^x}  \big|
\end{align*}
owing to $\HWO$ applied at time $\G_n-\G_k-T/2$ where   $ \Xi^x_k\stackrel{d}{=}  X_{\frac T2}^{X^x_{\g_k}}$ and $\bar \Xi^x_k\stackrel{d}{=} X_{\frac T2}^{\bar X^x_{\g_k}}$. Finally, still following the lines of the proof of  Theorem~\ref{thm:multiplicative}$(b)$, we obtain {for a  constant $ C_{b,Ê\s, \underline \s_0,T, \boldsymbol{\gamma}} >0$}
\begin{align*}
\big|  \bar  P_{\g_1}\circ \cdots\circ \bar  P_{\g_{k-1}}\circ (  P_{\g_k} - \bar P_{\g_k}) \circ  P_{_{\G_n-\G_k}} f(x)  \big|& \le C_{b,Ê\s, \underline \s_0,T, \boldsymbol{\gamma}} e^{-\rho (\G_n-\G_k)} [f]_{\rm Lip}  \g_k^2  V^2(x) .
\end{align*}

On the other hand, applying $\HWO$, we have  for the diffusion
\begin{align*}
{\cal W}_1\big([X^x_{\Gam_n}],\nu\big) &=Ê\int_{\R^d} \nu(dy) {\cal W}_1\big([X^x_{\Gam_n}],[X^y_{\Gam_n}]\big) \\
& \le c\,\nu( |x-\cdot| ) e^{-\rho \G_n} \le c \big(|x|+ \nu(|\cdot|)\big)e^{-\rho \G_n}
\end{align*}
so that we obtain:
\[
{\cal W}_1([\bar{X}^x_{\Gam_n}],\nu)\le C_{b,\s,V, T, \|\boldsymbol{\gamma}\|} \vartheta (x)\left( e^{-\rho\Gam_n}+\sum_{k=1}^{N(\Gam_n-T)}\gamma_{k}^2 e^{-\rho(\Gam_n-\Gam_k)}
+\sum_{k=N(\Gam_n-T)+1}^{n-1}\frac{\gamma_{k}^{{2}}}{{\Gam_n-\Gam_k}}
+\g_n\right)
\]
with $  \vartheta(x)=( |x|+1)\vee V^2(x)$.   As $\varpi <\rho$, $e^{-\rho\Gam_n}+\sum_{1\le k\le N(\Gam_n-T)}\gamma_{k}^2 e^{-\rho(\Gam_n-\Gam_k)}\le C\,\g_n$ like  in the proof of claim~$(b)$, owing to  Lemma~\ref{lem:gestionsuite}$(ii)$-$(iii)$. As for the last sum, one proceeds as follows:
still using $\varpi <+\infty$, one checks that $\sup_{n\ge 1}\frac{\g_n}{\g_{n+1}}<+\infty$ so that, for $k\le n-1$, 
\[
\frac{\Gam_n-\Gam_{k-1}}{\Gam_n -\Gam_k} = \frac{\Gam_n-\Gam_{k}+\g_k}{\Gam_n -\Gam_k} = 1+ \frac{\g_k}{\Gam_n-\Gam_k} \le 1+\frac{\g_k}{\g_{k+1}}\le C_{\boldsymbol{\gamma}}.\]
Consequently (still with $C_{{\boldsymbol{\gamma}}}>0$  a real constant that may vary from line to line),
\begin{align}
\nonumber \sum_{k=N(\Gam_n-T)+1}^{n-1}\frac{\gamma_{k}^{2}} {{\Gam_n-\Gam_k}}&\le C_{{\boldsymbol{\gamma}}} \sum_{k=N(\Gam_n-T)+1}^{n-1}\frac{\gamma_{k}^{2}}{\Gam_n-\Gam_{k-1}}\\
\nonumber  &\le C_{{\boldsymbol{\gamma}}}\cdot \gamma_{N(\Gam_n-T)} \int_{\Gam_{N(\G_n-T)}}^{\G_{n-1}}\frac{1}{\Gam_n-t} dt\\
 &\le C_{{\boldsymbol{\gamma}}}\cdot \g_n\log\Big(\tfrac{\G_n-\G_{N(\G_n-T)}}{\g_n}\Big)\le C_{\boldsymbol{\gamma}} \g_n\log\Big(\tfrac{T+\|\boldsymbol{\gamma}\|_{\infty}}{\g_n}\Big)
\label{eq:gamloggamma}
\end{align}
 where we used in the second line that $(\g_n)_{n\ge 1}$ is non-increasing and a classical comparison argument between sums and integrals and, in the third line, Lemma~\ref{lem:gestionsuite}$(ii)$.
This completes the proof.

\subsection{Proof of Theorem~\ref{thm:additive}} \label{subsec:proofadditive}

We will follow the global structure of the proof of Theorem~\ref{thm:multiplicative}$(a)$ for both  distances. However,  taking advantage of the fact that when $\s$ is constant the distributions of the diffusions and the Euler scheme on finite horizon $T$  are equivalent, we will  replace Theorem~\ref{thm:MalliavinWeakEr} by a more straightforward and less technical Pinsker's inequality, as developed in the next proposition.

\begin{prop}\label{prop:TVsigcst} If $b$ is Lipschitz continuous, $\s(x)= \s \!\in GL(d,\R)$  is constant (so that it satisfies~$\ELLIP$). Then there exists a real constant $\kappa_{\s}>0$ solution to $ue^u=\frac{\underline \s_0}{\|\s\|}$ and a real constant $C= C_{b,\s}$ such that, for every $\g \!\in \big(0, \frac{\kappa_{\s}}{[b]_{\rm Lip}}\big)$ and every bounded Borel function $f:\R^d\to \R$, 
\[
\big|  \E_{\P}\, f(X^x_\g) -  \E_{\P}\, f(\bar X^{\g,x}_\g) \big| \le  \|f\|_{\sup}
C \cdot V(x)^{1/2} \g.
\]
\end{prop}
\noindent {\it Proof.} Set 
\[
\Q_\g = {\cal E}\Big(-\int_0^{\cdot} \sigma^{-1}(b(X^x_s)-b(x))dW_s\Big)_{\g}\cdot \P=  L_\g \cdot \P
\]
where $ {\cal E}$ denotes the Dol\'eans exponential.\\
 First we prove that $\Q_\g $ is a true probability measure. 
\begin{align*}
|X^x_t-x|&\le \int_0^t |b(X^x_s)-b(x)| ds +|b(x)t +\s W_t|\\
&\le [b]_{\rm Lip} \int_0^t |X^x_s-x| ds +|b(x)|t +\s W^\star_t,
\end{align*}
where $W^\star_t = \sup_{0\le s \le t} |W_s|$. By Gronwall's lemma,  
\[
|X^x_t-x| \le e^{[b]_{\rm Lip}t} \big(|b(x)|t + \s W^\star_t\big)
\]
so that 
\begin{align*}
\int_0^{\g} |X^x_t-x|^2 dt &\le e^{2[b]_{\rm Lip}\g} \int_0^{\g} \big(|b(x)| t +\s W^\star_t\big)^2 dt\\
& \le e^{2[b]_{\rm Lip}\g} \Big(|b(x)|^2(1+1/\eta) \frac{\g^3}{3}+Ê\|\s\|^2(1+\eta) \g(W^\star_\g)^2\Big),
\end{align*}
where the second inequality holds for any $\eta>0$. By Novikov's criterion (see $e.g.$~\cite{revuz-yor}),  it easily follows that $\Q_\g $is a probability measure if  for some small enough $\eta>0$,
\[
\E \,\exp{{\Big(\tfrac 12 \tfrac{ [b]_{\rm Lip}^2}{\underline \s_0^2} e^{2[b]_{\rm Lip}\g} \|\s\|^2(1+\eta) \g (W^\star)^2_\g\Big)}}<+\infty.
\]
The Brownian motions $W^1, \cdots,W^d$ being independent and $(W^\star_\g)^2\le \big((W^1)^\star_\g\big)^2+\cdots+ \big((W^d)^\star_\g\big)^2$ it is suffices (in fact equivalent) to show that
\[
\E \exp{{\Big(\tfrac 12[b]_{\rm Lip}^2 e^{2[b]_{\rm Lip}\g}  \tfrac{\|\s\|^2}{\underline \s_0^2}(1+\eta) \g ((W^1)^\star_\g)^2\big)}}<+\infty.
\]
Now, it is classical background that
\[
\E\,e^{\lambda (W^\star)^2_{t}}  \le  \E\,e^{\lambda (\overline W_t)^2}+  \E\,e^{\lambda (\overline{- W})_t^2}
\]
where $\overline B_t =\sup_{0\le s\le t} B_s$. As $-W$  is a standard Brownian motion and $\overline W_t \stackrel{\cal L}{\sim} \sqrt{t}|B_1|$, we derive that, if $\lambda t<\tfrac 12$, then 
\[
\E\, e^{\lambda (W^\star)^2_{t}} \le 2 \,\E\, e^{\lambda t B_1^2}= \frac{2}{\sqrt{1-2\lambda t}} <+\infty
\]
Consequently, the above measure $\Q_\g$ is a probability if 
$$
[b]_{\rm Lip}^2\g^2e^{2[b]_{\rm Lip}\g} <\Big( \frac{\underline \s_0}{\|\s\|}\Big)^2,
$$
which is equivalent to 
\[
0<\g< \frac{\kappa_{\s}}{[b]_{\rm Lip}},
\]
where $\kappa_{\s}$  is the unique solution to $u \,e^u = \frac{\underline \s_0}{\|\s\|} $.
By Girsanov's Theorem
\[
B_t = W_t + \int_0^t\s^{-1} \big(b(X^x_s)-b(x)\big)ds \quad \mbox{is a $\;\Q_\g$-M.B.S.}
\]
so that, under $\Q_\g$,  
\[
X^x_t = b(x) t+\s B_t, \; t\!\in [0, \g].
\]

Hence, for every bounded Borel function $f:\R^d\to \R$, 
\[
 \E_{\P}\, f(X^x_\g)  = \E_{\Q_\g} L_\g^{-1}f(x+\g b(x)+\s B_\g)\quad\mbox{ and }\quad 
 \E_{\P}\, f(\bar X^{\g,x}_\g) =  \E_{\Q_\g}  f(x+\g b(x)+\s B_\g).
\]
It follows from Pinsker's inequality (see~\cite{CeBiLu})
that 
\begin{align*}
 d_{TV}(\P, \Q_\g)^2&\le 2 \int_{\Omega} \log\big(L_\g^{-1}\big) L_\g^{-1}d\Q_\g= -2 \int_{\Omega}\log L_\g d\P\\
&= 2\E\left[\int_0^{\g}\big( \s^{-1}(b(X^x_s)-b(x))|dW_s\big) + \int_0^{\g}\big|\s^{-1}(b(X^x_s)-b(x))\big|^2 ds \right]\\
&\le \frac{[b]^2_{\rm Lip}}{\underline \s_0^2}\int_0^{\g} \E_{\P} |X^x_s-x|^2ds.
\end{align*}
It follows  from Lemma~\ref{lem:lem2} $(a)$ (see~\eqref{eq:lem2b})  and the fact that $S_2(x)= (1+ |b(x)|+ \|\s\|)$
that for $s\!\in (0, \kappa_{\s}/[b]_{\rm Lip})$
\[
\E_{\P}\, |X^x_s-x|^2 \le C'_{b,\|\s\|_{\sup}} \Big(|b(x)|^2+1 \Big)s.
\]
Hence
\begin{align*}
 d_{TV}(\P, \Q_\g)^2 &\le  C'_{b,\|\s\|_{\sup}}  \frac{[b]^2_{\rm Lip}}{\underline \s_0^2}    \big( |b(x)|^2+ 1 \big)\frac{\g^2}{2}
\end{align*}
so that, for $\g \!\in (0, \kappa_{\s}/[b]_{\rm Lip})$,
\[
d_{TV}(\P, \Q_\g) \le  C_{\underline \s_0,b,\|\s\|_{\sup}, V}V(x)^{1/2} \g
\]
Finally, for a bounded Borel function $f$ 
\[
\hskip 1cm \big|  \E_{\P}\, f(X^x_\g) -  \E_{\P}\, f(\bar X^{\g,x}_\g)\big| \le \|f\|_{\sup}  d_{TV}(\P, \Q_\g)\le \|f\|_{\sup}  C''_{\underline \s_0,b,\|\s\|_{\sup}, V}V(x)^{1/2} \g.
\hskip 1cm \Box
\]

\medskip
\begin{rem}
In fact we could avoid to call upon Pinsker's inequality by noting that 
\[
\E_{\Q_{\g}}|L_{\g}^{-1}-1| = \E_{\P}|L_{\g}-1| = \E \Big|\int_0^{\g} L_s \s^{-1} \big(b(X^x_s)-b(x)\big)dW_s  \Big|\le \frac{[b]_{\rm Lip}}{\underline{\s}_0}\left(\int_0^{\g}\|L_s\|^2_4 \|X^x_s-x\|^2_4ds \right)^{1/2}.
\] 
Then the conclusion follows from Lemma~\ref{lem:lem2} applied with $p=4$ (after having classically controlled $\sup_{0\le s \le \frac{\kappa_{\s}}{[b]_{\rm Lip}} }\|L_s\|_4 $). The resulting constants are (probably) less sharp. 
\end{rem}

\noindent \textbf{Proof of   Theorem~\ref{thm:additive}}. ({\em Wasserstein distance}).
%
Let $T>0$ be fixed and let $n$ be such that $\G_n>T$.  Like in the  proof of Theorem~\ref{thm:multiplicative}$(a)$ (see the footnote), we may assume  that $n$ is large enough so  that $\G_n>T$. Then we write for a Lipschitz continuous  function $f:\R^d\to \R$
\begin{align*}
\big| \E\, f(X^x_{\G_n})&-\E\, f(\bar X^x_{\G_n})\big|  \le 
 \sum_{k=1}^{N(\Gam_n-T)} \big|  \bar  P_{\g_1}\circ \cdots\circ \bar  P_{\g_{k-1}}\circ (  P_{\g_k} - \bar P_{\g_k}) \circ  P_{_{\G_n-\G_k}} f(x)  \big| \\
&+\big|\bar P_{\g_1} \!\circ\! \cdots \!\circ\! \bar P_{\g_{N(\G_n-T)}} \!\circ\! \big(P_{\G_n-\G_{N(\G_n-T)}}-   \bar P_{\g_{N(\G_n-T)+1}}\!\circ\! \cdots \!\circ\! \bar P_{\g_n}\big)f(x)\big)\big|.
\end{align*}

\noindent {\sc Step~1.} First we note that 
\begin{align*}
\Big| \bar P_{\g_1} \!\circ\! \cdots \!\circ\! \bar P_{\g_{N(\G_n-T)}} \circ  &\big(P_{\G_n-\G_{N(\G_n-T)}}   -   \bar P_{\g_{N(\G_n-T)+1}}\!\circ\! \cdots \!\circ\! \bar P_{\G_n-\G_{n-1}}\big)f(x)\Big| \\
&=\Big| \E \Big[f \Big(X^{\bar X^x_{\G_{N(\G_n-T)}}}_{\G_n -\G_{N(\G_n-T)}}  \Big) - f\Big(\bar X^{\bar X^x_{\G_{N(\G_n-T)}}}_{\G_n -\G_{N(\G_n-T)}}  \Big) \Big]\Big|\\
&\le [f]_{\rm Lip} \int \big| X^{\xi}_{\G_n-\G_{N(\G_n-T)}} -\bar  X^{\xi}_{\G_n-\G_{N(\G_n-T)}}\big| \P_{\bar X^x_{\G_n -\G_{N(\G_n-T)}}}(d\xi)\\
&\le  [f]_{\rm Lip} C_{_{T+\gsup}} \g_{_{{N(\G_n-T)+1}}} \int V^{1/2} \big(\xi) \P_{\bar X^x_{\G_n -\G_{N(\G_n-T)}}}(d\xi)\\
&  \le  [f]_{\rm Lip} C_{_{T+\gsup}} \g_{{N(\G_n-T)}} \E \,V^{1/2}  \big(\bar X^x_{\G_{N(\G_n-T)}}\big)\\
&\le  [f]_{\rm Lip} C_{_{T+\gsup},\boldsymbol{\gamma}} \cdot \g_{{N(\G_n-T)}}  V^{1/2}  \big( x\big),
\end{align*}
{where we used Proposition ~\ref{prop:unifboundsES} $(a)$ in the last inequality and,  in the second one, the fact} that the Euler scheme with decreasing step is of order $1$ when $\s$ is constant. This expected result  follows by mimicking the proof of  the convergence rate of the Euler scheme with decreasing step from  in~\cite{pages_panloup_aap} adapted by taking advantage of the one step strong error from Lemma~\ref{lem:lem2}$(c)$ with $p=2$~(\footnote{Thus, one shows for the Euler scheme with decreasing step, say $\d_n$ with $t_n :=\d_1+\cdots+\d_n \to +\infty$, that for every $T>0$, there exists a real constant (not depending on $(\d_n)$) such that
\[
\Big \| \max_{k:t_k\le T} |X^x_{t_k}  -\bar X^x_{t_k}  | \Big\|_2\le C_{b,\s,T} (1+|b(x)|+\|\s(x)\|) \d_{1}\le C_{b,\s,T} V^{1/2}(x)\d_{1}.
\]
}).
We know from Lemma~\ref{lem:gestionsuite}$(ii)$ that $\limsup_n \frac{\g_{\G_{N(\G_n-T)+1}}}{\g_n} \le \limsup_n \frac{\g_{\G_{N(\G_n-T)}}}{\g_n}<+\infty$ so that finally 
\[
\big|\bar P_{\g_1} \circ \cdots \circ \bar P_{\g_{N(\G_n-T)}} \circ \big(P_{\G_n-\G_{N(\G_n-T)}}f(x)  -   \bar P_{\g_{N(\G_n-T)+1}}\circ \cdots \circ \bar P_{\G_n-\G_{n-1}}f(x)\big) \big|\le  [f]_{\rm Lip} C_{_{T, \boldsymbol{\gamma}} }\!\g_{n}  V^{1/2}  \big( x\big).
\]
%
%
%
%

\smallskip
\noindent {\sc Step~2.} Let $k\!\in \{1,\ldots, N(\G_n-T)\}$. Using that $\G_n-\G_k\ge T$ and adapting the treatment of term $(A)$ in the proof of Theorem~\ref{thm:multiplicative}$(b)$, we have
\begin{align*}
\big|  \bar  P_{\g_1}\circ \cdots\circ \bar  P_{\g_{k-1}}\circ (  P_{\g_k} - \bar P_{\g_k}) \circ  P_{_{\G_n-\G_k}} f(x)  \big|& \le C_{b,Ê\s, \underline \s_0,T, \boldsymbol{\gamma}} e^{-\rho (\G_n-\G_k)} [f]_{\rm Lip}  \g_k^2  V^2(x) .
\end{align*}

%
%
Thus, it follows from the Kantorovich-Rubinstein  representation   of the ${\cal W}_1$-distance 
$$
{\cal W}_1([\bar{X}^x_{\Gam_n}],\nu)\le  C_{b,\s,T,\boldsymbol{\gamma}}\cdot \vartheta(x) \left( e^{-\rho\Gam_n}
+\g_n
 +\sum_{k=1}^{N(\Gam_n-T)}\gamma_{k}^2 e^{-\rho(\Gam_n-\Gam_k)}\right)
$$
with $\vartheta(x) = \big( V^2(x)\vee( |x|+1)\big)$ and one concludes that, since $\rho <\varpi $,  
$$
{\cal W}_1([\bar{X}^x_{\Gam_n}],\nu)\le   C_{b,\s,T,\boldsymbol{\gamma}}\cdot \g_n\vartheta(x).
$$

\smallskip
\noindent \noindent \textbf{Proof of   Theorem~\ref{thm:additive}} ({\em $TV$ distance, first {\it TV}-bound}). First note that $\ELLIP$ is satisfied so that $\HTVO$ holds by Proposition~\ref{prop:hwo}. Then, we will use~\eqref{eq:WeakErSmooth2} from Proposition~\ref{prop:holderlip2} in its less sharp form
\begin{align*}
\nonumber |\E\,[g(\bar{X}_\gamma^x)]-\E\,[g(X_\gamma^x)]| 
 &\le  \gamma^2 \max \big(\|\nabla g\|_\infty\vee\|D^2g\|_\infty\big)
S_{d,b,\s,\bar\g}(x)^2\\
&\quad + \g^{5/2}\max\big(\|D^2g\|_{\sup}, \|D^3g\|_{\sup}\big)S_{d,b,\s,\bar\g}(x)^3.
\end{align*}
 We rely again on the three-fold decomposition used for the proof of Theorem~\ref{thm:multiplicative}$(a)$, this time with $f:\R^d \to \R$ a bounded Borel function.

We still  consider $T>\bar \g$ with $ \bar \g=\|\boldsymbol{\gamma}\|$.  First, we may assume w.l.g. that  $n$ is large enough so that $\G_n>T$ and $\g_n \le  \frac{\kappa_{\s}}{2[b]_{\rm Lip}}$ (coming from the above Proposition~\ref{prop:TVsigcst}) since for $n\le N(T)\vee n_0$ (with $\g_{n_0+1}\le  \frac{\kappa_{\s}}{2[b]_{\rm Lip}}<\gamma_{n_0}$), we may artificially bound  $| \E\, f(X^x_{\G_n})-\E\, f(\bar X^x_{\G_n})|$  by $2\|f\|_{\sup} \g_{N(t)\vee n_0}^{-1} \g_n$. Then we may apply Proposition~\ref{prop:TVsigcst} and Lemma~\ref{lem:lem2} respectively with steps $\g_n$.

\smallskip
\noindent  {\em Term} $(a)$. Let $k\!\in \{1, \ldots, N(\G_n-T\}$. The proof used in Theorem~\ref{thm:multiplicative}$(b)$ {with  $\s$ non constant for term $(A)$} still works here without modification ({see in particular~\eqref{eq:termergodic2b})}: it follows from  {$\HWO$ (which implies $\HTVO$ {with $t_0=\g_1$} by Proposition~\ref{prop:hwo})} that
\begin{align*}
\big| P_{\g_k}\circ P_{_{\G_n-\G_k}}f(x) -\bar P_{\g_k}\circ P_{_{\G_n-\G_k}}f(x)  \big| 
&\le C_{b,\s,T} \|f\|_{\sup}\gamma_n^2 \,e^{-\rho(\G_n-\G_k)} V^2(x)
\end{align*}
and, as $\varpi <\rho$, one still has $\sum_{1\le k\le N(\Gam_n-T)}\gamma_{k}^2 e^{-\rho(\Gam_n-\Gam_k)}\le C_{\boldsymbol{\gamma}}\cdot\gamma_n$ which yields
\[
|(a)|\le  C_{b,\s,T,\boldsymbol{\gamma}}\g_n \|f\|_{\sup} V^2(x).
\]
\noindent  {\em Term} $(b)$. Let $k\!\in \{N(\G_n-T)+1, \ldots, n-1\}$. Applying  Proposition~\ref{prop:holderlip2}$(b)$  to $g= P_{_{\G_n-\G_k}}f$ with the help of  {\em BEL} identity and the resulting inequalities yields
\[
\big| P_{\g_k}\circ P_{_{\G_n-\G_k}}f(x) -\bar P_{\g_k}\circ P_{_{\G_n-\G_k}}f(x)  \big| \le   C_{d,b,\s,\bar \g}\cdot \|f\|_{\sup}\Big( V(x)\frac{\g^2_k}{\G_n-\G_k} + V^{3/2}(x) \frac{\g^{5/2}_k}{(\G_n-\G_k)^{3/2}}\Big).
\] 
Now, as  in the proof of Theorem~\ref{thm:multiplicative}$(a)$, still using that $\varpi< \rho$,  
\[
\sum_{k=N(\G_n-T)+1}^{n-1}\frac{\g^2_k}{\G_n-\G_k}\le  C_{\boldsymbol{\gamma}}\cdot \g_{n}\log\Big(\tfrac{T+\|\boldsymbol{\gamma}\|}{\g_n}\Big)
\]
and, proceeding likewise 
\[
\sum_{k=N(\G_n-T)+1}^{n-1}\frac{\g^{5/2}_k}{(\G_n-\G_k)^{3/2}}\le  C_{\boldsymbol{\gamma}}\cdot \g^{3/2}_{N(\G_n-T)}\int_{\G_{N(\G_n-T)}}^{\G_{n-1}}\frac{dt}{(\G_{n}-t)^{3/2}}\le C_{\boldsymbol{\gamma}}\g^{3/2}_{_{N(\G_n-T)}} \g^{-1/2}_n.
\]
It follows from Lemma~\ref{lem:gestionsuite}$(ii)$ that $\g^{3/2}_{N(\G_n-T)+1}\le C_{ \boldsymbol{\gamma},T}  \cdot \g^{3/2}_n$ so that,  still using Proposition~\ref{prop:unifboundsES}$(a)$,
\[
|(b)|Ê\le C_{b,\s, \boldsymbol{\gamma},T}\cdot \g_n.
\]
\smallskip
\noindent  {\em Term} $(c)$. It follows from the former Proposition~\ref{prop:TVsigcst} that 
\[
\big| P_{\g_n}f(x) -\bar P_{\g_n}f(x)  \big| = \big| \E\, f(X^x_{\gamma_n})-\E\, f(\bar X^x_{\gamma_n}) \big| \le C_{b,\s}\|f\|_{\sup}\g_n V^{1/2}(x).
\] 
One concludes as in the multiplicative setting.

\smallskip
\noindent 
\textbf{Proof of Theorem~\ref{thm:additive}} ({\em $TV$ distance, second {\it TV}-bound}).  {Assume now that   $T> 2\bar \g$   (still with {$\bar \g=\|\boldsymbol{\gamma}\|$})}. In addition to the former constraints on $\g_n$, we  may assume w.l.g. in this specific setting that $n\ge n_0$ where  
 $\g_{N(\Gam_{n_0}-2T)} < \frac{1}{2d\|\nabla b\|_{\infty}}$.
We rely now on a four fold decomposition
\begin{align*}
\big| \E\, f(X^x_{\G_n})&-\E\, f(\bar X^x_{\G_n})\big|  \le 
 \sum_{k=1}^{N(\Gam_n-T)} \big|  \bar  P_{\g_1}\circ \cdots\circ \bar  P_{\g_{k-1}}\circ (  P_{\g_k} - \bar P_{\g_k}) \circ  P_{_{\G_n-\G_k}} f(x)  \big| \\
&\quad + \sum_{k=N(\Gam_n-T)+1}^{n-1} \big|  \bar  P_{\g_1}\circ \cdots\circ \bar  P_{\g_{k-1}}\circ \left(  (P_{\g_k} - \bar P_{\g_k}){\circ   P_{_{\G_n-\G_k}}f(x)}-\tfrac{\g_k^2}{2} \tens(P_{_{\G_n-\G_k}}f,b,\s)\right)   \big|\\
&\quad + \tfrac 12  \left| \sum_{k=N(\Gam_n-T)+1}^{n-1} \g_{k}^2 \bar  P_{\g_1}\circ \cdots\circ \bar  P_{\g_{k-1}}\tens(P_{_{\G_n-\G_k}}f,b,\s)(x)\right|\\
&\quad + \big|   \bar  P_{\g_1}\circ \cdots\circ \bar  P_{\g_{n-1}}\circ(P_{\g_n} - \bar P_{\g_n}) \circ   f(x)  \big|.
\end{align*}

Let us call the second and third term of the decomposition  $(b)$ and $(b')$ respectively, the treatment of other terms being unchanged.

\noindent {\em Term $(b)$.} Now using the sharp form of~\eqref{eq:WeakErSmooth2} and  using the same tools ({inequalities derived from {\em BEL} identities}),  we can upper bound this ``corrected '' term by 
\[
 C_{d,b,\s,\bar \g}\cdot \|f\|_{\sup}\left( V(x)\frac{\g^2_k}{(\G_n-\G_k)^{1/2}} + V^{3/2}(x) \frac{\g^{5/2}_k}{(\G_n-\G_k)^{3/2}}\right)
\]
and we check that by the usual arguments that 
\begin{equation}\label{eq:confinbis24}
\sum_{k=N(\G_n-T)+1}^{n-1}\frac{\g^2_k}{(\G_n-\G_k)^{1/2}}\le  C_{_{\| \boldsymbol{\gamma}\|}}\g_{n}\int_{\G_{N(\G_n-T)}}^{\G_{n-1}}\frac{dt}{(\G_n-t)^{1/2}}\le C_{\|\boldsymbol{\gamma}\|}\g_n.
\end{equation}

\noindent {\em Term $(b')$.} 
First, remark that, for every $k\in\{N(\Gam_n-T)+1,\ldots,{n-1}\}$,
\begin{align*}
\bar  P_{\g_1}\circ& \cdots\circ \bar  P_{\g_{k-1}}\tens(P_{_{\G_n-\G_k}}f,b,\s)(x)= \sum_{1\le i,j\le d}\ES\,[\partial^2_{x_ix_j} P_{_{\G_n-\G_k}}f(\bar X^x_{\G_{k-1}})\big((\s\s^*)_{i\cdot}|\nabla b_j(\bar X^x_{\G_{k-1}})\big)]\\
&=\sum_{1\le i,j, \ell\le d}(\sigma\sigma^*)_{i\ell} \ES\,[\Upsilon_{i,j,\ell,k}(\bar X^x_{\Gam_{N(\Gam_n-2T)}})]
\quad \hbox{with}\quad \Upsilon_{i,j,\ell,k}(x)=\ES_x[\partial_{x_i} f_j(\bar X^{\tilde{\boldsymbol{\gamma}},x}_{t_{k-1}})\partial_\ell b_j(\bar X^{\tilde{\boldsymbol{\gamma}}, x}_{t_{k-1}})],
\end{align*}
where $\tilde \g_\ell= \G_{N(\G_n-2T)+\ell}-\Gam_{N(\G_n-2T)+\ell-1}$, $\ell\ge 1$, $\bar X^{\tilde{\boldsymbol{\gamma}},x}$ is the Euler scheme with time step sequence $\tilde{\boldsymbol{\gamma}}$,  $t_{k-1}= \G_{k-1}-\G_{N(\Gam_n-2T)}= \widetilde \G_{k-1-N(\G_n-2T)}$ and $f_j=\partial_{x_j} P_{_{\G_n-\G_k}}f$. 
The next step is  to perform an integration by parts using Malliavin calculus for  $\bar X^{\tilde{\boldsymbol{\gamma}},x}$ using the ``toolbox''  developed in Appendix~\ref{app:C} for the $TV$-convergence with varying $\s$, but taking into account that now the tangent process of the scheme is $GL_{d}(\ER)$-valued  without any truncation.
More precisely, 
with the notations of  Proposition~\ref{lem:mallderivsuc},  the tangent process $({\bar{Y}}_{t})_{t\ge0}$ {of}  the (continuous-time version of) $\bar X^{\tilde{\boldsymbol{\gamma}},x}$ reads $\tilde Y^{(x)}_0=I_d$ and \small $\tilde Y^{(x)}_{t}= (I_d+(t-\widetilde \G_{\ell-1})\nabla b (\bar X^{\tilde{\boldsymbol{\gamma}},x}_{\widetilde \G_{\ell-1}}))\tilde Y^{(x)}_{{\widetilde \Gamma}_{\ell-1}}$ \normalsize  for any $t\in[\widetilde \G_{\ell-1},\widetilde \G_{\ell}]$.  Hence, as $\tilde \g_1\le \g_{N(\Gam_{n_0}-2T)} < \frac{1}{2d\|\nabla b\|_{\infty}}$,   for any  $\Theta>0$,  $\inf_{x\in \R^d, t\in[0,\Theta]}{\rm det} (\bar Y^{(x)}_{t})$ is lower-bounded by a positive deterministic constant.  Applying this
 with $\Theta=2T+\bar \g$ 
and noting that $T/2\le t_k\le 2T+\bar\g$  for every $k\in\{N(\Gam_n-T)+1,\ldots,{n-1}\}$  for large enough $n$, one  checks that the (determinant of the) Malliavin covariance of $\bar X^{\tilde{\boldsymbol{\gamma}},x}_{t_{k-1}}$ (see Proposition~\ref{lem:mallderivsuc} for similar computations) is bounded from below by a positive constant $\kappa_{b,\s}$ only depending on $\|\nabla b\|_{\sup}$, $\underline \sigma_0^2$ and $T$. This allows us to apply~\eqref{eq:formulanondegenerate} (which comes from Lemma 2.4$(i)$ of~\cite{bally_caramellino_poly}) with $f=f_j$, $\bar{F}=  \bar X^{\tilde{\boldsymbol{\gamma}},x}_{t_{k-1}}$, $G=\partial_\ell b_j(\bar X^{\tilde{\boldsymbol{\gamma}},x}_{t_{k-1}})$ and $|\alpha|=1$.  With the notation introduced in Section~\ref{subsec:malliavinbounds}, this leads to 
$$
\big| \ES_x[\partial_{x_i} f_j(\bar X^{\tilde{\boldsymbol{\gamma}},x}_{t_{k-1}} ) \partial_\ell b_j(\bar X^{\tilde{\boldsymbol{\gamma}},x}_{t_{k-1}})] \big| \le C \|f_j\|_\infty
\Big| \ES\left[(1+|\bar X^{\tilde{\boldsymbol{\gamma}},x}_{t_{k-1}}|_{1,2}^{4d-2})(|\bar X^x_{t_k}|_{1,2}+|L\bar X^{\tilde{\boldsymbol{\gamma}},x}_{t_{k-1}}|_{1})|\partial_\ell b_j(\bar X^{\tilde{\boldsymbol{\gamma}},x}_{t_{k-1}})|_1\right]\Big|.
$$
By Proposition~\ref{Prop:higherdiff}$(a)$, $\|f_j\|_\infty\le C(\G_n-\G_k)^{-\frac{1}{2}}\|f\|_{\sup}$. By~\eqref{eq:malliavcontrol11} and Proposition~\ref{lem:mallderivsuc}$(ii)$, $\ES\,[|\bar X^{\tilde{\boldsymbol{\gamma}},x}_{t_{k-1}}|_{1,2}^p]\le C_{p,T}$  for any $p>0$, where $C_{p,T}$ does not depend on $x$ and $k$. As well, using that $\partial_\ell b_j$ is bounded with bounded partial derivatives, $\|\partial_\ell b_j(\bar X^{\tilde{\boldsymbol{\gamma}},x}_{t_{k-1}})|_{1,p}\le 
C_{p,T}$ where $C_{p,T}$ is again a constant independent of $x$ and $k$. Finally, by~\eqref{eq:malliavcontrol11} and the fact that $b$ is ${\cal C}^3$, one checks that  for any $p>0$,
$$\|L\bar X^{\tilde{\boldsymbol{\gamma}},x}_{t_{k-1}}\|_{1,p}\le C_{p,T} (1+\ES[|\bar X^{\tilde{\boldsymbol{\gamma}},x}_{t_{k-1}}|^p]^{\frac{1}{p}})\le C_{p,T}(1+|x|),$$
where in the second line, we used a Gronwall argument.
Finally, using H\"older inequality, we deduce that a constant $C_{p,T}$ exists such that
$$
\big| \ES_x[\partial_{x_i} f_j(\bar X^{\tilde{\boldsymbol{\gamma}},x}_{t_{k-1}}) \partial_\ell b_j(\bar X^{\tilde{\boldsymbol{\gamma}},x}_{t_{k-1}})]\big|\le \frac{C_{p,T}}{\sqrt{\G_n-\G_k}}(1+|x|).
$$
If $\displaystyle \liminf_{|x|\rightarrow+\infty} V(x)/|x|^r>0$, we deduce from Proposition~\ref{prop:unifboundsES}{$(a)$}  and~\eqref{eq:confinbis24},
 that
\begin{align*}
\hskip2cm   |(b')|&\le C_{p,T}\hskip-0.5cm \sum_{k=N(\Gam_n-T)+1}^{n-1} \hskip-0.15cm \frac{\g_k^2}{\sqrt{\G_n-\G_k}}\sup_{k\ge 0}\ES[V^{\frac{1}{r}}(\bar X^x_{\G_k})]\\
& \le
 C_{p,T} \hskip-0.5cm \sum_{k=N(\Gam_n-T)+1}^{n-1} \hskip-0.15cm \frac{\g_k^2}{\sqrt{\G_n-\G_k}}V^{\frac{1}{r}}(x)\le  C_{p,T, \boldsymbol{\gamma},V}\, \g_nV^{\frac{1}{r}}(x).
\end{align*}
The alternative growth assumption on $V$ can be treated likewise owing to Proposition~\ref{prop:unifboundsES}$(b)$.
%
%
\hfill$\Box$

\subsection{Proof of Corollary~\ref{cor:stronglyconvex}}\label{sec:proof-applic1}
 The result is a consequence of the following lemma.
\begin{lem}\label{prop:VT} 
Assumption  $\mathbf{(C_\a)}$ implies that $\HWO$ holds with $\rho=\alpha$. To be more precise, one has
$$
\forall\, x,\, y\!\in \R^d,\; \forall\, t\ge 0, \quad \ES|X_t^x-X_t^y|^2\le e^{-2\alpha t}|x-y|
$$
so that $\hskip 4cm{\cal W}_1([X^x_t], [X^y_t]) \le {\cal W}_2([X^x_t], [X^y_t])\le e^{-\a t}|x-y|$.
\end{lem}
 \begin{proof}  It follows from   It\^o's formula applied to $e^{2\alpha t} |X_t^x-X_t^y|^2$ that this process is a supermartingale starting from $|x-y|^2$ owing to $\mathbf{(C_\a)}$.  
\end{proof}

\subsection{Proof of Corollary~\ref{cor:weaklyconvex}}\label{sec:proof-applic}
 By Proposition~\ref{prop:hwo}, it is enough to show that $\HWO$ holds true. When $\sigma$ is constant, this is a direct consequence of~\cite{Luo_Wang}. In the multiplicative case, we rely on \cite[Theorem 2.6]{wang-feng}. Since $\sigma$ is bounded, we remark that Assumption $(2.17)$ of~\cite{wang-feng} is true as soon as there exist positive constants $K_1$, $K_2$ and $R_0$ such that for every $x$, $y\in\ER^d$,
\begin{equation}\label{bdbx}
(b(x)-b(y)|x-y)\le K_1 \mbox{\bf 1}_{\{|x-y|\le R_0\}}-K_2 |x-y|^2.
\end{equation}
But it is easy to check that this assumption is equivalent to the existence of some  $\alpha,\,R>0$ such that 
\begin{equation}\label{bdbx22}
\forall (x,y)\in B(0,R)^c,\quad (b(x)-b(y)|x-y)\le-\alpha |x-y|^2.
\end{equation}
Actually, the direct implication is obvious by setting $R=R_0$ and $\alpha= K_2$.


In order to prove the converse, set $R_0 =Ê 4R \Big( 1+ \frac{[b]_{\rm Lip}}{\a}\Big)$.  Let $x, \, y\!\in \R^d$ be such  that $|x-y|\ge R_0$. If both $x$ and $y$ lie outside $ B(0, R)$ (closed Euclidean ball centered at $0$ with radius $R$), then $(b(x)-b(y)\,|x-y) \le -\a |x-y|^2$. Otherwise, one may assume w.l.g. that $x\in B(0,R)$ and  $y\notin B(0,R)$ since $ R_0>2R$.  Then let $\widetilde x = \lambda x + (1-\lambda)y$ be such that $|\tilde x|=R$ (i.e  the  point of  the    segment $[x,y]$ which intersects the boundary of the ball $B(0,R)$). It is clear that $\lambda \!\in (0,1]$ and that
\[
x-y = \frac{\tilde x-y}{\lambda}\quad\mbox{ and} \quad 1-\lambda = \frac{|x-\tilde x |}{|x-y]}\le \frac{2R}{R_0} = \frac{\a}{2(\a +[b]_{\rm Lip})}.
\]
Consequently 
\begin{align*}
\big(b(x)-b(y)\,|\, x-y\big)&\le  \big(b(x)-b(\tilde x)\,|\, x-y\big) + \frac{\big(b(\tilde x)-b(y)\,|\, \tilde x-y\big)}{\lambda}\\
& \le [b]_{\rm Lip}|x-\tilde x||x-y| -\tfrac{\a}{\lambda}|\tilde x-y|^2\\
& = -\Big(\a \lambda -[b]_{\rm Lip}(1-\lambda)\Big)|x-y|^2\\
& = -\Big(\a - (1-\lambda) (\a+[b]_{\rm Lip})\Big)|x-y|^2\le -\tfrac{\a}{2}|x-y|^2.
\end{align*}
Finally, ~\eqref{bdbx} holds with $R_0$ defined above, $K_1 =[b]_{\rm Lip}R^2_0$ and $K_2 =-\tfrac {\a}{2}$. 
\subsection{Explicit bounds for the Ornstein-Uhlenbeck process}\label{subsec:ornuhl}
 Let us  consider the $\a $-confluent centered Ornstein-Uhlenbeck process defined by
 \[
 dX_t = -\a X_t dt +\sigma dW_t, \quad X_0=0,
 \]
where $\sigma>0$.  It satisfies $\HWO$ and   $\mathbf{(S)}$  with $\rho = \a$.  
As $X_t = e^{-\a t} \int_0^t e^{\a s} dW_s$, one checks that 
 \[
 {\rm Var}(X_t) = \sigma^2 e^{-2\a t} \int_0^t e^{2\a s}ds = \frac{\sigma^2}{2\a} \big(1-e^{-2\a t}\big)
 \]
 and its (unique)  invariant distribution is given  by $\nu = \mathcal{N}\Big(0, \frac{\sigma^2}{2\a}\Big)$.
 \smallskip 
 
\noindent Now, let us consider the Euler scheme with a decreasing step $(\g_n)_{n\ge 1}$   such that $\varpi <\a$ and $\sum_n \g_n^2<+\infty$. It reads
 \[
 \bar X_{\Gam_{n+1}} = \bar X_{\Gam_n} (1- \a \g_{n+1}) +\sigma \big(W_{\Gam_{n+1}}-W_{\Gam_n}\big), \quad X_0=0.
 \]
The scheme is centered and its variance $\sigma^2_n= {\rm Var}(\bar X_{\G_n})$ at time $\Gam_n$  satisfies $\s^2=0$ and 
 \[
 \sigma^2_{n+1} = \sigma^2_n  (1- \a \g_{n+1})^2 +\sigma^2\g_{n+1}, \; n\ge 0.  
 \]
Elementary computations show that, 
  \begin{align*} 
 \sigma^2_n -\frac{\sigma^2}{2\a} &= \frac{ \s^2}{2}\a\lt[\prod_{k=1}^n  (1-\a \g_{k})^2\rt]\sum_{k=1}^n \frac{\g^2_k }{\prod_{1\le \ell\le k} (1- \a \g_{\ell})^2} \\
 &\asymp     \int_0^{\Gam_n}e^{-2\a  (\G_n-s)}\g_{N(s)} ds \ge   \int_0^{\G_n} e^{-2\a  (\G_n-s)}\ \gamma_{N(\Gam_n)}=  \frac{1-e^{-2\a \G_n}}{2\a} \g_n\sim \frac{1}{2\a} \g_n 
 \end{align*}
 where, for two sequences $(a_n)$ and $(b_n)$,  $a_n \asymp b_n $ means $a_n = O(b_n) $ and $b_n = O(a_n)$ as $n\to +\infty$.  


Hence, one checks that, as $\sigma_n \to \sigma$,
\[
{\cal W}_1\big([\bar X_{\Gam_n}] , \nu\big)  = |\sigma_n-\sigma/\sqrt{2\a}| \E\, |Z| \asymp  \g_n.
\]
 As for the total variation distance we rely on the lower bound from~\cite{Devroyeetal} for two one dimensional Gaussian distributions (sharing the same mean)
 \[
 \big\| [\bar X_{\Gam_n}] - \nu\big\|_{TV}\ge \frac{1}{200} \min\Big( 1, \Big|1-\frac{\s^2_n}{\s^2/(2\a)} \Big|\Big)
  \ge c_\a \g_n 
   \]
 for large enough $n$  where $c_{\a} >0$ so that $ \big\| [\bar X_{\Gam_n}] - \nu\big\|_{TV}\asymp \g_n$.

\small
\bibliographystyle{alpha}
\bibliography{bib_VTergo}

\appendix
\section{Useful properties for the Euler scheme and its step sequence}\label{app:A}

\subsection{Bounds on the moments of the Euler scheme with decreasing step}

\begin{prop}\label{prop:unifboundsES}  Assume $\mathbf{(S)}$ and $(\G)$. 

\smallskip
\noindent $(a)$ For every $a>0$ such that there exist  a real constant $\kappa_{b,\s,a}>0$ such that,
for any invariant distribution $\nu$, one has 
$$
\nu\big(V^a\big) \le \kappa_{b,\s,a} 
$$  
Furthermore,  there exist  real constants $C_{b,\s,a}>$  and $\bar C_{b,\s,a, \boldsymbol{\gamma}}>0$ such that, for every $x\!\in \R^d$,  
\begin{equation}\label{eq:controlVPanyorder}
   \sup_{t\ge 0} \ES \, V^a (X_t^x) \le C_{b,\s,a} V^a(x)  \quad \mbox{ and }\quad \sup_{n\ge 0}\ES\,   V^a (\bar{X}_{\Gam_n}^x)\le \bar C_{b,\s,a,\boldsymbol{\gamma}} V^a(x) .
\end{equation}
\noindent $(b)$   There exists $\lambda_{\sup}>0$ such that, for any invariant distribution $\nu$, for every $x\!\in \R^d$ and $\forall\, \lambda \!\in (0, \lambda_{\sup})$, $\displaystyle  \nu\big( e^{\lambda V}\big) <+\infty$.

Furthermore, there exists real constants $C_{b,\s,\lambda}>0$ and $C_{b,\s,\lambda,  \boldsymbol{\gamma}}>0$ such that, for every $x\!\in \R^d$, 
\begin{equation}\label{eq:controlVPanyorder2}
\sup_{t\ge 0} \ES \, e^{\lambda V (X_t^x)}\le C_{b,\s,\lambda} e^{\lambda V(x)} \quad \mbox{ and }\quad \sup_{n\ge 0}\ES\, e^{\lambda V (\bar{X}_{\Gam_n}^x)}\le C_{b,\s,\lambda,  \boldsymbol{\gamma}} e^{\lambda V(x)}. 
\end{equation}

\end{prop}
The  bounds in $(a)$ are straightforward consequences of (the proof of) Lemma 2
in~\cite{LambPag1}  (established in more general setting where $\s$ is possibly unbounded). The bounds in~$(b)$ are established in~\cite[Theorem II.1]{lemairethese} for the diffusion and~\cite[Corollary~III.1]{lemairethese}  for the Euler scheme (see also \cite[Lemma D.5 and D.6]{gadat2020cost} for sharper exponential bounds in the additive setting).

\subsection{Strong $L^p$-errors for the one-step Euler scheme (proofs of Lemmas~\ref{lem:A1} and~\ref{lem:lem2})}\label{subsec:A1}
%
%
{\bf Proof of Lemma~\ref{lem:A1}.} $(a)$ It follows from the generalized Minkowski inequality and the B.D.G. inequality  that
\begin{align*}
\qquad\|X^x_t-\bar X^{\g,x}_t\|_p & \le \Big \| \int_0^t \big(b(X^x_s)-b(x) \big)ds   \Big\|_p +  \Big \| \int_0^t \big(\sigma(X^x_s)-\sigma(x) \big)dW_s   \Big\|_p\\
&\le [b]_{\rm Lip} \int_0^t \| X^x_s-x\|_p ds +  C^{BDG}_p[\sigma]_{\rm Lip} \left( \Big\| \int_0^t  |X^x_s-x|^2 ds\Big\|_{\frac p2}\right)^{1/2}\\
& \le  [b]_{\rm Lip} \int_0^t \| X^x_s-x\|_p ds +  C^{BDG}_p[\sigma]_{\rm Lip} \left( \int_0^t \| X^x_s-x\|^2_p ds\right)^{1/2}
\end{align*}
where $[\s]_{\rm Lip}$ should be understood with respect to the Frobenius norm.\hfill$\Box$
\medskip
\noindent {\bf Proof of Lemma~\ref{lem:lem2}} $(a)$ 
%
%
One has by the general Minkowski inequality and BDG inequality 
\begin{align*}
\|X^x_t-x\|_p &\le \int_0^t \| b(X^x_s)\|_pds  +   \left\| \int_0^t   \s(X^x_s)  dW_s\right\| _{ p}\\
&\le t|b(x)| + \sqrt{t} \|W_1\|_p \|\s(x)\| +\int_0^t \| b(X^x_s)-b(x)\|_pds  +   \left\| \int_0^t   (\s(X^x_s) -\s(x)) dW_s\right\| _{ p}\\
&\le t|b(x)| + \sqrt{t} \|W_1\|_p \|\s(x)\| +[b]_{\rm Lip}\int_0^t \| X^x_s-x\|_pds  +  C^{BDG}_{d,p} \left\| \int_0^t \| \s(X^x_s)-\s(x)\|^2 ds\right\|^{1/2}_{\frac p2}\\
&\le t|b(x)| + \sqrt{t} \|W_1\|_p \|\s(x)\| +[b]_{\rm Lip}\int_0^t \| X^x_s-x\|_pds  +  [\s]_{\rm Lip} C^{BDG}_{d,p}\left(  \int_0^t \| X^x_s-x\|_p^2 ds\right)^{1/2}.
\end{align*}
Set $\varphi(t) = \sup_{0\le s\le t}\|X^x_s-x\|_p$ and $\psi(t)= t|b(x)| + \sqrt{t} \|W_1\|_p \|\s(x)\| $. Both  functions are  nondecreasing  so that one derives from the above inequality that 
\[
\varphi(t) \le \psi(t)+ [b]_{\rm Lip}\int_0^t \varphi(s)ds +  [\s]_{\rm Lip} C^{BDG}_{d,p}\left(  \int_0^t \varphi(s)^2 ds\right)^{1/2}.
\]
Now using that $\varphi$ is non-decreasing, we derive for every  $a >0$,
$$
\left(  \int_0^t \varphi(s)^2 ds\right)^{1/2}\le \sqrt{\varphi(t)}\sqrt{  \int_0^t \varphi(s) ds}\le \tfrac a2 \varphi(t) +  \tfrac{1}{2a} \int_0^t \varphi(s)ds. 
$$ 
As a consequence, setting  $a = \frac{1}{[\s]_{\rm Lip} C^{BDG}_{d,p}}$, yields
\[
\varphi(t) \le 2 \psi(t) + \Big( 2[b]_{\rm Lip} +(C^{BDG}_{d,p}[\s]_{\rm Lip})^2 \Big)\int_0^t\varphi(s)ds.
\]
It follows from Gronwall's Lemma that, for every $t\!\in [0, \bar \g]$
\[
\varphi(t) \le 2\, e^{(2[b]_{\rm Lip} + [\s]_{\rm Lip}^2C^{BDG}_{d,p}\bar \g}  \psi(t)
\]
which  completes the proof.

\smallskip
\noindent $(b)$-$(c)$ Having in mind that $\|\cdot\|_p\le \|\cdot\|_{p\vee2}$, it follows from Lemma~\ref{lem:A1} that
\begin{align*}
\quad \|X^x_t-\bar X^{\g,x}_t\|_p  &\le  S_{p\vee2}(x)\left([b]_{\rm Lip} \!\int_0^t \!\!\sqrt{s} ds +[\sigma]_{\rm Lip}\Big(\!\int_0^t   \!\!sds\Big)^{1/2} \right)
=  S_{p\vee 2}(x) \left(  \tfrac 23[b]_{\rm Lip} \sqrt{t} +\frac{[\s]_{\rm Lip}}{\sqrt{2}} \right)t.
\quad\Box
\end{align*}


\begin{lem} \label{lem:PhiVrp}$(a)$  Let $\Phi: \R^d \to (E, |\cdot|)$ be a Borel function with values in a normed vector space $E$ and let $V: \R^d\to (0,+\infty)$ be a function such that $\sqrt{V}$ is Lipschitz continuous.  If  
\[
|\Phi|\le  C\cdot V^r \quad \mbox{ for some $C,\,r>0$},
\]
then, for any $L^p(\P)$-integrable $\R^d$-valued random vectors $Y,Z$, $p\!\in [1, +\infty)$, 
\[
\Big\|\sup_{\xi\in (Y,Z)}|\Phi(\xi)|  \Big\|_p \le C_{\Phi,V,r}\Big(\big\|V(Y)\wedge V(Z) \Big\|^r_{rp}+\|Y-Z\|^{2r}_{2rp}\Big).
\]
\noindent $(b)$ Assume that the diffusion coefficients $b$ and $\s$ are Lipschitz continuous  and satisfy $|b|^2+\|\s\|^2\le C.V$ where $\sqrt{V}$ is Lipschitz. Then, there exists a real constant  for every $C _{\Phi,V,b,\s,p,\bar \g}$ such that, for every $\g \!\in (0, \bar \g)$,
\begin{equation}\label{eq:LpPhiV1}
\Big\|\sup_{\xi\in (x,X^x_\g)}|b(\xi)|  \Big\|_p+  \Big\|\sup_{\xi\in (X^x_\g,\bar X^x_\g)}|\s(\xi)|  \Big\|_p\le C _{\Phi,V,b,\s,p,\bar \g} V^{1/2}(x).
\end{equation}
\end{lem}

\begin{proof} $(a)$ This follows from the fact that $\sqrt{V}$ is Lipschitz continuous owing to assumption  $\mathbf{(S)}$ so  that, for every $\xi\in (Y,Z)$, 
\[
\sqrt{V}(\xi)- \sqrt{V}(Z) \le  [\sqrt{V}]_{\rm Lip}|\xi-Z|\le  [\sqrt{V}]_{\rm Lip}|Y-Z|
\]
and in turn
\[
V(\xi)^r \le 2^{(2r-1)_+}\big( V(Z)^r +    [\sqrt{V}]^{2r}_{\rm Lip}|Y-Z|^{2r}\big).
\]
One concludes using $L^p$-Minkowski's inequality. 

\smallskip
\noindent $(b)$ Note that by Lemma~\ref{lem:lem2}$(a)$, $\|X^x_\gamma-x\|_{rp} \le \bar\g^{\frac 12} S_{rp,b,\s}(x) \le   \bar\g^{\frac 12}   V^{1/2}(x)$ which yields the bound for the  first term on the left hand side. As for the second term, one proceeds likewise using Lemma~\ref{lem:lem2}$(b)$. 
\end{proof}


%


\subsection{Technical lemmas on the steps}\label{app:B}
\begin{lem}\label{lem:gestionsuite} Let $(\g_n)_{n\ge1}$ be a non-increasing positive sequence such that 
$$
\varpi = \limsup_n \frac{\g_n-\g_{n+1}}{\g_{n+1}^2}<+\infty.
$$

\noindent $(i)$ Let $\rho>\varpi$ and let $(u_n)_{n\ge0}$ be the sequence  defined by $u_0=0$ and, for every $n\ge 1$, by
$$ 
 u_n=\sum_{k=1}^n\gamma_k^2e^{-\rho (\G_n-\G_k)}.
$$
Then, \hskip4cm $\displaystyle \limsup_{n} \frac{u_n}{\g_n}<+\infty.$

\medskip
\noindent $(ii)$ 
For every $T>0$, we have
$$
{\limsup_n }\frac{\g_{N(\G_n-T)}}{\g_{n}}<+\infty
$$
(where $N(t)$ is defined in~\eqref{eq:N(t)t}).

\smallskip
\noindent $(iii)$ Assume $\rho> \varpi$. Then for any  $a\!\in \big(0,\frac{\rho}{\varpi}\big)$, 
$$
e^{-\rho \Gam_n} =o(\g^a_n)\quad\textnormal{as $n\rightarrow+\infty$}.
$$
\end{lem}
\begin{proof}$(i)$ Set $v_n=\frac{u_n}{\g_n}$, $n\ge1$.  We have:
\begin{equation}\label{eq:vnag}
v_{n+1}={v_n\theta_n}+\g_{n+1}\quad\textnormal{with}\quad \theta_n=\frac{\g_n}{\g_{n+1}}e^{-\rho \g_{n+1}}.
\end{equation}
{Under the assumption, there exists $c\!\in(\varpi,\rho)$ and  
$n_0\in\mathbb{N}$ such that for all $n\ge n_0$,
\begin{equation}\label{ineq:gn}
\frac{\g_n}{\g_{n+1}}\le 1+c \g_{n+1}\le e^{c \g_{n+1}}.
\end{equation}
Thus, for $n\ge n_0$, $\theta_n\le e^{(c-\rho)\g_{n+1}}$ so that 
plugging this inequality into~\eqref{eq:vnag}, we deduce 
\[
v_{n+1}Ê\le v_n e^{(c-\rho)\g_{n+1}}+\g_{n+1}
\]
or, equivalently, 
\[
e^{(\rho-c)\Gam_{n+1}}v_{n+1}Ê\le e^{(\rho-c)\Gam_{n}}v_{n}+C'e^{(\rho-c)\Gam_{n}}  \g_{n+1}
\]
where $C' = \sup_{k\ge 1}e^{(\rho-c)\g_k}$. Hence, by  induction
\begin{align*}
e^{(\rho-c)\Gam_{n}}v_{n}&\le e^{(\rho-c)\Gam_{n_0}}v_{n_0} + C' \int_{\Gam_{n_0}}^{\Gam_n} e^{(\rho-c)u}du
\le  e^{(\rho-c)\Gam_{n_0}}v_{n_0} + \frac{C'}{\rho-c} e^{(\rho-c)\Gam_{n}}
\end{align*}
which clearly implies the announced boundedness.}

\smallskip
\noindent $(ii)$ By~\eqref{ineq:gn}, for large enough $n$,
$$
\frac{\g_{N(\G_n-T)}}{\g_{n}}=\prod_{k=N(\G_n-T)}^{n-1}\frac{\g_k}{\g_{k+1}}\le  e^{c(\Gam_n-(\G_{N(\G_n-T)}))}\le e^{c(T+\|\boldsymbol{\gamma}\|)}.
$$

\smallskip
\noindent $(iii)$ {Set $w_n=e^{-\rho \Gam_n}/\g^a_n $. Let $\varepsilon>0$ be such that $a(\varpi+\varepsilon) <\rho$. Note that, for $n\ge n_0$, such that $\frac{\gamma_n-\gamma_{n+1}}{\gamma^2_{n+1}}\le \varpi+\varepsilon$ for every $n\ge n_0$, 
\begin{align*}
w_{n+1}=w_ne^{-\rho \g_{n+1}}\Big(\frac{\g_n}{\g_{n+1}}\Big)^a &=w_n e^{-\rho \g_{n+1}} e^{a\log(1+ \frac{\gamma_n-\gamma_{n+1}}{\gamma_{n+1}})}\\
&\le w_ne^{(a(\varpi+\varepsilon)-\rho) \g_{n+1}} \le w_{n_0}e^{(a(\varpi+\varepsilon)-\rho)( \Gam_{n+1}-\Gam_{n_0})}.
\end{align*}
Hence, $\lim_n w_n=0$ since $a(\varpi+\varepsilon)-\rho<0$ and 
$\sum_{k\ge1} \g_k=+\infty$.}
\end{proof} 

\section{Proof of Domino-Malliavin Theorem}\label{app:C}
The aim of this section is to prove Theorem~\ref{thm:MalliavinWeakEr}. The proof is achieved in Subsection~\ref{subsec:proofdommal} {but} strongly relies on a series of Malliavin bounds established in Subsection~\ref{subsec:malliavinbounds}. Note that w.l.g., we may only prove the result for $\bar{h}$ small enough. Actually, since the left-hand side of the inequality is bounded by $2$, we can always extend to $\tilde{h}$ larger than $\bar{h}$ by artificially bounding the left-hand side by $2\bar{h}^{\varepsilon-2} h_1^{2-\varepsilon}$ for any $h_1$ greater than $\bar{h}$. 

\subsection{Proof of Theorem~\ref{thm:MalliavinWeakEr}}\label{subsec:proofdommal}
%
%
By classical density arguments, it is enough to prove the result for a smooth function $f:\ER^d\rightarrow\ER$ with bounded derivatives as soon as the constant $C$ {of Inequality~\eqref{eq:mallboundun}} only depends on $\|f\|_\infty$. Throughout the proof,  $f:\ER^d\rightarrow\ER$ is thus assumed to be ${\cal C}^\infty$, bounded with bounded derivatives.

\medskip
\noindent \textit{Step 1 (Expansion of $(P_\h-\bar{P}_\h) f(\xi))$.} Let $\xi\in\ER^d$ and let $\h>0$.
 We have
$$ P_\h f(\xi)=\ES f(X_\h^\xi)=f(\xi)+\int_0^\h \ES [(\nabla f(X_s^\xi)| b(X_s^\xi))] ds+\frac{1}{2}\int_0^\h \ES\,[{\rm Tr}(D^2 f(X_s^\xi) \sigma\sigma^* (X_s^\xi))] ds.$$
Again by Itô formula, for every $i\in\{1,\ldots,d\}$,
\begin{align*}
\ES [\partial_i f (X_s^\xi) b_i(X_s^\xi))]&= \partial_i f(\xi) b_i(\xi)+\int_0^s \ES\,[(\nabla(\partial_i f b_i)(X_u^\xi)|b(X_u^\xi))+\frac{1}{2}{\rm Tr}(D^2 (\partial_i f b_i)\sigma\sigma^*)(X_u^\xi)] du,
\end{align*}
and for every $i,j\in\{1,\ldots,d\}$,
\begin{align*}
\ES\,[(D^2 f(X_s^\xi) \sigma\sigma*)_{ii} (X_s^\xi)]=(D^2 f\sigma\sigma^*)_{ii} (\xi)+\int_0^s \ES\,[{\cal L} ((D^2 f\sigma\sigma^*)_{ii} )(X_u^\xi)] du.
\end{align*}
Thus, 
\begin{equation}\label{eq:pgammaf}
P_\h f(\xi)=\ES f(X_\h^\xi)=f(\xi)+\h{\cal L} f(\xi)+\int_0^\h\int_0^s \sum_{k=1}^{4} \sum_{|\alpha|=k}\ES\,[ \partial_\alpha f (X_u^\xi) \phi_\alpha(X_u^\xi))] ds,
\end{equation}
where for any $k$, the functions $\phi_\alpha$ are polynomial functions (which may be made explicit) of $b$, $\sigma$ and their partial derivatives up, respectively, to order  $2$.
%
Now, for the  Euler scheme, let us introduce, for a positive  $M$, a smooth and {radial}
function $\mathfrak{T}_M:\ER_d\rightarrow\ER_+$ equal to $1$ on $B(0,M)$ and $0$ on $B(0,2M)^c$ and such that the derivatives of $\mathfrak{T}_M$ are uniformly bounded. Then, 
\begin{align*}
\bar{P}_\h f(\xi)&= \ES\,[f(\bar{X}_\h^\xi)\mathfrak{T}_M(W_\h)]+{r_{h,M}(f),\quad \textnormal{with}\quad |r_{h,M}(f)|\le \|f\|_\infty \PE(|W_\h|> M).}
\end{align*}
Note that the rotation-invariance combined with the independence of the coordinates of the Brownian motion implies that for any  $(a_1,\ldots,a_d)\in \mathbb{N}^d$ with at least one odd integer,
$\ES\,[(W_h^1)^{a_1}\ldots(W_h^d)^{a_d}\mathfrak{T}_M(W_\h)]=0$. Setting $By^{\otimes \ell}=\sum_{i_1,\ldots,i_\ell} (B_{i_1,\ldots,i_\ell} )y_{i_1} \ldots y_\ell$ for a an element $B$ of $(\ER^d)^{\ell}$, we deduce that
$$ \ES\,[D^2f(\xi)(\sigma(\xi) W_h)^{\otimes 2}\mathfrak{T}_M(W_\h)]=a(M,h){\rm Tr}(D^2f \sigma\sigma^*)(\xi)\quad\textnormal{and}\quad \ES\,[D^3f(\xi)(\sigma(\xi) W_h)^{\otimes 3}\mathfrak{T}_M(W_\h)]=0,$$
where 
$$a(M,h)=\ES\,[(W_h^1)^2\mathfrak{T}_M(W_h)].$$
Then, it follows from the Taylor formula applied to $f(\bar{X}_\h^\xi)$ that
\begin{align}
\ES\,[f(\bar{X}_\h^\xi)&\mathfrak{T}_M(W_\h)]= \ES\,\mathfrak{T}_M(W_\h)\big(f(\xi)+\h (\nabla f(\xi)|b(\xi))\big)+a(M,h){\rm Tr}(D^2f \sigma\sigma^*)(\xi)\label{eq:taylorformulabis}\\
&+\h^2 \ES\,[\mathfrak{T}_M(W_\h)]\Big(\overbrace{\frac{1}{2} (D^2 f(\xi)b(\xi)|b(\xi))+ \h\frac{1}{6}\sum_{i,j,k}^3{\partial_{i,j,k}^3} f(\xi) \left(b_i(b_jb_k+(\sigma\sigma^*)_{jk})\right)(\xi)}^{\varphi_{\h}^{(1)}(\xi)}\Big)\nonumber\\
&+\frac{1}{24}\int_0^1\ES\,\big[\underbrace{D^4 f\left(\xi+\theta(\h b(\xi)+\sigma(\xi) W_\h)\right)(\h b(\xi)+\sigma(\xi) W_\h)^{\otimes 4}\mathfrak{T}_M(W_\h)}_{\varphi_{\h,M}^{(2)}(\xi,\theta,W_h)}\big]d\theta.\nonumber
 \end{align}
Thus, noting that $1-\ES\,[\mathfrak{T}_M(W_h)]\le \PE (|W_h|>M)$, we deduce from what precedes and from~\eqref{eq:pgammaf}, we get
\begin{align*}
\ES\,[f(X_\h^\xi)]&-\ES\,[f(\bX_\h^\xi)]= \varphi_{\h,M}(\xi)\\
\mbox{ where }\qquad\varphi_{\h,M}(\xi)&={r_{h,M}(f)}+O(\h \PE(|W_\h|> M)) (\nabla f(\xi)|b(\xi))\\
&\quad +\frac{1}{2}(h-a(M,h)){\rm Tr}(D^2f \sigma\sigma^*)(\xi)-\h^2 \ES\,[\mathfrak{T}_M(W_\h)]{\varphi_{\h}^{(1)}(\xi)}\\
&\quad +\int_0^\h\int_0^s \sum_{k=1}^{4} \sum_{|\alpha|=k}{\ES\,[ \partial_\alpha^{k} f (X_u^\xi) \phi_\alpha(X_u^\xi))]}ds-\frac{1}{24}\int_0^1\ES\,[{\varphi_{\h,M}^{(2)}(\xi,\theta,W_h)}]d\theta.\qquad\qquad\quad
\end{align*}
\textit{Step 2:} Assume now that $\xi=\bar{X}_{t_{n-1}}$, the Euler scheme at time $t_{n-1}$ related to the step sequence
$(h_n:=t_n-t_{n-1})_{n\ge1}$ starting from $x\in\ER^d$. Let $\sigma_{\bar{X}_{t_{n-1}}}$ denote the Malliavin matrix of
$\bar{X}_{t_{n-1}}$ (whose definition is recalled  in Equation~\eqref{eq:defmalmat}). {For  $\eta\!\in (0,1]$}, let $\Psi_{\eta}$ denote a smooth function on $\ER$ such that $\Psi_\eta(x)=0$ on $(-\infty,\eta/2)$ and $1$ on $(\eta,+\infty)$. We can furthermore assume that for every integer $\ell$,  $\|\Psi_\eta^{(\ell)}\|_\infty\le C\eta^{-\ell}$ where $C$ is a universal constant. Using that $W_{t_n}-W_{t_{n-1}}$ is independent from $\bar{X}_{t_{n-1}}$ and that $0\le 1-\Psi_\eta(u)\le 1_{\{u\le\eta\}}$,
\begin{align*}
|\bar  P_{\h_1}\circ \cdots\circ \bar  P_{\h_{n-1}}\circ(P_{\h_n} - \bar P_{\h_n}) \circ   f(x)|&\le
2\|f\|_\infty\PE\big({{\rm det} \, \sigma_{{{\bar{X}_{t_{n-1}}}}}\le \eta}\big) +\Big|\ES\lt[\varphi_{\h_{n},M}({{\bar{X}_{t_{n-1}}}})\Psi_\eta({\rm det} \, \sigma_{{{\bar{X}_{t_{n-1}}}}})\rt]\Big|.
\end{align*}
Let us  denote the  unique solution at time $u$ starting from $x$ of~\eqref{eds:intro} by $\mathfrak{X}(u,x)$ ($(u,x)\mapsto \mathfrak{X}(u,x)$ is the stochastic flow related to~\eqref{eds:intro}). 
Note that $\PE(|W_\h|>M)=O(e^{-\frac{M^2}{4\h}})$ and that 
$$0\le h-a(M,h)=\ES\,[(W_h^1)^2(1-\mathfrak{T}_M(W_h)]\le \ES\,[(W_h^1)^2 1_{|W_h|>M}]\le C h e^{-\frac{M^2}{8h}},$$
by  Cauchy-Schwarz and (exponential) Markov inequalities.
Then, using the expansion of $\varphi_{\h,M}$ obtained at the end of Step $1$, we get,
\begin{align}
 \Big| \bar  P_{\h_1}\circ \cdots\circ& \bar  P_{\h_{n-1}}\circ(P_{\h_n} - \bar P_{\h_n}) \circ   f(x)\Big|
 \le {2}\|f\|_\infty\lt(e^{-\frac{M^2}{4\h_n}}+ \PE ({\rm det} \, \sigma_{{{\bar{X}_{t_{n-1}}}}}\le \eta)\rt)\label{eq:termcontinu332}\\
&+O\big(\h_n e^{-\frac{M^2}{4\h_n}})\left|\ES\,[(\nabla f|b)({{\bar{X}_{t_{n-1}}}})\Psi_\eta({\rm det} \, \sigma_{\bar{X}_{t_{n-1}}})]\right|\label{eq:termcontinu333}\\
&+O\big(\h_n e^{-\frac{M^2}{8\h_n}}\big)\left|\ES\,[{\rm Tr}(D^2f\sigma\sigma^*)({{\bar{X}_{t_{n-1}}}})\Psi_\eta({\rm det} \, \sigma_{\bar{X}_{t_{n-1}}})]\right|
\label{eq:termcontinu333bis}\\
&+O(\h_n^2 \big)\left| \ES\,[\varphi_{\h_n}^{(1)}({\bar{X}_{t_{n-1}}})\Psi_\eta({\rm det} \, \sigma_{{\bar{X}_{t_{n-1}}}})]\right|\label{eq:termcontinu335}\\
 &+\int_0^\h\int_0^s \sum_{k=1}^{4}\sum_{|\alpha|=k}\left| \ES\,[\partial_\alpha f(\mathfrak{X}(u,{\bar{X}_{t_{n-1}}}))\phi_\alpha(\mathfrak{X}(u,{\bar{X}_{t_{n-1}}}))\Psi_\eta({\rm det} \, \sigma_{{\bar{X}_{t_{n-1}}}})] \right|ds\label{eq:termcontinu334}\\
&+\frac{1}{24}\int_0^1\left| \ES\,[\varphi^{(2)}_{h_n,M}({\bar{X}_{t_{n-1}}},\theta,W_{t_n}-W_{t_{n-1}})\Psi_\eta({\rm det} \, \sigma_{{\bar{X}_{t_{n-1}}}})]\right| d\theta.\label{eq:termcontinu336}
\end{align}
Let us now consider all the above terms separately. We begin by the first term related to the probability of ``degeneracy'' of  $ \sigma_{{{\bar{X}_{t_{n-1}}}}}$. By Proposition~\ref{lem:mallderivsuc}$(i)$ applied with $r=2$ a given positive $T$, we know that if $T/2\le t_{n-1}\le T$, we have for every $p>0$,
$$
|\eqref{eq:termcontinu332}|\le C\|f\|_\infty \left(e^{-\frac{M^2}{4\h_n}}+h_1^2+\eta^p\right)\le C\|f\|_\infty \left(h_1^2+\eta^p\right),
$$
where in the second inequality, we used that $e^{-\frac{M^2}{x}}\le C_M x^{2}$ for $x\in(0,1]$.
For~\eqref{eq:termcontinu333} and~\eqref{eq:termcontinu333bis}, we use Lemma~\ref{lem:gestiondestermesss}$(i)$ with $F=\bar{X}_{t_{n-1}}$. First, note that, owing to Proposition~\ref{lem:mallderivsuc}$(ii)$ and {to the fact that $b$ and $\sigma$ are ${\cal C}^6$}, Assumption~\eqref{control:momentmalliav22} of this lemma holds true {with $k\le 4$}. Then, one remarks that it is enough to apply Lemma~\ref{lem:gestiondestermesss}$(i)$ with $|\alpha|=1$ and $G=b_i(F)$ ($i=1,\ldots,d$) for~\eqref{eq:termcontinu333},  and,
$|\alpha|=2$ and $G=\sigma_{i,j}\sigma_{k,i}(F)$, $(i,k)\in\{1,\ldots,d\}$ for~\eqref{eq:termcontinu333bis}. Since $b_i$ has linear growth and bounded derivatives, it follows from Proposition~\ref{lem:mallderivsuc}$(ii)$ that
$\|b_i(\bar{X}_{t_{n-1}})\|_{1,3}\le C(1+\ES\,[|\bar{X}_{t_{n-1}}|^3]^{\frac{1}{3}})$ whereas, since $\sigma$ and its derivatives are bounded,
$\|\sigma_{i,j}\sigma_{k,i}(\bar{X}_{t_{n-1}})\|_{2,3}\le C$, where $C$ does not depend on $n$.
By Lemma~\ref{lem:gestiondestermesss}$(i)$ and a Gronwall argument, it follows that a constant $C$ exists (depending on $T$) such that
\begin{align*}
|\eqref{eq:termcontinu333}|+|\eqref{eq:termcontinu333bis}|&\le C\h_n e^{-\frac{M^2}{8\h_n}} \|f\|_\infty \eta^{-{4}} (1+\ES_x[|\bar{X}_{t_{n-1}}|^{{6}}]^{\frac{1}{3}})(1+\ES_x[|\bar{X}_{t_{n-1}}|^3]^{\frac{1}{3}}).\\
&\le C \h_1^2 \|f\|_\infty \eta^{-4} (1+|x|^{{3}}).
\end{align*}
For~\eqref{eq:termcontinu335}, this is a direct application of Lemma~\ref{lem:gestiondestermesss}$(ii)$ combined with 
Proposition~\ref{lem:mallderivsuc} $(ii)$. This leads to
$$
|\eqref{eq:termcontinu335}|\le C h_n^2 \eta^{-6} (1+\ES_x[|\bar{X}_{t_{n-1}}|^9]^{\frac{2}{3}})\le C h_1^2 \eta^{-6} (1+|x|^6).
$$
For  any $\alpha$ involved in~\eqref{eq:termcontinu334}, we can apply Lemma~\ref{lem:gestiondestermesss}$(iii)$ 
with $F=\bar{X}_{t_{n-1}}$ and $\phi=\phi_\alpha$. Looking carefully into the definition of $\phi_\alpha$, one can check that 
for any $\alpha$, for any $\ell\in\{0,\ldots,|\alpha|\}$, $|\phi_\alpha^{(\ell)}(x)|\le C(1+|x|^2)$.
Thus, taking the worst case $|\alpha|=4$ in Lemma~\ref{lem:gestiondestermesss}$(iii)$, we get:
$$
 |\eqref{eq:termcontinu334}|\le C h_n^2 \|f\|_\infty  {\eta^{-12}}\big(1+\ES\,[|\bar{X}_{t_{n-1}}|^{24}]\big)^{\frac{1}{6}}\big(1+\ES\,[|\bar{X}_{t_{n-1}}|^{24}]\big)^{\frac{1}{12}}\le Ch_1^2\|f\|_\infty  {\eta^{-12}}(1+|x|^6).
 $$
Finally, the control of~\eqref{eq:termcontinu336} relies on Lemma~\ref{lem:gestiondestermesss}$(iv)$ with $F=\bar{X}_{t_{n-1}}$. Once again, this statement holds true by Proposition ~\ref{lem:mallderivsuc}$(ii)$. We have
$$
|\eqref{eq:termcontinu336}|\le C\|f\|_\infty h_n^2 \eta^{-12}
(1+\ES_x[|\bar{X}_{t_{n-1}}|^{24}])^{\frac{1}{3}}\le C\|f\|_\infty h_{1}^2 \eta^{-12}(1+|x|^8)
$$
by using again that $\ES\,[|\bar{X}_{t_{n-1}}|^p]\le C(1+|x|^p)$.

\smallskip
\noindent Combining all the above controls, we deduce that  there exists $\bar{h}>0$ and $T>0$ such that if $T/2\le t_{n-1}\le T$, then,
\begin{align*}
\Big| \bar  P_{\h_1}\circ \cdots\circ \bar  P_{\h_{n-1}}\circ(P_{\h_n} - \bar P_{\h_n}) \circ   f(x)\Big|&
 \le C  \|f\|_\infty (\eta^p+ h_1^2 \eta^{-12} (1+|x|^8)).
\end{align*}
For a given $\varepsilon>0$, it is now enough to fix {$\eta=h_1^{\frac{\varepsilon}{12}}$ and $p=24\varepsilon^{-1}$} to conclude the proof.
%

\subsection{Malliavin bounds}\label{subsec:malliavinbounds}
In this section, we detail the arguments which lead to the controls of the terms~\eqref{eq:termcontinu332} to~\eqref{eq:termcontinu336} involved in the decomposition of $\Big| \bar  P_{\h_1}\circ \cdots\circ \bar  P_{\h_{n-1}}\circ(P_{\h_n} - \bar P_{\h_n}) \circ   f(x)\Big|$. All these terms are managed with the help of Malliavin-type arguments.

\smallskip
\noindent Without going into the details (for this, see $e.g.$~\cite{nualart}), let us recall some basic notations of Malliavin calculus on Wiener space. We set ${\cal H}=L^2(\ER_+,\ER^d)$ and denote by $W=\{W(h),h\in{\cal H}\}$, an isonormal Gaussian process on ${\cal H}$ which is assumed to be defined on a complete filtered probability space $(\Omega,{\cal F}, \PE)$, and that ${\cal F}$ is generated by $W$. We also denote by $({\cal F}_t)_{t\ge0}$ the completed natural filtration of $(W_t)_{t\ge0}$. 

\smallskip
The Malliavin operator is denoted by $D$ and its domain by $\mathbb{D}^{1,p}$ for a given $p>1$ (closure of  ${\cal S}$, space of \textit{smooth random variables}, in $L^p(\Omega)$ for the norm $\|\,.\,\|_{1,p}$ defined in~\eqref{eq:seminormkp}). For a (${\cal F}$-measurable) random variable $F$ in $\mathbb{D}^{p,1}$, $DF$ is a random variable with values in ${\cal H}$ such that $\ES\,[\|DF\|_{{\cal H}}^p]<+\infty$. For every multi-index $\alpha\in\{1,\ldots,d\}^k\}$, the iterated derivative $D^\alpha F$ is defined on ${\cal H}^{\otimes k}$. The space $\mathbb{D}^{k,p}$ denotes the closure of ${\cal S}$ in $L^p(\Omega)$ for the norm $\|\,.\,\|_{k,p}$ defined for a given real-valued random variable $F$ by
\begin{equation}
\label{eq:seminormkp}
\| F\|_{k,p}=\ES\,[|F|_k^p]^{\frac{1}{p}}\quad \textnormal{with}\quad |F|_k=|F|+|F|_{k\setminus 0},
\quad \textnormal{where}\quad |F|_{k\setminus 0}= \sum_{\ell=1}^k \| D^{(\ell)} F\|_{{\cal H}^{\otimes \ell},}
\end{equation}
\noindent and for every $\ell\ge 1$, 
$$
\|D^{(\ell)} F\|_{{\cal H}^{\otimes \ell}}^2:=\sum_{|\alpha|=\ell}\int_{[0,+\infty)^\ell}|D^{\alpha}_{s_1 \ldots s_\ell} F|^2 ds_1 \ldots ds_\ell.
$$
For a random variable $F=(F^1,\ldots,F^m)$, $|F|_{k\setminus 0}=\sum_{i=1}^m |F_i|_{k\setminus 0}$, $F|_k=\sum_{i=1}^m |F_i|$ and $\|F\|_{k,p}^p= \sum_{i=1}^m \ES\,[|F_i|_k^p]$. Furthermore, for such $\ER^m$-valued Malliavin-differentiable random variable $F$,  the Malliavin matrix, denoted by $\sigma_F$, is defined by
\begin{equation}\label{eq:defmalmat}
\sigma_F=(\langle DF^i, DF^j\rangle_{\cal H})_{1\le i,j\le m}.
\end{equation}
For any element $A$ of $\ER^{m^{\ell+1}}$, we will denote by $\|\,.\,\|$ the $L^2$-norm defined by
\begin{equation}\label{eq:frobenius}
\|A\|=\sqrt{\sum_{1\le i_1,\ldots i_{\ell+1}\le d} |A_{i_1,\ldots,i_{\ell+1}}|^2}.
\end{equation}
Note that when $\ell=1$, this corresponds to the Frobenius norm on the space of $m\times m$ matrices.
\subsubsection{Bounds for a general random variable $F$}
In this subsection, we consider an ${\cal F}_{\Tun}$-measurable random variable $F$ and establish  some useful bounds under appropriate Malliavin assumptions. Then, since in the proof of  Theorem~\ref{thm:MalliavinWeakEr}, we will use them with $F=\bar{X}_{t_{n-1}}$, we will prove in the next subsection that the assumptions of the results of this section hold true.

\smallskip
\noindent In the following lemma, we recall that  $\mathfrak{X}(u,x)$ is the  unique solution at time $u$ starting from $x$  (more precisely, $(u,x)\mapsto \mathfrak{X}(u,x)$ is the stochastic flow related to~\eqref{eds:intro}). Furthermore, we  implicitly assume that  if $F$ is an ${\cal F}_{\Tun}$-measurable random variable,  $\mathfrak{X}(u,x)$ is built with the increments of $W_{\Tun+.}-W_{\Tun}$. In particular, $\mathfrak{X}(u,x)$ is viewed as an ${\cal F}_{\Tun+u}$ random variable.
\begin{lem}\label{lem:contrmallideriv} Let $\Tun>0$ Let $F$ denote an $\ER^d$-valued ${\cal F}_{\Tun}$-measurable Malliavin-differentiable random variable. Assume $b$ and $\sigma$ have bounded first partial derivatives. Then,

\smallskip
\noindent $(i)$  For every $p\ge1$ and $\eta>0$, a constant $\mathfrak{C}$ exists (which does not depend on $F$) such that
$$
\sup_{u\in[0,1]}\ES\,[|{\rm det}(\sigma_{\mathfrak{X}(u,F)})|^{-p}\mbox{\bf 1}_{\{{\rm det}\,\sigma_{F}\ge\eta\}}]\le
\mathfrak{C}\eta^{-p}.
$$
\noindent  $(ii)$ Set $\bar{\mathfrak{X}}_\theta(u,x)=x+\theta \big(u b(x)+\sigma(x)(W_{\Tun+u}-W_{\Tun})\big)$. Then, some positive $M$, $\bar{h}$ and $\mathfrak{C}$ exist such that for every $u\in(0,\bar{h}]$ and $\theta\in(0,1]$,
$$
\ES\,[|{\rm det}(\sigma_{\bar{\mathfrak{X}}_\theta(u,F)})|^{-p}\mbox{\bf 1}_{\{{\rm det}\,\sigma_{F}\ge\eta,|W_{\Tun+u}-W_{\Tun}|\le M\}}]\le
\mathfrak{C} \eta^{-p}.
$$
%
%
\end{lem}
\begin{rem} In $(i)$, we state that on the set where $\sigma_F$ is not degenerated, nor is $\mathfrak{X}(u,F)$  (with a non-degeneracy which is quantified along the parameter $\eta$). In $(ii)$, we show that for the Euler scheme, this property is still true but up to a truncation of the Brownian increments (by $M$). Here, one retrieves that unfortunately, the Malliavin matrix of the Euler scheme is not invertible everywhere (see  Proposition~\ref{lem:mallderivsuc}$(i)$ for a control of the lack of invertibility of  $\sigma_{\bar{X}_{t_n}}$).
\end{rem}
\noindent {\it Proof.}  $(i)$ As mentioned before the lemma, we implicitly assume that   $\mathfrak{X}(u,x)$ is built with the increments of $W_{\Tun+.}-W_{\Tun}$. Thus $\mathfrak{X}(u, F)$ is a functional of $(W_s,0\le s\le T+u)$. Then, owing to the chain rule for Malliavin calculus, we remark that for any $s\in[0,T]$, for any  $i$ and $j$ $\in\{1,\ldots,d\}$,
\begin{equation}\label{malliav-deriv-xibar}
D_s^j \mathfrak{X}_i(u, F)= \sum_{\ell=1}^d Y_u^{i\ell, F} D_s^j F^{\ell},
\end{equation}
where we recall that $Y_u^{i\ell,x}=\partial_{x_\ell} X^{i,x}_u$ (where $X^{i,x}$ stands for the $i$th coordinate of $X_u^x$).
It follows that
$$
\sigma_{ \mathfrak{X}(u, F)}=\int_0^{t_n} Y_u D_s F (Y_u D_s F)^* ds+\int_{t_n}^u
D_s \mathfrak{X}(u, F)(D_s \mathfrak{X}(u, F))^* ds.
$$
Since for two symmetric positive matrices $A$ and $B$, ${\rm det}(A+B)\ge \max({\rm det}\, A,{\rm det}\, B)$, we deduce that
$$
{\rm det}\, \sigma_{ \mathfrak{X}(u,F)}\ge |{\rm det}(Y_u^{F})|^2{\rm det}\left(\int_0^{t_n} D_s F ( D_s F)^* ds\right)=|{\rm det}(Y_u^{F})|^2{\rm det}(\sigma_{F}).
$$
Thus, 
$$
\ES\,[{\rm det} \, \sigma_{ \mathfrak{X}(u, F)}\mbox{\bf 1}_{\{{\rm det}\,\sigma_{F\ge\eta}\}}]\le
\eta^{-p}\ES\,[|{\rm det}(Y_u^{F})|^{-2p}]\le C_p\eta^{-p},
$$
where $C_p=\sup_{x\in\ER^d,u\in[0,1]} \ES\,[|{\rm det}(Y_u^{x})|^{-2p}]<+\infty$ (the fact that $Y_0^x=I_d$ and that $\nabla b$ and $\nabla \sigma$ are bounded implies that $C_p$ is finite, with the help of  a Gronwall argument, similar to the one used in $(iii)$).

\smallskip
\noindent $(iii)$ The map $x\mapsto \bar{\mathfrak{X}}_\theta(u,x)$ is differentiable on $\ER^d$.
Then, owing to the chain rule for Malliavin calculus, for every $j\in\{1,\ldots,d\}$,
$$
D_s^j \bar{\mathfrak{X}}_\theta(u,F)= \nabla_x \bar{\mathfrak{X}}_\theta(u,F)\circ D_s^j F,
$$
and with the same arguments as in $(i)$,
$$
{\rm det}\,\sigma_{ \bar{\mathfrak{X}}_\theta(u,F)}\ge |{\rm det}(\nabla_x \bar{\mathfrak{X}}_\theta(u,F))|^2
({\rm det}\,\sigma_{F}).
$$
Now,
$$
\nabla_x \bar{\mathfrak{X}}_\theta(u,x)={\rm I_d}+\theta( u \nabla b(x)+\nabla \sigma(x)(W_{\Tun+u}-W_{\Tun})),$$
and one  checks that
$$
\|\theta(u \nabla b(x)+\nabla \sigma(x)(W_{\Tun+u}-W_{\Tun}))\|_F\le u d\|\nabla b\|_\infty+d\|\nabla \sigma\|_\infty |W_{\Tun+u}-W_{\Tun}|.
$$
Thus, setting 
$$
M=\frac{1}{4d(\|\nabla\sigma\|_\infty\wedge 1)}\quad\textnormal{and}\quad \bar{h}=\frac{1}{4d\|\nabla b\|_\infty},$$
we conclude the proof by noting that, on the event $\big\{{\rm det}\,\sigma_{F}\ge\eta\big\}$, 
$$
\hskip3cm \inf_{h\in(0,\bar{h}]} {\rm det} \,\sigma_{ \bar{\mathfrak{X}}_\theta(u,F)} \ge 2^{-2d} \eta.\hskip3cm
$$

\begin{lem}\label{lem:gestiondestermesss} 
Let $k$ be a positive integer. Assume that $|b(x)|\le C(1+|x|)$ and that $\sigma$ is bounded. Let $\Tun>0$. Let $F$ be an $\ER^d$-valued ${\cal F}_{\Tun}$-measurable random variable, Malliavin-differentiable up to order $k+2$, such that for every $p\ge1$,
\begin{equation}\label{control:momentmalliav22}
\sup_{1\le \ell\le k+2} \sup_{|\alpha|=\ell} \sup_{s_1,\ldots,s_\ell\in[0,\Tun]} \left(\ES\,[ \|\Der_{s_1,\ldots,s_\ell}^{\alpha} F\|^{p}]\right)^{\frac{1}{p}}=: \mathfrak{d}_{p,\Tun}^{(k+2)}<+\infty.
\end{equation}
Let $\Psi_{\eta}$ denote a smooth function on $\ER$ such that $\Psi_\eta(x)=0$ on $(-\infty,\eta/2)$ and $1$ on $(\eta,+\infty)$. 
Then, some positive $C$ and $M$  exist such that for any $\eta>0$ 

\smallskip
\noindent $(i)$ 
For any $\alpha\in\{1,\ldots,d\}^k$ with $|\alpha|=k$, for any $G$ in $\mathbb{D}^{k,3}$, 
$$
|\ES\,[\partial_\alpha f(F)G\Psi_\eta({\rm det}\, \sigma_F)]|\le C \|f\|_\infty \eta^{-2k} \big(1+\ES\,[|F|^{3k}]^{\frac{1}{3}}\big)\|G\|_{k,3}.
$$

\noindent $(ii)$ Assume that $b$ and $\sigma$ are ${\cal C}^{{3}}$ with bounded existing partial derivatives.
$$
\left| \ES\,[\varphi_{\h}^{(1)}(F)\Psi_\eta({\rm det}\, \sigma_{F})]\right|\le C\eta^{-6} \big(1+\ES\,[|F|^9]^{\frac{2}{3}}\big)
$$
where $C$ depends on $b$, $\sigma$,  $T$ and $\mathfrak{d}_{p,\mathfrak{t}}^{(3)}$ for a given $p$ (which could be made explicit).

\smallskip
\noindent {$(iii)$   Assume that $b$ and $\sigma$ are ${\cal C}^{k+2}$ with bounded existing partial derivatives. Let $\phi:\ER^d\mapsto\ER$ denote a ${\cal C}^{k}$-function such that $|\phi(x)|+\sum_{\ell=1}^k \|\nabla^{(\ell)} \phi(x)\|\le C(1+|x|^2)$. Then, for any $\alpha\in\{1,\ldots,d\}^{k}$, }
$$
\sup_{u\in[0,\tau]}\left| \ES\,[\partial_\alpha f(\mathfrak{X}(u,F))\phi(\mathfrak{X}(u,F))\Psi_\eta({\rm det}\, \sigma_{F})] \right|
\le C \|f\|_\infty \eta^{-3k}\big(1+\ES\,[|F|^{6k}]^{\frac{1}{6}}\big)\big(1+\ES\,[|F|^{24}]\big)^{\frac{1}{12}}.
$$
$(iv)$ Let $Z$ be a ${\cal N}(0,I_d)$-random variable independent of $F$. For any $\theta\in[0,1]$,
$$
\left|\ES\,[\varphi^{(2)}_{h,M}(F,\theta,\sqrt{h} Z)\Psi_\eta({\rm det} \, \sigma_{F})]\right|\le C\|f\|_\infty h^2 \eta^{-12}
\big(1+\ES\,[|F|^{24}]\big)^{\frac{1}{3}}.
$$  
\end{lem}
\begin{proof}
The proof strongly relies on \cite[Lemmas 2.3, 2.4]{bally_caramellino_poly}.

\noindent \textit{(i)} Let $\alpha$ denote a multi-index. By  Lemmas 2.3 and 2.4$(ii)$ of~\cite{bally_caramellino_poly} (applied with $k=0$ and $n=|\alpha|$),
\begin{align*}
\ES\,[\partial_\alpha  f({\FF}) G\Psi_\eta({\rm det }\sigma_{{\FF}})]&= \ES\,[f({\FF}) H_\alpha({\FF},G\Psi_\eta({\rm det }\sigma_{{\FF}}))],
\end{align*}
where for some random variables $F$ and ${G}$ in $\mathbb{D}^{|\alpha|,p}$,
$$
|H_\alpha(\FF,G\Psi_\eta({\rm det }\sigma_{{\FF}}))|\le C \eta^{-2|\alpha|}\left(|F|_{\alplunmoinszero}+|LF|_{|\alpha|}\right)^{|\alpha|}(1+|F|_{\alplunmoinszero})^{4d|\alpha|}|G|_{|\alpha|}.
$$
where $|F|_{\kmoinszero}$ is defined in~\eqref{eq:seminormkp}  and $L$ denotes the Ornstein-Uhlenbeck operator. Thus, using Hölder inequality, we deduce that
\begin{align*} 
 \ES\,[|H_\alpha(\FF,G\Psi_\eta({\rm det }\sigma_{{\FF}}))|]&\le C \eta^{-2|\alpha|}\ES\,[\left(|F|_{\alplunmoinszero}+|LF|_{|\alpha|}\right)^{3|\alpha|}]^{\frac{1}{3}}\ES\,[(1+|F|_{\alplunmoinszero})^{12d|\alpha|}]^{\frac{1}{3}} \|G\|_{|\alpha|,3}.
 \end{align*}
Now, on the one hand, by the definition of $|F|_{\kmoinszero}$ and Assumption~\eqref{control:momentmalliav22}, one easily checks that for every positive integer $k$ and positive $p\ge1$,
\begin{equation}\label{eq:malliavcontrol11}
\ES\,[|F|_{\kmoinszero}^p]^{\frac{1}{p}}\le  C_{p,\Tun} \mathfrak{d}_{p,\Tun}^{(k)}.
\end{equation}
On the other hand,  the term involving the Ornstein-Uhlenbeck operator $L$ can be classically controlled by Meyer inequalities (see $e.g.$ \cite[Theorem 1.5.1]{nualart} or \cite[Section 2.4]{bally_caramellino}), which ensures for every integer $m$ and positive $p$, the existence of a constant $C_{m,p}$ such that 
\begin{equation}\label{eq:malliavcontrol12}
\| L F\|_{m,p}\le C_{m,p} \| F\|_{m+2,p}\le C_{m,p,\Tun}\left(\ES\,[|F|^p]^{\frac{1}{p}}+\mathfrak{d}_{p,\Tun}^{(m+2)}\right).
\end{equation}
 where $\|\,,\,\|_{k,p}$ is defined by~\eqref{eq:seminormkp}.  
 Thus, by the Minkowski inequality, we deduce that a constant $C$ exists depending on $T$, $|\alpha|$ and 
 $\mathfrak{d}_{12 d|\alpha|,\Tun}^{(|\alpha|+2)}$ such that,
 \begin{align*} 
 \ES\,[|H_\alpha(\FF,G\Psi_\eta({\rm det }\sigma_{{\FF}}))|]&\le C \eta^{-2|\alpha|}\big(1+\ES\,[|F|^{3|\alpha|}]^{\frac{1}{3}}\big)\|G\|_{|\alpha|,3}.
 \end{align*}
$(ii)$ We have to apply $(i)$ for some multi-indices $\alpha$ with $|\alpha|=2$ or $|\alpha|=3$. 
More precisely, on the one hand, the first term of $\varphi_{\h}^{(1)}(F)$ can be written as follows:
$$ (D^2 f (F) b(F)|b(F))=\sum_{i,j} \partial^{2}_{i,j} f(F) G_{i,j}\quad \textnormal{with}\quad G_{i,j}=b_i(F)b_j(F).$$
Thus, since $|b(x)|\le C(1+|x|)$ and $b$ has bounded derivatives, one checks (using the chain rule for Malliavin calculus and~\eqref{eq:malliavcontrol11}) that,
$$\|G_{i,j}\|_{2,3}\le C (1+\ES\,[|F|^6]^{\frac{1}{3}}),$$
where $C$ depends on $T$ and $\mathfrak{d}_{p,\Tun}^{(4)}$ with $p=6$. Thus, it follows from $(i)$ (applied with $|\alpha|=2$) that
$$\left| \ES\,[(D^2 f (F) b(F)|b(F))\Psi_\eta({\rm det} \, \sigma_{F})]\right|\le C \|f\|_\infty \eta^{-4} (1+\ES\,[|F|^{6}]^{\frac{2}{3}})$$
where $C$ depends on $T$ and $\mathfrak{d}_{p,\Tun}^{(|\alpha|+2)}$ with $p=24d$.
On the other hand, the second term of $\varphi_{\h}^{(1)}(F)$ has the following form:
$$\frac{1}{6}\sum_{i,j,k}^3{\partial_{i,j,k}^3} f(F) G_{i,j,k}\quad\textnormal{with}\quad 
G_{i,j,k}=\left(b_i(b_jb_k+(\sigma\sigma^*)_{jk})\right)(F).$$
Using the assumptions on $b$ and $\sigma$, one checks (using the chain rule for Malliavin calculus and~\eqref{eq:malliavcontrol11}) that,
$$\|G_{i,j,k}\|_{3,3}\le C (1+\ES\,[|F|^9]^{\frac{1}{3}}),$$
where $C$ depends on $T$ and $\mathfrak{d}_{p,\Tun}^{(5)}$ with $p=9$.
Thus, it follows from $(ii)$ (applied with $|\alpha|=3$) that for every $(i,j,k)\in\{1,\ldots,d\}^3$,
$$\ES\,[{\partial_{i,j,k}^3} f(F) G_{i,j,k}\Psi_\eta({\rm det} \, \sigma_{F})]\le C\eta^{-6} (1+\ES\,[|F|^9]^{\frac{2}{3}}),$$
where, once again, $C$ depends on $T$ and $\mathfrak{d}_{p,\Tun}^{(|\alpha|+2)}$ for a given value of $p$ ($p=36d$).
The result follows.

\smallskip
\noindent $(iii)$ For this statement and the following, we use Lemma 2.4$(i)$ of~\cite{bally_caramellino_poly}, which states that 
for some random variables $\bar{F}$ and ${G}$ in $\mathbb{D}^{|\alpha|,p}$,
\begin{align}\label{eq:formulanondegenerate}
\ES\,[\partial_\alpha  f({\bar{F}}) G]&= \ES\,[f({\bar{F}}) H_\alpha({\bar{F}},G)],
\end{align}
where, on the set ${\rm det} \, \sigma_{\bar{F}}>0$,
$$
|H_\alpha(\bar{F},G)|\le C \left(\frac{|{\bar{F}}|_{\alplunmoinszero}^{2(d-1)}(|{\bar{F}}|_{\alplunmoinszero}+|L{\bar{F}}|_{|\alpha|})}{{\rm det }\,\sigma_{\bar{F}}}\right)^{|\alpha|} \times\sum_{p_1+p_2\le |\alpha|}|G|_{p_2}\left(1+\frac{|{\bar{F}}|_{\alplunmoinszero}^{2d}}{{\rm det} \, \sigma_{\bar{F}}}\right)^{p_1}.
$$
It follows that on the set $\{{\rm det} \, \sigma_{\bar{F}}>0\}$,
$$
|H_\alpha(\bar{F},G)|\le C \overbrace{\left({1+{\rm det }\,\sigma_{\bar{F}}+|{\bar{F}}|_{\alplunmoinszero}^{2d}}\right)^{|\alpha|}
\left((|{\bar{F}}|_{\alplunmoinszero}+|L{\bar{F}}|_{|\alpha|})\right)^{|\alpha|}}^{\Upsilon_\alpha(\bar{F})} |G|_{|\alpha|} 
\lt(1+\left({\rm det }\,\sigma_{\bar{F}}\right)^{-{2|\alpha|}}\rt).
$$
By Hölder inequality, we deduce that
\begin{equation}\label{eq:controlgenmall}
|\ES\,[\partial_\alpha  f({\bar{F}}) G]|\le \ES\,[\Upsilon_\alpha(\bar{F})^3]^{\frac{1}{3}}
\|G\|_{|\alpha|,3}\ES\,\lt[{\lt(1+\left({\rm det }\,\sigma_{\bar{F}}\right)^{-{2|\alpha|}}\rt)^3} 1_{|G|>0}\rt]^{\frac{1}{3}}.
\end{equation}
Let us  upper-bound $\ES\,[\Upsilon_\alpha(\bar{F})^3]^{\frac{1}{3}}$ by a simpler quantity. First, denoting the largest eigenvalue of a symmetric matrix $A$ by $\bar{\lambda}_A$, we remark that
$$|{\rm det }\,\sigma_{\bar{F}}|\le \bar{\lambda}_{\sigma_{\bar{F}}}^{d}\le C \|\sigma_{\bar{F}}\|^d$$
where  $\|\,.\,\|$ stands for the Frobenius norm and where the second inequality follows from the equivalence of norms in finite dimension. But, one easily checks that
$$\|\sigma_{\bar{F}}\|^2\le C |\bar{F}|_{1,1}^2$$
so that 
$$1+{\rm det }\,\sigma_{\bar{F}}+|{\bar{F}}|_{\alplunmoinszero}^{2d}\le C(1+ |\bar{F}|_{\alplunmoinszero}^{2d}).$$
Thus, by the elementary inequality $|u+v|^{|\alpha|}\le 2^{|\alpha|-1}(|u|^{|\alpha|}+|v|^{|\alpha|})$, we get:
$${
\Upsilon_\alpha(\bar{F})\le C\left[(1+|\bar{F}|_{\alplunmoinszero}^{(2d+1)|\alpha|})+(1+ |\bar{F}|_{\alplunmoinszero}^{2d})|L{\bar{F}}|_{|\alpha|}^{|\alpha|}\right].
}
$$
Thus, using~\eqref{eq:malliavcontrol12} (Meyer inequality) and Cauchy-Schwarz inequality, we deduce that
\begin{align*}
\ES\,[\Upsilon_\alpha(\bar{F})^3]&\le C \left(\ES\,[1+|\bar{F}|_{\alplunmoinszero}^{3(2d+1)|\alpha|}]+\ES\,[1+|\bar{F}|_{\alplunmoinszero}^{{12}d|\alpha|}]^{\frac{1}{2}}\ES\,[|\bar{F}|_{|\alpha|+2}^{6|\alpha|}]^{\frac{1}{2}}\right)
\end{align*}
and, hence, if 

$$
\ES\,[\Upsilon_\alpha(\bar{F})^3]^{\frac{1}{3}}\le C \ES\,[(1+|\bar{F}|_{\alplunmoinszero})^{12d|\alpha|}]^{\frac{1}{3}}(1+\ES\,[|\bar{F}|^{6|\alpha|}_{|\alpha|+2}]^{\frac{1}{6}}).
$$
Thus, we get the following inequality (where as usual, $C$ denotes a constant which may change from line to line):
\begin{equation}\label{eq:controlgenmall2}
|\ES\,[\partial_\alpha  f({\bar{F}}) G]|\le C (1+\ES\,[|\bar{F}|_{\alplunmoinszero}^{12d|\alpha|}]^{\frac{1}{3}})(1+\ES\,[|\bar{F}|^{6|\alpha|}_{|\alpha|+2}]^{\frac{1}{6}})\|G\|_{|\alpha|,3}(1+\ES\,[({\rm det }\,\sigma_{\bar{F}})^{-6|\alpha|} 1_{|G|>0}]^{\frac{1}{3}}).
\end{equation}
We now want to apply Inequality~\eqref{eq:controlgenmall2} with ${\bar{F}}=\mathfrak{X}(u,F)$ and $G=\phi(\mathfrak{X}(u,F))\Psi_\eta({\rm det} \, \sigma_{F})$. Note that as in Lemma~\ref{lem:contrmallideriv}, we implicitly assume that $\mathfrak{X}(u,F)$ is an ${\cal F}_{T+u}$-measurable random variable. On the one hand,
\begin{equation}\label{eq:cdfe2}
\ES\,[({\rm det }\,\sigma_{\bar{F}})^{-6|\alpha|} 1_{|G|>0}]^{\frac{1}{3}}\le \ES\,[\left({\rm det }\,\sigma_{\mathfrak{X}(u,F)}\right)^{-6|\alpha|}  1_{{\rm det} \, \sigma_{F}\ge\frac{\eta}{2}}]^{\frac{1}{3}}\le C \eta^{-2|\alpha|},
\end{equation}
by Lemma~\ref{lem:contrmallideriv}$(i)$. Now, since $x\mapsto\mathfrak{X}(u,x)$ is ${\cal C}^{k+2}$ (since $b$ and $\sigma$ are ${\cal C}^{k+2}$), one derives that if $F$ is Malliavin-differentiable up to order $k$, then $\mathfrak{X}(u,F)$ so is. Furthermore,
since $b$ and $\sigma$ have bounded derivatives, one can check that for every multi-index $\beta$ such that $1\le |\beta|\le|\alpha|$, for every $p>0$, for every $\tau>0$,
\begin{equation}\label{eq:bound1}
\sup_{x\in\ER^d}\sup_{u\in[0,\Tdeux]}\ES\,[\|\partial_x^\beta\mathfrak{X}(u,x)\|^p]<+\infty.
\end{equation}
Then, using Assumption~\eqref{control:momentmalliav22}, the boundedness of $\sigma$ and the Hölder inequality, a tedious computation of the Malliavin derivatives of $\mathfrak{X}(u,F)$ shows that for every $p>0$,
\begin{equation}\label{eq:malliavinmathfrak}
\sup_{u\in[0,\Tdeux]}\ES\,[|\mathfrak{X}(u,F)|_{\alpldeuxmoinszero}^p]\le C_p<+\infty.
\end{equation}
Thus, in view of~\eqref{eq:controlgenmall2}, we deduce that a constant $C$ exists (which does only depend on $\tau$) such that
$$\ES\,[|\mathfrak{X}(u,F)|_{\alplunmoinszero}^{12d|\alpha|}]\le C \quad\textnormal{and}\quad \ES\,[|\mathfrak{X}(u,F)|^{6|\alpha|}_{|\alpha|+2}]\le C(1+\ES\,[|\mathfrak{X}(u,F)|^{6|\alpha|}]).$$
Now, by a classical Gronwall argument, for every $\Tdeux>0$,  for every $p>0$,
\begin{equation}\label{eq:bound2}
\ES\,[|\mathfrak{X}(u,x)|^p]\le C(1+|x|^p),
\end{equation}
so that 
\begin{equation}\label{eq:cdfe1}
\ES\,[|\mathfrak{X}(u,F)|^{6|\alpha|}]\le C(1+\ES\,[|F|^{6|\alpha|}]).
\end{equation} 
At this stage, we thus deduce from~\eqref{eq:controlgenmall2},
\begin{equation}\label{eq:confin24}
|\ES\,[\partial_\alpha  f(\mathfrak{X}(u,F)) \phi(\mathfrak{X}(u,F))\Psi_\eta({\rm det} \, \sigma_{F})]|\le C \eta^{-2|\alpha|}(1+\ES\,[|F|^{6|\alpha|}]^{\frac{1}{6}})\|\phi(\mathfrak{X}(u,F))\Psi_\eta({\rm det} \, \sigma_{F})\|_{|\alpha|,3}.\end{equation}
It thus remains to bound  the last right-hand term. We again use chain rule for Malliavin calculus. In view of the application of the Leibniz formula (for the derivative of the product of functions), we  study the  Malliavin derivatives of  $\phi(\mathfrak{X}(u,F))$ and $\Psi_\eta({\rm det} \, \sigma_{F})$ separately. For  $\phi(\mathfrak{X}(u,F))$, we choose to write the arguments in the one-dimensional case (the extension to multidimensional case involves technicalities but leads to the same conclusion~\eqref{eq:cotegauchemall} below). In this case, $D^{(\ell)} \phi(\mathfrak{X}(u,F))$ takes the form:
\begin{equation}\label{eq:confin23}
D^{(\ell)} \phi(\mathfrak{X}(u,F))=\sum_{r=1}^{\ell} \phi^{(r)}(\mathfrak{X}(u,F)) Q^{(r)} ( D\mathfrak{X}(u,F),\ldots, D^{(\ell)} \mathfrak{X}(u,F)),
\end{equation}
where $Q^{(r)}$ denotes a multivariate polynomial function (with degree lower than $r$). Since $|\phi^{(r)}(x)|\le C(1+|x|^2)$, it follows from a Gronwall argument that for every $p>0$, for every $\Tdeux>0$, a constant $C$ exists such that
$$\sup_{u\in[0,\Tdeux]}\ES\,[|\phi^{(r)}(\mathfrak{X}(u,F))|^p]\le C(1+\ES\,[|F|^{2p}]).$$
On the other hand, by~\eqref{eq:malliavinmathfrak} and Assumption~\eqref{control:momentmalliav22}, one deduces that  for every positive $p$ and $\Tdeux$ , a constant $C$ exists such that
$$
\sup_{u\in[0,\Tdeux]}\sup_{s_1,s_2,\ldots, s_\ell\in[0,\Tun+\Tdeux]} \ES\,[|Q^{(r)}_{s_1,\ldots,s_\ell} ( D\mathfrak{X}(u,F),\ldots, D^{(\ell)} \mathfrak{X}(u,F))|^p]<+\infty.
$$
By Cauchy-Schwarz inequality, one deduces that for every positive $\Tun$, $\Tdeux$ and $p$
\begin{equation}\label{eq:cotegauchemall}
\sup_{u\in[0,\Tdeux]}\sup_{(s_1,s_2,\ldots, s_\ell)\in[0,\Tun+\Tdeux]} \ES\,[\|
D^{(\ell)}_{s_1,\ldots,s_\ell} \phi(\mathfrak{X}(u,F))\|^p]^{\frac{1}{p}}\le C (1+\ES\,[|F|^{4p}]^{\frac{1}{2p}}).
\end{equation}
Let us now consider $\Psi_\eta({\rm det} \, \sigma_{F})$. We have $\|\Psi^{(\ell)}_\eta\|_\infty\le C\eta^{-\ell}$. Then,  
using that ${\rm det}$  is a polynomial function and Assumption~\eqref{control:momentmalliav22}, one can deduce that
\begin{equation}\label{eq:cotegauchemall3}
\sup_{(s_1,s_2,\ldots, s_\ell)\in[0,\Tun]} \ES\,[\|
D^{(\ell)}_{s_1,\ldots,s_\ell} \Psi_\eta({\rm det} \, \sigma_{F})\|^p]^{\frac{1}{p}}\le C \eta^{-\ell}.
\end{equation}
Then, by Leibniz formula and Cauchy-Schwarz inequality, we deduce from~\eqref{eq:cotegauchemall} and~\eqref{eq:cotegauchemall3} (applied with $p=6$) that for every $\ell\in\{1,\ldots,|\alpha|\}$,
$$\sup_{u\in[0,\Tdeux]}\sup_{(s_1,s_2,\ldots, s_\ell)\in[0,\Tun+\Tdeux]}\ES\,[\| D^{(\ell)}_{s_1,\ldots,s_\ell} \left(\phi(\mathfrak{X}(u,F))\Psi_\eta({\rm det} \, \sigma_{F})\right)\|^3]^{\frac{1}{3}}\le 
C\eta^{-\ell} (1+\ES\,[|F|^{24}]^{\frac{1}{12}}).
$$
Now, since $\Psi_\eta$ is bounded, one easily checks that 
$$\sup_{u\in[0,\Tdeux]}\ES\,[|\phi(\mathfrak{X}(u,F))\Psi_\eta({\rm det} \, \sigma_{F})|^p]\le C(1+\ES\,[|F|^{2p}]).$$
It follows from the two previous inequalities that 
$$\sup_{u\in[0,\Tdeux]}\|\phi(\mathfrak{X}(u,F))\Psi_\eta({\rm det} \, \sigma_{F})\|_{|\alpha|,3}\le C\eta^{-|\alpha|}(1+\ES\,[|F|^{24}]^{\frac{1}{12}}).
$$
Plugging this inequality into~\eqref{eq:confin23}, the result follows.

\smallskip
\noindent $(iv)$ We have:
$$
\ES\,[\varphi^{(2)}_{h,M}(F,\theta,\sqrt{h} Z)\Psi_\eta({\rm det} \, \sigma_{F})]=\sum_{\alpha, |\alpha|=4}\ES\,[\partial_\alpha f(\bar{\mathfrak{X}}_\theta(h,F)) G_\alpha]
$$
where for a given $\alpha=(\alpha_1,\ldots,\alpha_4)$
$$
G_\alpha=\prod_{i=1}^4 \left(h b_{\alpha_i}(F)+\sigma_{\alpha_i,.}(F)(W_{\Tun+h}-W_{\Tun})\right) \mathfrak{T}_M(W_{\Tun+h}-W_{\Tun}) \Psi_\eta({\rm det} \, \sigma_{F}).
$$
The strategy is then quite similar to $(iii)$. More precisely, for any $\alpha=(\alpha_1,\ldots,\alpha_4)$, we start by applying~\eqref{eq:formulanondegenerate} with $\bar{F}=\bar{\mathfrak{X}}_\theta(h,F)$ and $G=G_\alpha$, which leads to the inequality~\eqref{eq:controlgenmall2}. Then, as in $(iii)$, it remains to control each term of the right-hand side of~\eqref{eq:controlgenmall2}. 
Let us begin by the last one. Noting that (with the definition of $\mathfrak{T}_M$),
$$
\{|G_\alpha|>0\}\subset \{{\rm det} \, \sigma_{F}\ge \frac{\eta}{2}, |W_{\Tun+h}-W_{\Tun}|\le 2M\},
$$
we deduce from  Lemma~\ref{lem:contrmallideriv}$(ii)$ that
$$
\ES\,[{\rm det }\,\sigma_{\bar{\mathfrak{X}}_\theta(h,F)}^{-6|\alpha|} 1_{|G_\alpha|>0}]^{\frac{1}{3}}\le C \eta^{-6|\alpha|}.
$$
Then, since $x\mapsto \bar{\mathfrak{X}}_\theta(h,x)$ admits similar bounds as
$x\mapsto \mathfrak{X}(u,x)$ (in particular~\eqref{eq:bound1} and~\eqref{eq:bound2}), some arguments similar to $(iii)$ lead to an inequality similar to~\eqref{eq:confin24} (with $|\alpha|=4$): $\forall \,\alpha=(\alpha_1,\ldots,\alpha_4)$,
\begin{equation*}
|\ES\,[\partial_\alpha  f(\bar{\mathfrak{X}}_\theta(h,F))G_\alpha]|\le C \eta^{-8}(1+\ES\,[|F|^{24}])^{\frac{1}{6}}\|G_\alpha\|_{4,3}.\end{equation*}
It remains to control $\|G_\alpha\|_{4,3}$. The strategy of proof follows the lines of the ones for the control of $\|G\|_{|\alpha|,3}$ in $(iii)$. Once again, a tedious computation using that $b$ has sublinear growth and the fact $\mathfrak{T}_M$ is smooth with bounded derivatives leads to:
$$
\|G_\alpha\|_{4,3}\le C h^2  \eta^{-4} (1+\ES\,[|F|^{24}])^{\frac{1}{6}}.
$$
The result follows.
\end{proof}

\subsubsection{Bounds of Malliavin derivatives and Semi-nondegeneracy for the Euler scheme}
\begin{prop}\label{lem:mallderivsuc} Let $(\bar{X}_{t_n})$ denote a Euler scheme starting from $x$ with non-increasing step sequence $(h_n:=t_n-t_{n-1})_{n\ge1}$. Let $\ell\ge1$. Assume that $b$ and $\sigma$ are ${\cal C}^{\ell}$ with bounded partial derivatives and $\sigma$ satisfies $\ELLIP$.

\smallskip
\noindent $(i)$ Let the  smoothness assumption hold with $\ell=1$. 
Then, for any $p>0$, there exists a real  constant $\bar{h}>0$  such that, for any   $T,r>0$, there is areal  constant $\mathfrak{C}=C(T,p,\bar{h},r)>0$ satisfying: if $h_1\le \bar{h}$,   for any $\eta>0$ and  any $n$ such that $T/2\le t_n\le T,$ 
$$
\PE({\rm det} \,\sigma_{\bar{X}_{t_n}}\le \eta)\le \mathfrak{C}(h_1^r+\eta^p).
$$
\noindent $(ii)$    Furthermore, if smoothness assumption holds  for $\ell\ge1$,
$$\sup_{t_n\in[0,T],(s_1,\ldots,s_\ell)\in[0,t_n]^{\ell}} \ES\,[ \|\Der _{s_1\ldots s_\ell} ^{(\ell)} \bar{X}_{t_n}\|^{2p}]\le \mathfrak{d}_{p,T,\ell}<+\infty,$$
where $\mathfrak{d}_{p,T,\ell}$ is a finite positive constant. 
\end{prop}
\begin{proof} 
$(i)$ 
Let $s\in[0,T)$. Using the chain rule for Malliavin derivatives, one checks that $D_s \bar{X}_{s+.}$ formally satisfies for any $u\ge0$:
\begin{equation}\label{eq:formulamall}
D_s \bar{X}_{s+u}=
\begin{cases}
\sigma(\bar{X}_{\un{s}})&\textnormal{if $u\le \bar{s}-s$}\\
\sigma(\bar{X}_{\un{s}})+\int_{\bar{s}}^{s+u} \nabla b (\bar{X}_{\un{v}}) D_s \bar{X}_{\un{v}} dv+\int_{\bar{s}}^{s+u}\nabla \sigma (\bar{X}_{\un{v}}) D_s \bar{X}_{\un{v}} dW_v&\textnormal{if $u> \bar{s}-s$}.
\end{cases}
\end{equation}
By ``formally'', we mean that we do not detail the rules for the operations between tensors. With some more precise notations, this yields in the case $u> \bar{s}-s$: for every $\ell$ and $i\in\{1,\ldots,d\}$,
\begin{equation*}
D_s^\ell \bar{X}_{s+u}^i=
\sigma_{i,\ell}(\bar{X}_{\un{s}})+\sum_{k=1}^d \int_{\bar{s}}^{s+u}  \partial_k b_i (\bar{X}_{\un{v}}) D_s^\ell \bar{X}_{\un{v}}^k dv+\sum_{k,j=1}^d\int_{\bar{s}}^{s+u}\partial_{k}\sigma_{i,j} (\bar{X}_{\un{v}}) D_s^\ell \bar{X}_{\un{v}}^k dW_v^j.
\end{equation*}

For the sake of readability, we keep such formal notations in the sequel of the proof.
Let us denote by $(\bar{Y}_t)_{t\ge0}$ the ``pseudo-tangent'' process: $\bar{Y}_t=(\partial_{x_j} \bar{X}_t^i)_{1\le i,j\le d}$
where $(\bar{X}_t)_{t\ge0}$ is the continuous-time Euler scheme defined by~\eqref{eq:genuine}. One checks that  $(\bar{Y}_t)_{t\ge0}$ is recursively defined by: $\bar{Y}_0={\rm I_d}$ and 
for any $t\ge0$:
$$
\bar{Y}_t=(I_d+A_{\un{t}t})\bar{Y}_{\un{t}},
$$
where
$$
A_{\un{t}t}=(t-\un{t})\nabla b (\bar{X}_{\un{t}})+\nabla \sigma (\bar{X}_{\un{t}}) (W_t-W_{\un{t}}).
$$

Set 
\begin{equation}\label{eq:defomega}
\Omega_{\zeta}:=\bigcap_{k=0}^{n-1} \Big\{\sup_{u\in[t_k,t_{k+1})} \|A_{\un{u} u}\|<\zeta\Big\},\quad \zeta\in(0,1],
\end{equation}
where $\|\cdot\|$  stands for the Fr\"obenius norm and {$\zeta$ will be specified  later (see~\eqref{eq:zeta}) as a constant only depending on $d$}. On $\Omega_{\zeta}$, $\bar{Y}_s$ is invertible (as a product of invertible matrices), for every $s\in[0,t_n]$, and one checks that
$$
D_s \bar{X}_{t}=\bar{Y}_{t\vee\bar{s}} \bar{Y}_{\bar{s}}^{-1}\sigma(\bar{X}_{\un{s}}).
$$
Let  $t_n\in(0, T]$ and let $F_n=\bar{X}_{t_n}$. The Malliavin matrix $\sigma_{F_n}$ of $F_n$ is given by:
$$
\sigma_{F_n}= \int_0^{t_n} D_s F_n D_s F_n^* ds,
$$
which, after classical computations yields:
$$ 
\sigma_{F_n}=\bar{Y}_{t_n}\bar{U}_{t_n} \bar{Y}_{t_n}^*\quad\textnormal{with}\quad
\bar{U}_{t_n}=\int_0^{t_n}\bar{Y}_{\bar{s}}^{-1} (\sigma\sigma^*)(\bar{X}_{\un{s}})(\bar{Y}_{\bar{s}}^{-1} )^*ds.
$$
For any $\eta>0$ and $p>0$,
\begin{equation}\label{eq:fjidojsdd}
\PE({\rm det}\, \sigma_{F_n}\le \eta)\le \PE(\Omega_{\zeta}^c)+{\eta}^p\ES\,[{\rm det}\, \sigma_{F_n}^{-p}\mbox{\bf 1}_{\Omega_{\zeta}}].
\end{equation}
On the one hand, using  that the partial derivatives of $b$ and $\sigma$ are bounded, we remark that
$$
\sup_{t\in[t_{k-1},t_{k} )} \|A_{\un{t} t}\|\le C\big(h_k+\sup_{t\in[t_{k-1},t_{k}]} |W_t-W_{t_{k-1}}|\big),
$$
where $C=C_{b,\s,d} >0$ is a real constant depending on $d$, $\|\nabla b\|_\infty$ and $\|\nabla\sigma \|_\infty$. Hence, owing to the independence and the stationarity of the increments of the Brownian motion, 
we get
$$
\PE(\Omega_{\zeta})\ge   \prod_{k=1}^n\Big(1- \P\big(\sup_{t\in[0,h_{k}]} |W_t| \ge \zeta C^{-1}-h_k \big)\Big).
$$
Moreover, if $B$ denotes a standard one-dimensional Brownian motion, one has for every $h$ and $u> 0$,
\[
\P\Big(\sup_{t\in[0,h]} |W_t| \ge u\Big) \le q\,\P\Big(\sup_{t\in[0,h]} |B_t| \ge \tfrac uq\Big)\le 2q\, \P\Big(\sup_{t\in[0,h]} B_t \ge \tfrac uq\Big)=2q\, \P\Big(\sup_{t\in[0,1]} B_t \ge \tfrac{u}{q\sqrt{h}}\Big),
\] 
where we used that $B\stackrel{d}{=}-B$, $|x|= \max(x,-x)$   in the second inequality and the scaling property in the equality. Now $\sup_{t\in[0,1]} B_t \stackrel{d}{=}|Z|$ with $Z\stackrel{d}{=}{\cal N}(0;1)$ and $\P(|Z|\ge z) \le e^{-\frac{z^2}{2}}$ for every $z\ge 0$ and we deduce that
\[
\P\big(\Omega_{\zeta})\ge \prod_{k=1}^n \Big(1-2q e^{-\frac{(\zeta C^{-1}-h_k)^2}{2q^2h_k}}  \Big)\ge   \prod_{k=1}^n \Big(1-\kappa_0 e^{-\frac{\zeta^2}{2C^2q^2h_k}} \Big)
\]
where $\kappa_0 =  2q e^{-\frac{\zeta}{Cq^2 }}$ and $\kappa_1= \frac{\zeta^2}{2C^2q^2}$ only depend on $q$, $d$, $b$ and $\s$.
For $h_1\!\in (0,\bar h]$ with $\bar h$ small enough (and $\le T$) so that that $\kappa_0 e^{-\frac{\kappa_1}{h_1}}\le \frac{1}{2}$, we have 
$\kappa_0 e^{-\frac{\kappa_1}{h_k}} \le \frac{1}{2}$ for every  $k\ge1$ since $(\h_k)_k$ is non-increasing. Thus, combining this with the elementary inequalities $\log(1+u)\ge 2 u$ on $[-1/2,0]$ and $1-e^{-u}\le u$ on $[0,+\infty)$, we deduce that
$$
\PE(\Omega_\zeta^c)\le 1-\exp\Big(-2 \kappa_0 \sum_{k=1}^n e^{- \frac{\kappa_1}{h_k}}\Big)\le 2\kappa_0\sum_{k=1}^n e^{- \frac{\kappa_1}{h_k}}.
$$
Now, for any $r>0$, there exists a constant $C$  such that $e^{-\frac{\kappa_1}{x}}\le C_r x^{r+1}$ for any $x\in[0, \bar h]$. Thus,
$$
\PE(\Omega_\zeta^c)\le 2\kappa_0\,C_r\sum_{k=1}^n h_k^{r+1}\le 2\kappa_0\,C_rt_n h_1^r\le C_{T,r,\zeta, b,\s, d,q}\cdot  h_1^{r}.
$$
Let us now turn to the second term of~\eqref{eq:fjidojsdd}. Recall that ${\rm det}$ is log-concave on ${\cal S}^{+\!+}(d,\ER)$, hence $M\mapsto {\rm det}^{-p}(M)$ is convex on ${\cal S}^{+\!+}(d,\ER)$  for any $p>0$. Thus,
by Jensen's inequality, we get 
\begin{equation}\label{eq:sigmaFn}
\ES\,\big[{\rm det} \,\sigma_{F_n}^{-p}\mbox{\bf 1}_\Omega\big]\le t_n^{-pd-1} \int_0^{t_n}\ES\,\big[{\rm det}^{-p} \left(\bar{Y}_{t_n}\bar{Y}_{\bar{s}}^{-1} (\sigma\sigma^*)(\bar{X}_{\un{s}})(\bar{Y}_{\bar{s}}^{-1})^*\bar{Y}_{t_n}^*\right)\mbox{\bf 1}_\Omega\big]ds.
\end{equation}
Using that $\sigma\sigma^\star\ge {\underline\sigma_0^2 } I_d$ {and det is also non-decreasing on ${\cal S}^{+\!+}(d,\ER)$},  we get
\begin{align*}
\ES\,\big[{\rm det} \, \sigma_{F_n}^{-p}\mbox{\bf 1}_\Omega\big]&\le \underline \sigma_0^{-2p d} t_n^{-pd} 
\sup_{s\in[0,t_n]}\ES\,[|{\rm det}(\bar{Y}_{t_n}\bar{Y}_{\bar{s}}^{-1})|^{-2p}\mbox{\bf 1}_\Omega ]\\
&\le  \underline\sigma_0^{-2p d} t_n^{-pd} \ES\,[|{\rm det}( \bar{Y}_{t_n}^{-1})|^{4p}\mbox{\bf 1}_\Omega ]^{\frac{1}{2}}
\sup_{s\in[0,T]}\ES\,[|{\rm det}( \bar{Y}_{\bar{s}})|^{4p} ]^{\frac{1}{2}}\\
& \le C_T \underline\sigma_0^{-2p d} t_n^{-pd} \ES\,[|{\rm det}( \bar{Y}_{t_n})|^{-4p}\mbox{\bf 1}_\Omega ]^{\frac{1}{2}},
\end{align*}
where in the last line we used that $\sup_{s\in[0,T]}\ES\,[|{\rm det}( \bar{Y}_{\bar{s}})|^{4p} ]\le C_T$ (using that the moments of $\bar{Y}_t$ can be uniformly bounded on $[0,T]$ with the help of a Gronwall argument).
Then, having in mind that the Trace operator is the differential of the determinant at $I_d$,  yields a constant $C$ such that,  for $M\!\in {\cal M}=\{M\!\in {\cal M}(d,\R): \|M\|\le 1/2\}$,
$$
{\rm det}(I_d+M)\ge 1+{\rm Tr}(M)-C \|M\|^2.
$$
Furthermore,  one can choose $\zeta= \zeta_d>0$ small enough  in such a way that  
\begin{equation}\label{eq:zeta}
\|M\|\le\zeta_d\Longrightarrow 1+{\rm Tr}(M)-C \|M\|^2\ge 1/2.
\end{equation}
Thus, taking such a $\zeta$ in~\eqref{eq:defomega} and using again that $\log(1+x)\ge 2 x$ on $[{-}1/2,0]$, we obtain:
\begin{align*}
\ES\,\big[{\rm det} \, \sigma_{F_n}^{-p}\mbox{\bf 1}_\Omega\big]&\le C_{T}  \underline\sigma_0^{-{2}p d} t_n^{-pd}\ES\prod_{k=1}^{n}(1+{\rm Tr}(A_{t_{k-1} t_{k}})-C \|A_{t_{k-1} t_{k}}\|^2)^{-p}\\
&\le C_{T}  \underline\sigma_0^{-{2}p d} t_n^{-pd}\ES\exp\Big(-2p\sum_{k=1}^{n}\Big({\rm Tr}(A_{t_{k-1} t_{k}})-C\|A_{t_{k-1} t_{k}}\|^2\Big)\Big)\\
&\le C_{T}  \underline\sigma_0^{-{2}p d} t_n^{-pd}\ES\exp\Big(Cp\Big(1-{\rm Tr}\Big(\int_0^{t_n}\nabla \sigma (\bar{X}_{\un{t}}) dW_t\Big)+
\|\nabla \sigma\|_\infty \sum_{k=1}^n |W_{t_k}-W_{t_{k-1}}|^2\Big)\Big),
%
%
\end{align*}
where in the last line, we used that $\nabla b$ is a bounded function.  Now, using that for all $(u,v)\in\ER^2$ such that $v<1/2$,
$\E_{{Z \sim {\cal N}(0,1)}}\big[e^{u Z+v Z^2}\big]= (1-2v)^{-\frac{1}{2}}e^{\frac{u^2}{2(1-2v)}},$ we deduce from a chain rule of towered expectations that if $Cp\|\nabla \sigma\|_{\infty}h_1\le 1/4$,
\begin{align*}
\ES\,\big[{\rm det} \, \sigma_{F_n}^{-p}\mbox{\bf 1}_\Omega\big]&\le C_{T}  \underline\sigma_0^{-{2}p d} t_n^{-pd}
 \exp\Big[\tilde{C}_{p,T}\|\nabla\sigma \|_{\infty}^2 t_n\Big] \Big(\prod_{k=1}^n \big(1-2Cp\|\nabla \sigma\|_{\infty}(t_k-t_{k-1})\big)\Big)^{-q/2}\\
 &\le C_{p,T}  \underline\sigma_0^{-{2}p d} t_n^{-pd} \exp \big( 2Cpq \|\nabla \sigma\|_{\infty}\sum_{k=1}^n h_k\big) \le C_{p,T}  \underline\sigma_0^{-{2}p d} t_n^{-pd}\exp \big( 2Cpq \|\nabla \sigma\|_{\infty} T\big),
\end{align*}
where in the penultimate inequality, we again used that $\log(1+x)\ge 2x$ on $[-1/2,0]$ (and where as usual, the constants may have changed from line to line). Hence, taking  $t_n\in[T/2,T]$ yields
\begin{equation*}
\ES\,\big[{\rm det} \, \sigma_{F_n}^{-p}\mbox{\bf 1}_\Omega\big] \le C_{p,q,T} \un{\varepsilon}_0^{-2p d} (T/2)^{-pd},
\end{equation*}
where $C_{p,T}$ stands for  a finite constant depending on $p$, $q$ and $T$. The statement follows.

\noindent $(ii)$ For the sake of simplicity, we only prove the result in the one-dimensional case.
 For $\ell=1$,  we start with the formula~\eqref{eq:formulamall} which implies, for any $s\in[0,t_n]$ and any $p>0$,
 $$ 
 |D_s \bar{X}_{t_n}|^p =  \Big(\sigma^2(\bar{X}_{\un{s}})\prod_{k=N(s)+1}^{n-1}\left(1+(t_{k+1}-t_{k})b'(\bar{X}_{\un{v}})+
 \sigma'(\bar{X}_{t_k})(W_{t_{k+1}}-W_{t_k})\right)^2\Big)^{\frac p2}
 $$
 with $N(s)=\max\{k, t_k\le s\}$ and the convention $\prod_{\emptyset}=1$.
 Thus, using the elementary inequality $\log\big((1+x)^2\big) = \log(1+2x+x^2)\le 2x+x^2$ on $\ER$ (with the convention $\log(0)=-\infty$), we get for any $p>0$,
 $$
 \ES\,[|D_s \bar{X}_{t_n}|^{p}]\le \|\sigma\|_\infty^{p}\ES\exp\left(\frac{p}{2}\sum_{k=N(s)+1}^{n-1} (2A_{t_k t_{k+1}}+|A_{t_k t_{k+1}}|^{{2}})\right)
 $$
 with $A_{t_k t_{k+1}}=(t_{k+1}-t_{k})b'(\bar{X}_{\un{v}})+ \sigma'(\bar{X}_{t_k})(W_{t_{k+1}}-W_{t_k})$. Then, 
 $$
 \ES\,[|D_s \bar{X}_{t_n}|^{p}]\le  \|\sigma\|_\infty^{p}\exp{\Big(\frac{p T}{2}\|b'\|_\infty +\frac{p^2 T}{4}\|\sigma'\|_\infty^2\Big)}
 $$
 owing to standard estimates for exponential of stochastic integrals.
 
  When $\ell\ge1$, the idea is to iterate the Malliavin differentiation in~\eqref{eq:formulamall}. We give the main ideas when $\ell=2$ but do not detail the general case. When $\ell=2$,  one can deduce from the chain rule and~\eqref{eq:formulamall} that for any $t\ge0$ and $(s,v) \in [0,t]^2$,
  \begin{align*}
 D^2_{vs}\bar{X}_{t}=D_{v}(D_s\bar{X}_{t})=
 \begin{cases}
 \sigma'(\bar{X}_{\un{t}}) D_v \bar{X}_{\un{t}}&\textnormal{if $v< \un{t}\le s< t$,}\\
 \sigma'(\bar{X}_{\un{s}}) D_v \bar{X}_{\un{s}}
+\int_{\bar{s}}^t D_v \bar{X}_{\un{u}}D_s \bar{X}_{\un{u}}(b''(\bar{X}_{\un{u}})du+ \sigma''(\bar{X}_{\un{u}}) dW_u)&\\
 \hskip 2.05cm  + \int_{\bar{s}}^tD_{vs}^2\bar{X}_{\un{u}}(b'(\bar{X}_{\un{u}})du+\sigma'(\bar{X}_{\un{u}})  dW_u)
& \textnormal{if  $0\le s,v< \un{t}$.}
 \end{cases}
 \end{align*}
 Thus, applying Itô formula to $|D^2_{vs}\bar{X}_{t}|^p$ with $p\ge2$, we easily deduce from  martingale arguments and the boundedness  of the derivatives that if $s,v<\un{t}$
 \begin{align*}
\ES\,[|D^2_{vs}\bar{X}_{t}|^p]&\le \ES\,[|\sigma'(\bar{X}_{\un{s}}) D_v \bar{X}_{\un{s}}|^p]+c_p \int_{\bar{s}}^t\ES\,[ |D^2_{vs}\bar{X}_{u}|^{p-1} |D_v \bar{X}_{\un{u}}D_s \bar{X}_{\un{u}}|]du\\
&\quad + c_p \int_{\bar{s}}^t\ES\,[|D_{vs}^2\bar{X}_{{u}}|^{p-1}|D_{vs}^2\bar{X}_{\un{u}}|] du+c_p\int_{\bar{s}}^t\ES\,[|D_{vs}^2\bar{X}_{{u}}|^{p-2}\left(1+
|D_v \bar{X}_{\un{u}}D_s \bar{X}_{\un{u}}|^2] du\right).
\end{align*}
Then, by the Young inequality and the control of the moments of $D_. \bar{X}_{\un{u}}$ previously established, 
we get by setting $S_t=\sup_{v\in[0,t]} |D^2_{vs}\bar{X}_{v}|^p$,
 \begin{align*}
\ES\,[S_t]\le C_p+\int_{\bar{s}}^t (1+\ES\,[S_u]) du.
\end{align*}
The result then follows from Gronwall's inequality. 
 \end{proof}

 %

Before proving Theorems~\ref{thm:multiplicative}$(b)$ and~\ref{thm:additive}$(b)$, {let us make the connection between $\HWO$ and its {\em TV}-counterpart for uniformly elliptic diffusions.  

\section{Proof of  Theorem~\ref{thm:bismutmultidim} and of Propositions~\ref{prop:hwo},~\ref{thm:generalsigma} and~\ref{Prop:higherdiff}}  
\subsection{Proof of extended BEL identity (Theorem~\ref{thm:bismutmultidim})}\label{Annex:BEL}
Let $M>0$  and   $f_M(x)=f(x)1_{\{|f(x)
|\le M\}}$. Set $\phi_M(x)=P_t f_M(x)$.
First, since $f(X_t^x)$ belongs to $L^1$ for any $x$ since $f$ has polynomial growth and $b$ and $\s$ are Lipschitz continuous. We deduce from the dominated convergence theorem that 
$\phi_M$ converges simply to $\phi= P_tf$. Furthermore, as $f_M$ is bounded, 
$$
\phi'_M(x)=\ES\Big[f_M(X^x_t)\frac 1t \int_0^t (\sigma(X_s)^{-1}Y^{(x)}_s)^*dW_s  \Big].
$$
We wish to prove that $\phi'_M$ converge uniformly on compact set $K$ i.e.
$$
\sup_{x\in K} \ES\Big[\big|f(X^x_t)\big|\mbox{\bf 1}_{\{|f(X^x_t)|>M\}} \left| \int_0^t (\sigma(X_s)^{-1}Y^{(x)}_s)^*dW_s\right|  \Big]\xrightarrow{M\rightarrow+\infty}0.
$$
It follows from Cauchy-Schwarz inequality  that
\begin{align*}
\ES \left[\big|f(X^x_t)\big|\mbox{\bf 1}_{\{|f(X^x_t)|>M\}} \left| \int_0^t (\sigma(X^x_s)^{-1}Y^{(x)}_s)^*dW_s\right| \right]& \le \Big[ \ES\, |f(X^x_t)|^2\mbox{\bf 1}_{\{|f(X^x_t)|>M\}}\Big]^{\frac{1}{2}} \Big[\ES  \int_0^t \left|\sigma(X^x_s)^{-1}Y^{(x)}_s\right|^2ds  \Big]^{\frac{1}{2}}\\
& \le  \frac{1}{\sqrt{M}} \Big[ \ES\,|f(X^x_t)|^3\Big]^{\frac{1}{2}} \Big[\int_0^t\ES  \left|\sigma(X^x_s)^{-1}Y^{(x)}_s\right|^2 ds  \Big]^{\frac{1}{2}}\\
& \le  \frac{1}{\s_0\sqrt{M}} \Big[1+C_f \ES\,|X^x_t|^{3r}\Big]^{\frac{1}{2}} \Big[\int_0^t\ES  \big| Y^{(x)}_s\big|^2 ds  \Big]^{\frac{1}{2}}
\end{align*}
where $|f(\xi)|^3\le C_f(1+|\xi|^{3r})$. Now, as $b$ and $\s$ have bounded partial derivatives (hence Lipschitz continuous),  it is classical background that $\ES\,|X^x_t|^{3r}\le C_{r,t}(1+|x|^3)$ and $\sup_{x \in \R^d, s\in [0,t]}\ES \, \big| Y^{(x)}_s\big|^2<+\infty$. This shows that the right hand side goes to $0$ as $M\to +\infty$ uniformly on compact sets of $\R^d$.%
\subsection{Proof of Proposition~\ref{prop:hwo}}\label{annexe:prophwo}
Owing to Remark~\ref{rk:2.3}, 
we may assume that $\HWO$ holds starting from $t=0$. Let $t_0>0$ being fixed. Let $t\ge t_0$ and let  $f:\ER^d\rightarrow\ER$ be a  bounded Borel function. By  the Markov property,
 $$
 \ES\,[f(X_t^x)-f(X_t^y)]=\ES\,[P_{t_0} f(X_{t-t_0}^x)-P_{t_0} f(X_{t-t_0}^y)].
 $$
 By  {\em BEL} identity (see Proposition~\ref{thm:bismutmultidim}), for any $z_1$ and $z_2\in\ER^d$,
$$ 
P_{t_0} f(z_2)-P_{t_0} f(z_1)=\big(\nabla P_{t_0} f(\xi)\,|\, z_2-z_1\big) =\frac{1}{t_0}\ES\left[ f(X_t)\Big(\int_0^{t_0}(\sigma^{-1}(X^{\xi}_s) Y^{(\xi)}_s)^*dW_s  \,|\,z_2-z_1\Big)\right ],
$$
where $\xi\in(z_1,z_2)$ (geometric interval) and $(Y^{(\xi)}_s)_{s\ge0}$ denotes the tangent process of $(X_s^\xi)$. But since $b$ and $\sigma$ have bounded derivatives and $(Y^{(x)}_s)_{s\ge0}$ starts from $I_d$, a Gronwall argument (see~\cite{Kunita}) shows that 
$$
\sup_{\xi\in\ER^d, \, s\in [0,t_0]} \ES\,\big \| Y^{(\xi)}_s\big\|^2<+\infty.
$$
By a standard martingale argument and the ellipticity condition $\ELLIP$, we deduce that $x\mapsto P_{t_0} f(x)$ is Lipschitz continuous and that 
$$
[P_{t_0} f]_{\rm Lip}\le C_0 \|f\|_\infty.
$$
Then, it follows from the Kantorovich-Rubinstein  representation of the $L^1$-Wasserstein distance   and the definition of total variation distance that 
 $$
 \|X_t^x-X_t^y\|_{TV}\le C_0 {\cal W}_1(X_{t-t_0}^x,X_{t-t_0}^y).
 $$
 But under $\HWO$ it follows from what precedes, for every $t\ge t_0$,
$$ 
{\cal W}_1(X_{t-t_0}^x,X_{t-t_0}^y)\le c  e^{-\rho (t-t_0)},
$$
for some real constant $c>0$. Hence, there exists a constant $C>0$  such that,  for every $t\ge t_0$, 
$$
\hskip 5cm\|X_t^x-X_t^y\|_{TV}\le C |x-y| e^{-\rho t}.\hskip 5cm \Box
$$
\subsection{Proof of Proposition~\ref{thm:generalsigma} }\label{annexe:thm:generalsigma}
We need to  check that  $g= e^{-V}$ satisfies the stationary Fokker-Planck equation ${\cal L}^*g=0$ where ${\cal L}^*$ denotes the adjoint operator of ${\cal L}={\cal L}_{_X}$ reading on $C^2$ test functions $g$
\[
{\cal L}^* g = -\sum_{i=1}^d \partial_{x_i} (b_ig) + \tfrac 12 \sum^d_{i,j=1}\partial^2_{x_i x_j}\big((\s\s^*)_{ij}g\big).
\]
Temporarily set $a=\sigma\sigma^*$.
For  every $i,j\!\in \{1,\ldots, d\}$, elementary computations show that
\begin{align*}
 \partial_{x_j} (b_ig)  & =  \frac{e^{-V}}{2}\left[\sum_{j=1}^d a_{ij}(\partial_{x_i}V)(\partial_{x_j}V) - (\partial_{x_i}a_{ij})(\partial_{x_j} V)  -( \partial_{x_j}a_{ij})(\partial_{x_i}V)  -a_{ij}(\partial_{x_ix_j}^2V)+\partial^2_{x_ix_j}a_{ij}\right]\\\
  \partial^2_{x_ix_j} (a_{ij}g)&= e^{-V}\left[\partial^2_{x_ix_j} a_{ij}-(\partial_{x_j}a_{ij})(\partial_{x_i}V)-(\partial_{x_i}a_{ij})(\partial_{x_j}V)+a_{ij}(\partial_{x_i}V)(\partial_{x_j}V) -a_{ij}(\partial^2_{x_ix_j}V)\right].
\end{align*}
One checks from these identities that  ${\cal L}^*g=0$.
Hence  ${{\nu_{_V}}} = C_{_V}e^{-V(x)}\cdot\lambda_d(dx)$ is an invariant distribution for $SDE$~\eqref{eq:SDEddim}. Uniqueness of the invariant distribution follows from uniform ellipticity. 
%

\subsection{Proof of Proposition~\ref{Prop:higherdiff}}\label{subset:Proofhigherdiff}
$(a)$ Start from
\begin{align*}
\partial_x P_tf (x) &= \frac 1t \E \left[f(X^x_t) \int_0^t (\sigma^{-1}(X^x_u) Y^{(x)}_u)^* dW_u\right] \\
& =  \partial_x  P_{t-s} P_{s} f(x)  = \frac{1}{t-s} \E \left[P_s f(X^x_{t-s}) \int_0^{t-s}  (\sigma^{-1}(X^x_u)Y^{(x)}_u)^* dW_u\right]
\end{align*}
so that, {using that $\displaystyle \sup_{x\in \R^d}\sup_{s\in [0,T]}\E\big[| Y^{(x)}_s|^2\big] \le C <+\infty$  since $b$ and $\s$ have bounded first partial derivatives,}
\[
|\partial_x P_tf(x) |\le \frac{C_1}{\underline{\s}_0\sqrt{t}}\|f\|_{\sup} 
\]
and (with $s= \frac t2$)
\begin{align}
\label{eq:derivPtf}\partial^2_{x^2} P_tf(x)  &= \frac{2}{t}  \partial_x  \E \left[P_{\frac t2} f(X^x_{\frac t2}) \int_0^{\frac t2}  (\sigma^{-1}(X^x_u)Y^{(x)}_u)^* dW_u\right]\\
\nonumber & = \frac{2}{t}   \E \left[\partial_x P_{\frac t2} f(X^x_{\frac t2}) \int_0^{\frac t2}  (\sigma^{-1}(X^x_u)Y^{(x)}_u)^* dW_u\right] +  \frac{2}{t}  \E \left[P_{\frac t2} f(X^x_{\frac t2}) \int_0^{\frac t2} \partial_x  (\sigma^{-1}(X^x_u)Y^{(x)}_u)^* dW_u\right].
\end{align}
Let us denote $(A)$ and $(B)$ the two terms on the  right hand side of the above equation.
Using the above upper-bound for the first derivative, we obtain  (with real constants varying from line to line denoted by capital letter $C$ depending on $b$ and $\s$ and $T$)
\[
\big|(A)  \big|\le \frac 2t \frac{C}{\underline{\s}^2_0\sqrt{t}}\|f\|_{\sup}C'\sqrt{t} =    \frac{C'}{\underline{\s}^2_0 t}\|f\|_{\sup}.
\]     
As for the second term 
\[
\big|(B)\big| \le  \frac 2t\|f\|_{\sup}\left[ \int_0^{\frac t2}\E\, \big| \partial_x  (\sigma^{-1}(X^x_u)Y^{(x)}_u\big)|^2 du\right]^{\frac 12}.
\]
Using that $b$ an $\s$ have bounded existing partial derivatives, we derive by standard  computations   that \\
\noindent $\sup_{x\in \R^d}\E\big[\sup_{s\in [0,T]}|\partial _x Y^{(x)}_s|^2\big] \le C <+\infty$ so that (still using the $\underline{\s}^2_0$-ellipticity of $\s\s^*$)
\[
 \int_0^{\frac t2}\E\, \big| \partial_x  (\sigma^{-1}(X^x_u)Y^{(x)}_u\big)|^2 du \le \frac{C''t}{\underline \s_0^4}
\]
which finally implies that 
\[
|\partial^2_{x^2} P_tf(x) |\le \frac{C_2}{\underline\sigma^2_0t}\|f\|_{\sup}. 
\]
One shows likewise with similar arguments that
\[
|\partial^3_{x^3} P_tf(x) |\le \frac{C_3}{\underline\sigma^{3}_0t^{\frac 32}}\|f\|_{\sup} .
\]

\noindent $(b)$  If $f$ is Lipschitz continuous, note that 
\[
\partial_x P_tf(x) = \frac 1t \E \left[\Big(f(X^x_t) -f(x)\Big) \int_0^t (\sigma^{-1}(X^x_u) Y^{(x)}_u)^* dW_u\right] 
\]
so that 
\[
|\partial_x P_tf(x) |\le \frac{C}{\underline{\s}_0\sqrt{t}} [f]_{\rm Lip} \|ÊX^x_t-x\|_2\le \frac{C_1'}{\underline{\s}_0} [f]_{\rm Lip}S_2(x).
\]

For the second differentiation, we still rely on~\eqref{eq:derivPtf} and its decomposition into two terms $(A)$ and $(B)$. Using the above bound for the first derivative, we derive like above that
\[
\big| (A) \big| \le \frac{C'}{\underline \s_0^2\sqrt{t}}S_2(x).
\]
As for $(B)$ we first note that 
\[
(B)  =   \frac{2}{t}  \E \left[\Big(P_{\frac t2} f(X^x_{\frac t2}) -  f(x)\Big)\int_0^{\frac t2} \partial_x  (\sigma^{-1}(X^x_u)Y^{(x)}_u)^* dW_u\right].
\]
Now 
\begin{align*}
\big| P_{\frac t2} f(X^x_{\frac t2}) - f(x)\big| &  \le \big|\E \big[ f(X^x_t)-f(x)\,|\, X^x_{\frac t2}\big]\big|\le [f]_{\rm Lip} \E \big[ |X^x_t-x|\,|\, X^x_{\frac t2}\big]
\end{align*}
so that, using  Cauchy-Schwarz  inequality, the $L^2$-contraction property of conditional expectation and the above bound for the stochastic integral yields 
\begin{align*}
\big| (B) \big| &\le \frac 2t [f]_{\rm Lip}\big\| X^x_{\frac t2}-x \big\|_2\left[ \int_0^{\frac t2}\E\, \big| \partial_x  (\sigma^{-1}(X^x_u)Y^{(x)}_u\big)|^2 du\right]^{\frac 12}
\\
& \le  \frac 2t [f]_{\rm Lip}S_2(x)\sqrt{t}\, \sqrt{\frac{C''t}{\underline \s^4_0}} \le C [f]_{\rm Lip}S_2(x).
\end{align*}
More generally,
if $k=1,2,3$, there exist real constants $C'_k$ such that
\[
|\partial^k_{x^k} P_tf(x) |\le \frac{C'_k}{\underline{\s}^k_0t^{\frac{k-1}{2}}} [f]_{\rm Lip}S_2(x).
\]

\end{document}